\documentclass[11 pt]{amsart}
\usepackage{amsmath}
\usepackage{etoolbox}

\makeatletter
\def\@seccntformat#1{\csname the#1\endcsname.\space}
\makeatother
\usepackage[foot]{amsaddr}
\usepackage{graphicx} % Required for inserting images
\usepackage{enumitem}
\usepackage{caption}
\usepackage{subfigure}
\usepackage{subcaption}
\usepackage{float}

\usepackage{tikz}
\usetikzlibrary{decorations.pathreplacing}
\usetikzlibrary{patterns}
\usetikzlibrary{calc}
\usepackage{tcolorbox}%for comments
\usepackage{blindtext}
\usepackage{amssymb}
\usepackage{amsthm}
\usepackage{mathtools}
\usepackage{mathrsfs}
\usepackage{mdframed}
\usepackage{lipsum}
\usepackage{mathtools}
\usepackage{color}
\usepackage{setspace} 

\usepackage[margin=1in]{geometry}

\usepackage{thmtools}
\usepackage{thm-restate}
 \usepackage[hidelinks]{hyperref}
\hypersetup{linktocpage, colorlinks=true, linkcolor= red }
\usepackage{cleveref}

\declaretheorem[numberwithin=section]{theorem}

\patchcmd{\section}{\scshape}{\bfseries\large}{}{}
\patchcmd{\subsubsection}{\itshape}{}{}{}

\newtheorem{lemma}[theorem]{Lemma}

\newtheorem{question}[theorem]{Question}

\newtheorem{definition}[theorem]{Definition}
\newtheorem{example}[theorem]{Example}

\newtheorem{proposition}[theorem]{Proposition}
\newtheorem{corollary}[theorem]{Corollary}
\newtheorem{remark}[theorem]{Remark}

\newcommand{\xdownarrow}[1]{%
  {\left\downarrow\vbox to #1{}\right.\kern-\nulldelimiterspace}
}
\expandafter\def\expandafter\normalsize\expandafter{%
\normalsize\setlength\abovedisplayskip{3pt}}
\expandafter\def\expandafter\normalsize\expandafter{%
\normalsize\setlength\belowdisplayskip{3pt}}

\usepackage{framed}
\definecolor{shadecolor}{named}{lightgray}
\newcommand{\Mod}[1]{\ (\mathrm{mod} \ #1) }
\newcommand{\B}{\mathbf{B}}

\newcommand{\N}{\mathbb{N}}
\newcommand{\Z}{\mathbb{Z}}
\newcommand{\A}{\mathcal{A}}
\newcommand{\F}{\mathcal{F}}
\newcommand{\R}{\mathcal{R}}

\newcommand\CB{\mathcal{B}}

\newcommand\CR{\mathcal{R}}

\newcommand\BC{\mathbb{C}}

\newcommand\BN{\mathbb{N}}

\title{Finitely Dependent Processes on Subshifts}
\author[Nishant Chandgotia]{Nishant Chandgotia}
\address{(NC) Tata Institute of Fundamental Research - Centre for Applicable Mathematics}
\email{nishant.chandgotia@gmail.com}

\author[Aditya Thorat]{Aditya Thorat}
\address{(AT) International Centre for Theoretical Sciences (ICTS-TIFR), Tata Institute of Fundamental Research}
\email{adityathorat365@gmail.com}
\makeatletter
\let\@wraptoccontribs\wraptoccontribs
\makeatother
\contrib[with an appendix titled ``Tileability of sufficiently large boxes by a given set of boxes'']{by Jan Grebík, Filip Kučerák and Aditya Thorat}

\begin{document}

\begin{abstract}
Finitely dependent processes generalize independent processes by requiring that the restrictions of the process to sufficiently separated sets are independent. The existence of stationary finitely dependent processes on combinatorial models like $\mathbb Z^d$ subshifts can be quite mysterious. For instance, Holroyd and Liggett constructed such processes on proper $4$-colorings of $\Z^d$ for all $d$ while Holroyd, Schramm and Wilson showed that there are no such processes on proper $3$-colorings of $\Z^d$ for $d>1$. In this paper, we take inspiration from these results and investigate them further. On the positive side, we show that there exists a dense set of stationary finitely dependent processes supported on subshifts with strong mixing properties like the finite extension property. On the negative side, we see that the cohomology of the subshifts can form an obstruction to the existence of such processes. In particular we use Conway-Lagarias-Thurston height functions to characterize when there exists a finitely dependent process on the space of tilings by boxes of $\Z^2$ answering the tiling problem posed by Gao, Jackson, Krohne and Seward in dimension $2$. The ideas also apply to many other models, such as graph homomorphisms and ribbon tilings. On the way, we also show that continuous cocycles on strongly irreducible subshifts valued in a special class of groups (including torsion free Gromov hyperbolic groups and free product of cyclic groups) are perturbations of  group homomorphisms.
\end{abstract}
\maketitle
\tableofcontents

\section{Introduction}
A finitely dependent process on a graph $G=(V, E)$ is a stochastic process $(X_v)_{v\in V}$ such that there is a constant $k$ for which given any two sets $A, B\subset V$ separated by a distance $k$, we have that $(X_v)_{v\in A}$ is independent of $(X_v)_{v\in B}$. The support of a process $(X_v)_{v\in V}$ taking values in a finite set $\A$ is the set of configurations $x \in \A^V$ such that for all finite sets $B\subset V$, 
$$\mathbb P(X_v=x_v\text{ for all }v\in B)>0.$$
In this paper, we study the support of shift-invariant (stationary) finitely dependent processes on the $\Z^d$ lattice. The support $X\subset \A^{\Z^d}$ of any shift-invariant process with state space $\A$ satisfies two natural properties: Firstly it is compact. Secondly, it is invariant under translations by $\Z^d$. These two properties define a subshift $X\subset \A^{\Z^d}$: A closed translation invariant subset of $ \A^{\Z^d}$. Additionally, if the process is finitely dependent (for the constant $k$ as above), then the support $X$ satisfies the property: If $x, y \in X$ and $A, B\subset \Z^d$ are separated by distance $k$ then there exists $z\in X$ such that $z|_A=x|_A$ and $z|_B=y|_B$. Indeed, if the probabilities of $X_v=x_v$ for all $v\in A$ and $X_w=y_w$ for all $w\in B$ are each individually positive then so is the joint probability.  Any subshift which satisfies the above property is said to be strongly irreducible. Thus, if there is a finitely dependent process whose support is contained in a subshift $X$, then $X$ must contain a strongly irreducible subshift. This leads us to one of the main questions motivating our research: 

\begin{question}
Are there general conditions on the subshift which imply the existence or non-existence of shift-invariant finitely dependent processes supported on them?  
\end{question}

One of the first results in this direction comes from the papers by Holroyd and Liggett \cite{Holroyd} and Holroyd, Wilson and Schramm \cite{Schramm}. It was proved in \cite{Holroyd} that if there is a stationary finitely dependent process on the space of proper colorings of $\Z$, then it is not a (so-called) block factor of an iid process. Using this, they went on to construct shift-invariant finitely dependent processes on proper colorings of $\Z$ which are not block factors of an iid process. Although examples of shift-invariant finitely dependent processes which are not block factors of an iid were known \cite{Burton}, these formed a somewhat more natural set of such examples. An a priori surprising result is that while there is a shift-invariant finitely dependent process on the space of proper $4$-colorings of $\Z^d$ for all $d$ \cite{Holroyd}, there are none on the space of proper $3$-colorings \cite{Schramm} of $\Z^d$ for $d\geq 2$. To prove that there are no shift-invariant finitely dependent proper $3$-colorings of $\Z^d$ for $d\geq 2$, Holroyd, Schramm and Wilson analyzed the variance of the height function associated to proper $3$-colorings of $\Z^d$. A somewhat different argument was found in \cite{NoSIsubshifts} to show that proper $3$-colorings do not contain a strongly irreducible subshift, bypassing the probabilistic argument by Holroyd, Schramm and Wilson. Although the two arguments seem very different at first, both of them have a strong common thread: the existence of height functions for the model and their rigid behavior for strongly irreducible subshifts. Taking inspiration from these ideas, we find a large class of subshifts for which there are no shift-invariant finitely dependent processes on them.

The main novel idea comes from the analysis of cocycles (see \cite{Schmidt}) valued in some special groups. Specifically, a cocycle on a subshift $X$ taking values in a group $G$ is a continuous function $c: X\times \Z^d \to G$ such that 
$c(x, i+j)=c(x, i)c(\sigma^i(x),j)$ where $\sigma^i(x)$ is a configuration $y\in X$ such that $y_j=x_{i+j}$ for all $i\in \Z^d$. Many subshifts of interest have interesting and useful cocycles: Tilings of $\mathbb R^2$ by $1\times m$ and $m'\times 1$ boxes for integers $m, m'>1$ (\cite{zbMATH01256699} following constructions by Conway-Lagarias \cite{Conway} and Thurston \cite{zbMATH04175886}), graph homomorphisms from $\Z^d$ to a fixed finite undirected graph $H$ \cite{arXiv:2510.11969} and ribbon tilings \cite{ScottRibbon}. For many of these subshifts, there are natural cocycles which take values in groups (or their direct products) with a very special property (which we call the CIC property): The centraliser of any infinite order element in the group is cyclic. This condition is satisfied by several classes of groups arising naturally in geometric group theory. For example, torsion-free Gromov hyperbolic groups \cite[Chapter III.$\Gamma$, Corollary 3.10]{Bridson} (in particular free groups) and free products of cyclic groups \cite[Chapter 4, exercise 28]{Combinatorialgrouptheory} satisfy this property.

In the following discussion, for a finitely generated group $G$ with generating set $S$ and $g\in G$, we denote by $|g|$ the minimum length of the representation of $g$ in terms of $S$. We extensively use the following rigidity result, which is also the main technical innovation of this paper.
\begin{restatable}{theorem}{cocycles}\label{triviality of cocycles}
Let $X \subset \A^{\Z^d}$ be a strongly irreducible subshift where $ \A$ is finite and $ d \geq 2$. Let $G$ be a hyperbolic group that satisfies the CIC property. Then for any cocycle $c$ on $X$ with values in $G$, there exists a homomorphism $ \theta: \Z^d \rightarrow G$ and a constant $C$ such that for all $x \in X \text{ and } v \in \Z^d$, there exist some $h_1=h_{1}(x, v)$ and $ h_2= h_2(x, v)$ with $ \max(|h_1|, |h_2|) \leq C$ satisfying, 
$$c(x, v ) = h_1 \theta(v) h_2.$$ 
\end{restatable}

Using this, we prove in \Cref{NO SI SUBSHIFTS} that there is no shift-invariant finitely dependent process on the following spaces:
\begin{enumerate}
    \item The space of tilings by $m\times 1$ and $1\times m'$ boxes of $\mathbb R^2$ when $m,m'>1$ are integers. In particular, there is no shift-invariant finitely dependent process on domino tilings.
    \item The space of graph homomorphisms from $\Z^d$ to a fixed undirected simple graph without four cycles.
    \item The space of ribbon tilings.
\end{enumerate}
There is a remarkable relationship between questions about finitely dependent processes and questions in Borel/continuous combinatorics \cite{Bernshteyn,Computerscience} which is still not completely understood. However, it is true that if there is a continuous equivariant map from the free part of the full shift to a given subshift $X$, then there is also a shift-invariant finitely dependent process on $X$. By the above results it follows that there is no continuous equivariant map from the free part of the full shift to the space of tilings by $m\times 1$ and $1\times m'$ boxes of $\mathbb R^2$ when $m,m'>1$ are integers. This answers a question raised in \cite[Problem 2, Section 4.5]{gao2023continuouscombinatoricsabeliangroup} in the negative which asked if there is such a map into the space of tilings by $ m \times m$ and $(m+1)\times (m+1)$ boxes for $ m \geq 2$.

In fact, we answer the following more general question in the affirmative for $d=2$ (\cite[Question 3.1.8]{gao2023continuouscombinatoricsabeliangroup}): Is it algorithmically decidable if for a given finite collection $ \R = \{ R_1, \cdots, R_k\}$ of boxes in $ \mathbb{R}^d$, there exists a continuous equivariant map from the free part of the full shift $\{0,1\}^{\Z^d}$ to the space of tilings by boxes in $ \R$?

With the help of results by Greb\'ik, Ku\v{c}er\'ak and Thorat (see the appendix), we give a complete characterization in the case of tiling by boxes of $\mathbb R^2$.
\begin{restatable}{theorem}{characterisationboxes}\label{theorem: characterisation boxes}
Let $\R = \{R_1, R_2, \ldots, R_k\}$ be a collection of boxes in $\mathbb R^2$ with integer sidelengths. The following are equivalent:
\begin{enumerate}
    \item There exist $m,n>1$ such that the $m\times 1$ and $1\times n$ boxes tile each element of $\R$.
    \item There exists no shift-invariant finitely dependent process on the space of tilings by boxes in $\R$.
    \item There exists no continuous equivariant map from the free part of the full shift to the space of tilings by boxes in $\R$.
\end{enumerate}
\end{restatable}

We remark that while this is a condition which is very easy to check, we believe that it is algorithmically undecidable whether there exists a shift-invariant finitely dependent process supported on a homshift. In fact, it is algorithmically undecidable whether there is a continuous equivariant map from the free part of a full shift to a homshift \cite{gao2023continuouscombinatoricsabeliangroup}.

One of the most interesting questions which we still struggle with is coming up with novel methods of constructing shift-invariant finitely dependent processes (see Question \ref{question: structure of findep}). Beyond the block factors of iid, one of the main ways to construct shift-invariant finitely dependent processes on subshifts still relies on constructing one \cite{Schramm} on the so-called $m$-net processes on $\Z^d$ (which in turn uses the `mysterious' shift-invariant  finitely dependent proper $4$-colorings of $\Z$ constructed in \cite{Holroyd} - see the remark in \cite[beginning of Page 2]{zbMATH07206761}). Questions of this form point more towards the second direction of our research. As observed earlier, the support of every shift-invariant finitely dependent process is a strongly irreducible subshift. We do not know whether the converse is true.

\begin{restatable}{question}{stronglyirreducible}\label{question:stronglyirreducible} Given a strongly irreducible $\Z^d$ subshift $X$, does there exist a shift-invariant finitely dependent process on $X$? Is there one with support equal to $X$?
\end{restatable}
We have some partial answers though.
\begin{restatable}{theorem}{onedimension}\label{onedimension}
Let $X$ be a strongly irreducible subshift of $\A^{\mathbb{Z}}$. The set of shift-invariant finitely dependent processes supported on $X$ is $weak$-\textasteriskcentered{} dense in the set of shift-invariant measures on $X$. 
\end{restatable}
The results we have for higher dimensional subshifts seem to require a stronger mixing property called the finite extension property. The finite extension property for subshifts of finite type was introduced in \cite{FEP} for a completely different purpose. Several classes of subshifts like the space of proper $d+2$ colorings of $\Z^d$ \cite{Mixing} and the space of graph homomorphisms from $\Z^d$ to finite dismantleable graphs (\cite[Lemma 5.2]{Dismantlable}) satisfy the finite extension property. We prove that any subshift $X$ with the finite extension property admits a shift-invariant finitely dependent process with full support (i.e. support equal to $X$).
\begin{restatable}{theorem}{FEPcase}\label{FEPcase}
    Let $ X \subset \A^{\Z^d}$ be a subshift with the finite extension property. Then, the set of fully supported shift-invariant finitely dependent processes is $weak$-\textasteriskcentered{} dense in the set of shift-invariant processes on $X$.
\end{restatable}
While we state and prove most of our results for $\Z^d$ some of them work more generally for other groups. We point out whenever such an extension is possible.

Here is the outline of the paper: In \Cref{definitions}, we give the necessary background and state our main results. In \Cref{Section: m net processes} we discuss the existence of shift-invariant finitely dependent $m$-net processes on finitely generated groups. We discuss the existence and density of shift-invariant finitely dependent processes for strongly irreducible $\Z$ subshifts in \Cref{Section: one dimension}. In \Cref{Section: FINDEP On FEP}, we discuss the existence and density of such processes for $ \Z^d$ subshifts satisfying the finite extension property. In \Cref{Section: rectangular tilings}, we prove the existence and density of the shift-invariant finitely dependent processes on tiling spaces by boxes under some conditions on the sidelengths (stronger than the conditions in the appendix which suffice just for the existence). In \Cref{Section: cocycle rigidity}, we discuss the rigidity of cocycles for strongly irreducible subshifts which take values in hyperbolic groups with the CIC property. We use this rigidity in \Cref{NO SI SUBSHIFTS} to prove the lack of shift-invariant finitely dependent processes on some subshifts. We mention some open questions in \Cref{Section: open directions}. The appendix provides a necessary and sufficient condition for tileability of sufficiently large boxes in $ \mathbb R^d$.

\subsection*{Notation:} \hspace*{\fill} \\
We denote the identity in the group $G$ by $e$. 
Often $ G$ also denotes the right Cayley graph of $ G$ with respect to a symmetric set of generators $S$, that is, a graph whose vertex set is $ G$ and edge set is $ \{ (g, gs): g \in G, s \in S\}$. 
$ d(g, h)$ denotes the length of the shortest path between $g$ and $h$ in the above Cayley graph. When $g=e$, we write $|h|=d(e,h)$. In particular $d(g,h)=|g^{-1}h|$.\\ 
For subsets $ A, B \subset G$, $d(A, B) := \min \{ d(g, h) ~:~ g \in A, h \in B\} $. \\ 
$\mathbf{B}(g,r) := \{h \in G : d(h, g) < r\}$ denotes the open ball of radius $ r$ around $g$ and $\mathbf{B}(r)$ denotes the open ball of radius $r$ around $ e$.\\
We usually view $\Z^d$ as a Cayley graph with the standard set of generators. $||\cdot||_1$ denotes the $l_1$ norm on $ \mathbb{Z}^d$ and $||\cdot||_{\infty}$ denotes the $ l_{\infty}$ norm on $\Z^d$. For $v, w \in \Z^d$, let $d_{\infty}(v, w) := ||v- w||_{\infty}$.\\  
For stochastic processes $ (X_v)_{v \in I}$ and $ (Y_v)_{v \in I}$ indexed by a set $I$, the notation $((X_v)_{v \in I})  \stackrel{d}{=} (Y_v)_{v \in I}$ denotes their equality in distribution.\\
If $G$ and $H$ are groups, then $G$\textasteriskcentered{}$H$ denotes the free product of $G$ and $H$.\\
If $ G$ is a hyperbolic group, then for any $g, h \in G$, $ [g, h]$ denotes a geodesic between $g$ and $h$.\\
For any subset $ F \subset \Z^d$, the internal boundary of $F$ is $ \partial F = \{ v \in F: \text{there exists $w \in \Z^d \setminus F$ adjacent to $v$} \}$ and $ F^{(r)} := F+ \B(r) $ is the $r$-envelope of $F$. If $ K \subset \Z^d$, then $\text{int}_{K}(F) := \{ v \in \Z^d: v+K \subset F\}$. \\
If $ F \subset \Z^d$, $ |F|$ denotes the cardinality of $F$.\\
$\mathbb N=\{1,2,3, \ldots\}$. For $ k \in \N$, $[k]$ denotes the subset $ \{ 0, 1, \cdots, k-1 \}.$\\
Unless stated otherwise, $ e_1, e_2, \cdots, e_d$ denote the standard unit vectors in $ \Z^d$.\\
Often, we will denote a stochastic process $ (Y_v)_{v \in I}$ simply by $Y$ where the indexing is understood.

\subsection*{Acknowledgements}\hspace*{\fill} \\
This research began with a summer project which the second author had with Riddhipratim Basu. We thank him for his guidance and questions. We thank Tom Meyerovitch, Ville Salo and Yinon Spinka for many discussions. We thank Alexander Holroyd for mentioning the example of a shift-invariant finitely dependent Markov random field in two dimensions and for helping us with references. We thank Deepak Dhar for patiently listening to us about our results and suggesting further directions. The second author thanks Lukasz Grabowski, Jan Greb\'ik and Filip Ku\v{c}er\'ak for extensive discussions and many helpful suggestions. We thank Benjamin Hellouin de Menibus for many helpful comments on our first draft. The second author acknowledges the support of the Max Planck Institute for Mathematics in the Sciences in Leipzig for the short term visit and the Department of Atomic Energy, Government of India, under project identification no. RTI 4019 and RTI 4014. The first author thanks the Department of Atomic Energy, Government of India, under project identification no. RTI 4014.

We acknowledge the use of ChatGPT to identify the most general assumptions under which Lemma \ref{conjuagate powers} and \ref{bounded distance} hold.
\section{Background and main results}\label{definitions}

\subsection{Finitely dependent processes}\hspace*{\fill} \\
Here onwards, we will assume that $G$ is a countable group. We will abuse the notation and let $G$ also denote the right Cayley graph of $ G$  with respect to a fixed set of generators. 
\begin{definition}
A stochastic process $(X_{g})_{g \in G}$ is \textbf{$k$-dependent} if for any subsets $ A, B \subset V$ with $d(A, B) > k$, the restrictions $X_{A} = (X_g)_{g \in A}$ and $X_B = (X_g)_{g \in B}$ are independent. A process $(X_g)_{g \in G}$ is \textbf{finitely dependent} if it is $k$-dependent for some $k \in \mathbb{N} \cup \{0\}$.
\end{definition}

A stochastic process $ (X_g)_{g \in G}$ is \textbf{shift-invariant} if $ (X_g)_{g \in G}$ has the same distribution as $(X_{h^{-1} \cdot g})_{g \in G}$ for all $ h \in G$. If $ G= \Z^d$, shift-invariant processes are commonly called stationary processes. But for uniformity, we will refer to them as shift-invariant processes. In this paper, we will only be interested in shift-invariant processes.

Clearly iid processes are shift-invariant finitely dependent processes. One natural way to construct shift-invariant finitely dependent processes is by applying locally defined maps called block factors to iid processes, which we now introduce.

Let $(X_g)_{g \in G}$ and $ (Y_g)_{g \in G} $ be shift-invariant processes where $X_g \in \mathcal{A}$ and $Y_g \in \mathcal{B}$. Then,

\begin{definition}
 A process $(Y_g)_{g \in G} $ is an \textbf{r-block factor} of $ (X_g)_{g \in G} $, 
 if there exists a measurable function $ \psi: \mathcal{A}^{\mathbf{B}(r)} \rightarrow \mathcal{B}$ such that $ Y_g = \psi((X_h)_{h \in g \cdot \B(r)})$ for all $ g \in G$. $(Y_g)_{g \in G}$ is a \textbf{block factor} of $(X_g)_{g \in G}$ if it is an $r$-block factor of $ (X_g)_{g \in G}$ for some $ r \in \mathbb{N}$.
\end{definition}

It is easy to see that block factors of iid processes are shift-invariant finitely dependent processes.
\begin{lemma}
    If $ (X_g)_{g \in G} $ is an iid process and if $ (Y_g)_{g \in G} $ is a $r$-block factor of $ (X_g)_{g \in G} $ for some $ r \in \mathbb{N}$, then $ (Y_g)_{g \in G} $ is a shift-invariant $(2r+1)$-dependent process.
\end{lemma}
\begin{proof}
    If $ d(S, T) > 2r+1$, then $ \mathbf{B}(S, r)$ and $ \mathbf{B}(T, r)$ are disjoint subsets of $G$. But $ (Y_g)_{g \in S}$ depends only on $ (X_h)_{h \in \mathbf{B}(S, r)} $ while $ (Y_g)_{g \in T}$ depends only on $ (X_h)_{h \in \mathbf{B}(T, r)}$. \end{proof}

It was an open question for a long time whether all shift-invariant finitely dependent processes arise in this way. Aaronson, Gilat, Keane and de Valk (\cite{Aaronson}) constructed a shift-invariant $1$-dependent process on $ \mathbb{Z}$ which is not a $2$-block factor of any iid process. Aaronson, Gilat and Keane \cite{AaronsonMarkov} constructed a $1$-dependent five-state stationary Markov chain example which is not a $2$-block factor of an iid. Burton, Goulet and Meester \cite{Burton} constructed a $1$-dependent hidden Markov process on $ \Z$ which is not a $ r$-block factor of any iid. More references and a history of this problem can be found in \cite{Holroyd}.

Holroyd and Liggett constructed more examples of shift-invariant finitely dependent processes which are not block factors of any iid in \cite{Holroyd}. Their processes are supported on proper $q-$colorings of $ \Z$ for $ q =3,4$ (see Example \ref{proper colorings}). In \cite{Schramm}, it was proved that there does not exist a block factor of iid supported on proper $q$-colorings  of $\Z$ for any $q \geq 2$. They also proved that there does not exist a shift-invariant one-dependent proper $3$-coloring of $ \Z$. An easy argument \cite[Section 5]{Holroyd} showed that any shift-invariant finitely dependent proper $q$-coloring cannot be a Markov chain. However, Holroyd and Liggett proved the following.
\begin{theorem}$($\cite{Holroyd}$)$\label{1 dependent coloring}
\begin{enumerate}
    \item There exists a shift-invariant $1$-dependent process supported on proper $4$-colorings of $ \mathbb{Z}$.
    \item There exists a shift-invariant $2$-dependent process supported on proper $3$-colorings of $ \mathbb{Z}$.
\end{enumerate}
\end{theorem}

Thus, this theorem also provides examples of shift-invariant finitely dependent processes which are not block factors of any iid.

Given the result for $\Z$, it seems natural to ask for the existence of finitely dependent proper colorings on general graphs. When the underlying graph is $\Z^d$, this was done already in \cite{Holroyd}.
\begin{theorem}\cite[Corollary 5]{Holroyd}
For any $ d \geq 2$, there exists
\begin{enumerate}
    \item a shift-invariant $1$-dependent proper $4^d$-coloring of $\Z^d$.
    \item a shift-invariant $ k$-dependent proper $ 4$-coloring of $\Z^d$, where $ k$ depends on $d$.
\end{enumerate}
\end{theorem}

This result was generalised in \cite{timar2024finitelydependentrandomcolorings} and \cite{Onlinelocality} to show the existence of such colorings for transitive bounded degree graphs using suitably large number of colors. Let $H = (V, E)$ be a graph. Then a stochastic process $(X_v)_{v \in V}$ is automorphism invariant if its distribution is invariant under all automorphisms of $H$. %In particular, (\cite[Theorem 1]{timar2024finitelydependentrandomcolorings}) implies the following result.
 The following result follows from (\cite[Theorem 1]{timar2024finitelydependentrandomcolorings}).
\begin{theorem}\label{Timar theorem}
Let $ H$ be a transitive graph of maximal degree $\delta$. Then there exists a $ 4$-dependent automorphism invariant proper $4^{\frac{\delta(\delta+1)}{2}}$ coloring of $H$.     
\end{theorem}

\subsection{Subshifts}\hspace*{\fill} \\

The proper $q$-colorings are examples of combinatorial models which are called subshifts of finite type. We will be interested in understanding which combinatorial models allow finitely dependent processes on them and the ubiquity of finitely dependent processes in shift-invariant processes. 
We will refer to $\mathcal{A}^{G}$ as the \textbf{full shift} in the alphabet $\mathcal{A}$. We equip $ \mathcal{A}^{G}$ with the product topology where $ \A$ is given the discrete topology. If $\A$ is finite, then the product topology is metrisable and $\mathcal{A}^{G}$ is a compact metric space. The shift action of $G$ on $\mathcal{A}^{G}$ is defined by,
$ (h \cdot x)_{g} := x_{h^{-1} \cdot g}$ for all $ x \in \mathcal{A}^G$ and for all $ g, h \in G.$ This defines a left action of $ G$ on $ \A^{G}$.

The support of a shift-invariant process $ (X_g)_{g \in G}$ is the set
$$\Omega=\{w\in \mathcal{A}^G~:~\text{ for all finite $B\subset G$}, \mathbb P(X_g=w_g\text{ for all }g\in B)>0\}.$$
The set $ \Omega$ has two essential features.
\begin{enumerate}
    \item  It is shift-invariant, i.e., $ g \cdot \Omega = \Omega$ for each $g \in G$. 
    \item  It is closed in $ \mathcal{A}^{G}$.
\end{enumerate}
\begin{definition}
A \textbf{subshift} of $ \mathcal{A}^{G}$ is a closed shift-invariant subset of $ \mathcal{A}^{G}.$     
\end{definition}
\begin{example}[proper $q$-colorings] \label{proper colorings}
Let $ q \geq 1$ and $ \mathcal{A}= \{ 0,1,2,\cdots, q-1\}$. Then the subshift of proper $q$-colorings of $ G$ is the set of all colorings of $ G$ with the colors from $\mathcal{A}$ such that colors on the adjacent vertices of $ G$ are distinct.
\[ \mathcal{X}_{q}^{G} := \{ x \in \mathcal{A}^{G} : x_{g} \neq x_{h}, \text{ if } g \text{ is adjacent to } h \}.\]
We will denote the set of proper $q$-colorings of $ \Z^d$ by $\mathcal{X}_{q}^{d}.$
\end{example}
\begin{example}[Homshifts]\label{Homomorphism shifts}
    Here is a way to generalise the previous example. Consider $ G$ as an (undirected) right Cayley graph with a fixed symmetric set of generators. Let $ H = (V, E)$ be a countable undirected graph. Let, 
    \[ \textit{Hom}(G, H) := \{ x \in V^{G}: x \text{ is a graph homomorphism from } G \rightarrow H \}. \]
    We will refer to $\textit{Hom}(G, H) $ as \textbf{homshifts} corresponding to the graph $H$.
    
 Note that if $ K_{q} $ is the complete graph with $q$ vertices, then $ \textit{Hom}(G, K_q)  = \mathcal{X}_{q}^{G}$.
 \end{example}

A \textbf{pattern} supported on a finite set $F \subset G$ is an element of $ \A^{F}$. We will denote by $ x|_{F} $ the restriction of $ x \in \A^{G}$ to the subset $F$. A translate of a pattern $ a \in \A^{F}$ by an element $ h \in G$ is the pattern $ h \cdot a \in \A^{h \cdot F}$ defined by $ (h \cdot a)_{g} := a_{h^{-1} \cdot g} $ for each $g \in F$. If $a \in \A^{F_1}$ and $ b \in \A^{F_{2}}$ where $ F_1 \cap F_2 = \emptyset$, then $ a \sqcup b \in \A^{F_1 \sqcup F_2}$ is a pattern whose restriction on $ F_1$ equals $ a$ and restriction on $ F_2$ equals $b$. We reserve the term \textbf{configuration} to elements of $ A^{F}$, where $ F$ is an infinite subset of $G$.

Let $X$ be a subshift of $ \A^{G}$. For $ a \in \A^{F}$, the cylinder set corresponding to $ a$ is defined as, 
\[ [a] := \{ x \in X: x|_{F} = a\} .\]

A pattern (or configuration) $ a \in \A^F$ is \textbf{globally allowed} in $X$, if it appears in some element of $X$, i.e., $ a = (g \cdot x)|_{F}$ for some $x \in X $ and $ g \in G.$ The set of all \textit{finite} globally allowed patterns is called the language of $X$, denoted by $L(X)$. If $ F$ is a finite subset of $ G$, then $L(X, F) = \{x|_{F}:x\in X \}$ denotes the set of all globally allowed patterns supported on $F$. 

There is an alternative and more direct way to define a subshift. Let $ \mathcal{F} = \{ a_1, a_2, \cdots\}$ be a countable collection of (forbidden) patterns, where $ a_i \in \A^{S_i}$ is a pattern supported on a finite set $ S_i \subset G$. Then the subshift corresponding to the collection $\F$ is the set of all configurations $ x \in \A^{G}$ such that  $x$ does not contain a translate of any pattern from $\mathcal{F}$ anywhere, that is,
\begin{equation}\label{Definition: subshift} X_{\mathcal{F}} : = \{ x \in \A^{G}: (g\cdot x)|_{ S_i}\neq a_i \text{ for any }  i \in \N \text{ and } g \in G\}.\end{equation}
Then $X_{\F}$ is a subshift. In fact, every subshift $X$ is of the form $ X_{\mathcal{F}}$ for some countable collection $\mathcal{F}$. See \cite[Chapter 6]{LindMarcus} for the proof when $G=\Z$. The statement for general $G$ follows similarly. $\mathcal{F}$ is called a set of \textbf{forbidden patterns} for $X$.
\begin{definition}
     A subshift $X$ is a \textbf{subshift of finite type} if there exists a finite collection $\mathcal{F}$ of forbidden patterns such that $ X = X_{\mathcal{F}}$. 
\end{definition}
The subshift in Example \ref{proper colorings} is a subshift of finite type. We can equivalently define a subshift of finite type by specifying a finite set of (locally) allowed patterns. Let $ T $ be a finite subset of $G$. Let $ \mathcal{P} \subset A^{T}$ be a collection of patterns supported on $ T$. Then define
\begin{equation} \label{allowed words sft}
X^{\mathcal{P}} := \{ x \in A^{G}: (g \cdot x)|_{T} \in P, \text{ for every } g \in G\}.                                                        
\end{equation}
Then $ X^{\mathcal{P}}$ is a subshift of finite type where the set of forbidden patterns is $ A^{T} \setminus \mathcal{P}$. Conversely, every subshift of finite type is of the form $ X^{\mathcal{P}} $ for some finite set $\mathcal{P}$ of patterns supported on a fixed finite subset of $G$.

A pattern (or configuration) $ w $ supported on a set $ S \subset G $ is called \textbf{locally allowed} if it does not contain a translate of any pattern from $ \mathcal{F}$. One of the challenges in working with multidimensional subshifts of finite type is that unlike one dimensional subshifts of finite type, it is not decidable whether a locally allowed pattern is globally allowed (\cite{Undecidability}).

\begin{example}[Rectangular tilings]\label{Rectangular tilings}
Let $ \mathcal{R} = \{ R_1, R_2, \cdots, R_k \} $ be a set of boxes in $ \mathbb{R}^2$ whose lower left corner is at the origin. Throughout the paper, we will always assume that the boxes have integer sidelengths. These boxes will be our tiles. A tiling of $ \mathbb{R}^2$ with the tile-set $\mathcal{R}$, is a decomposition of $\mathbb{R}^2$ into a union of translates of boxes in $ \mathcal{R}$ such that any two boxes intersect only at the boundary and corners of each box in the tiling lie on the $ \mathbb{Z}^2$ lattice. Whenever we refer to tilings of $\mathbb{R}^2$ with tiles from $ \mathcal{R}$, the corners of tiles will always be assumed to lie on the integer lattice. Let $ X(\mathcal{R})$ be the set of all tilings of $ \mathbb{R}^2$ with tile-set $\mathcal{R}$. We can define a finite set of forbidden patterns $\mathcal{F}$ such that each element of the corresponding subshift of finite type corresponds to a tiling of $ \mathbb{R}^2$ by $ \mathcal{R}$ and vice versa $($\cite[Section 4.2]{Schmidt}, \cite[Section 2]{Einsiedler}$)$. Hence $X(\mathcal{R})$ is a subshift of finite type of $ \A^{\mathbb{Z}^2}$ for some $\A$. We illustrate this for the case when $ \mathcal{R} = \{ R_1, R_2\}$ where $ R_1$ is a $ 1 \times 2$ box in $ \mathbb{R}^2$ and $ R_2$ is a $ 2 \times 1$ box in $ \mathbb{R}^2$. We refer to the boxes $R_1$ and $R_2$ as dominoes.

Note that tilings of $ \mathbb{R}^2$ by dominoes are in one-to-one correspondence with perfect matchings of the dual graph $\mathbb{Z}_{\text{dual}}^2$ of $ \Z^2$, where the correspondence is obtained by creating an edge in the dual graph if and only if the edge is enclosed within a domino in the tiling. Thus, we can equivalently consider $ X(\mathcal{R})$ to be the set of perfect matchings of $ \mathbb{Z}_{\text{dual}}^2$. Here (and only here), for convenience, we translate the tilings by the vector $ (0.5, 0.5)$ and view $ X(\mathcal{R})$ instead as the set of perfect matchings of $ \Z^2$. We refer to edges in perfect matchings as dominos.
\begin{figure}[!htbp]
    \centering
    \includegraphics[width=0.3\linewidth]{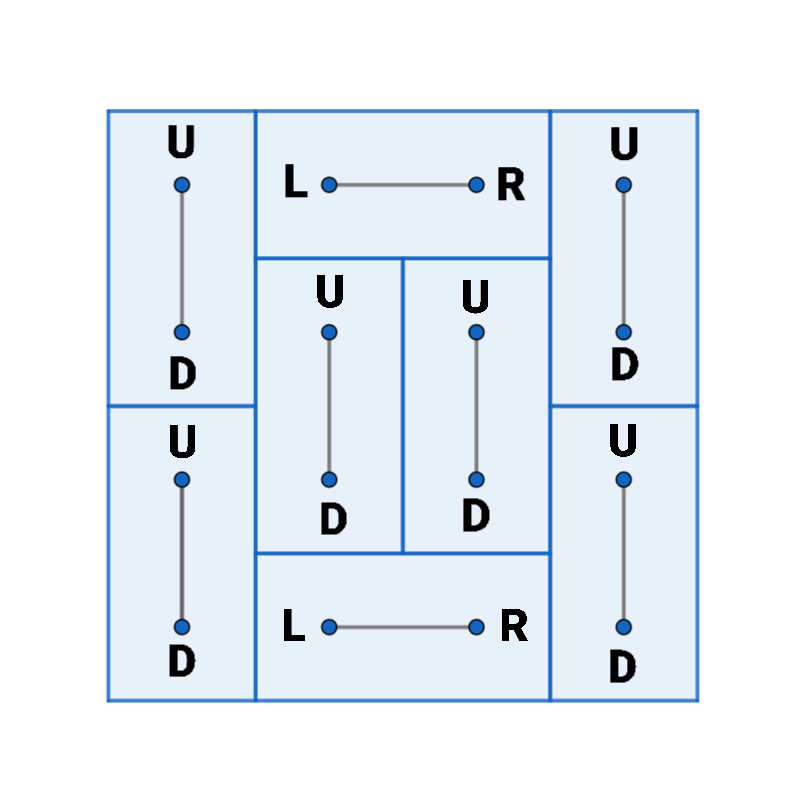}
    \caption{A pattern of domino tilings and corresponding perfect matching}
    \label{Domino SFT}
\end{figure}
Let $ \A= \{L,R,U,D\}$, $ T = \{ (0,0), (0,1), (1,0)\} \subset \mathbb{Z}^2$ and 
\[ P:= \{ \big(\begin{smallmatrix}
 \text{not }U & \text{ }\\
 L & R
\end{smallmatrix} \big), \big(\begin{smallmatrix}
 U & \text{ }\\
 D & \text{ not }R
\end{smallmatrix} \big), \big(\begin{smallmatrix}
 \text{not }U & \text{ }\\
 R & \text{ not }R
\end{smallmatrix} \big), \big(\begin{smallmatrix}
 \text{not }U & \text{ }\\
 U & \text{ not }R
\end{smallmatrix} \big) \} .\]
Here ``not $U$" means any element in $\A$ except $U$. Then by interpreting an occurrence of $L$ as a left vertex of a horizontal domino and an occurrence of $D$ as bottom vertex of a vertical domino etc., we see that $ X(\mathcal{R}) = X^{\mathcal{P}}$ (see Figure \ref{Domino SFT} and Definition \ref{allowed words sft}). Therefore, $ X(\mathcal{R})$ is a subshift of finite type of $ \{ L, R, U, D \}^{\mathbb{Z}^2}$.

\end{example}
\begin{example}[Ribbon tilings]\label{Ribbon tilings}
    For a fixed $ n$, color the unit square in $ \mathbb{R}^2$ with the lower left corner at $ (i, j) \in \Z^2$ by color $ (i+j) \Mod{n}\in \Z/n\Z$. A ribbon tile of order $n$ is a connected set of unit squares whose corners lie on the $\Z^2$ lattice that contains exactly one square of each color in $\Z/n\Z$ (see Figure \ref{ribbon figure}).
    \begin{figure}[!htbp]
        \centering
        \includegraphics[width=0.75\linewidth]{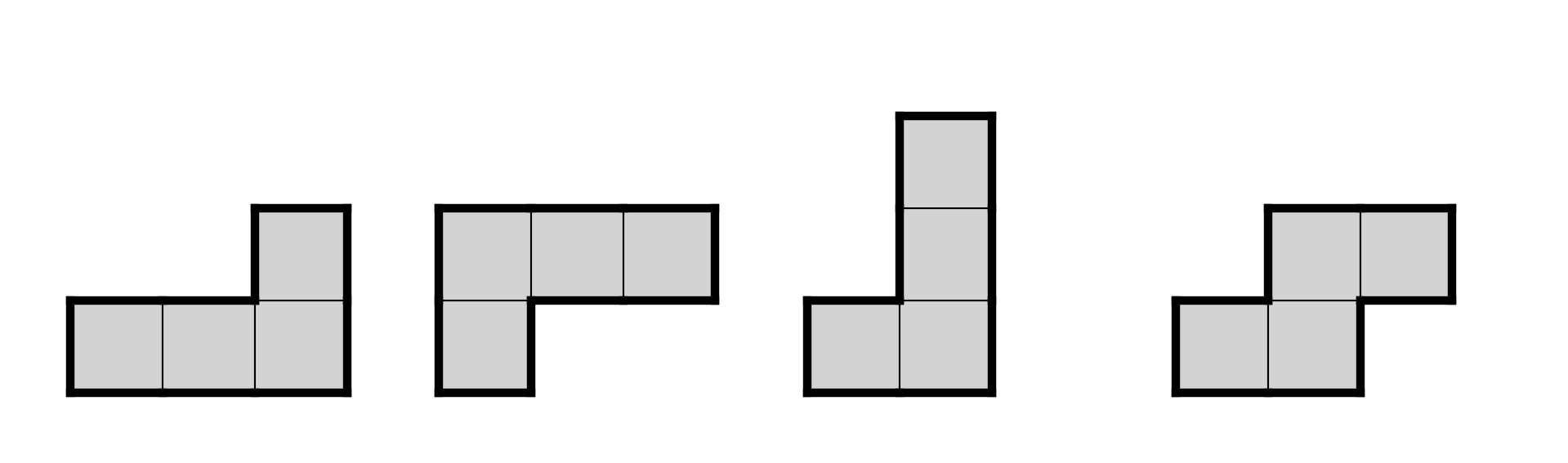}
        \caption{Some ribbon tiles of order 4 }
        \label{ribbon figure}
    \end{figure}
    As in Example \ref{Rectangular tilings}, it is easy to see that the set of all tilings of $ \mathbb{R}^2$ using translates of ribbon tiles of order $n$ (such that the corners of the tiles lie on the $\Z^2$ lattice) is a subshift of finite type. Note that the colors are helpful in describing what the tiles are but a tile is just a subset of $\mathbb{R}^2$ (without any reference to colors). We will denote this subshift by $ X(ribbon, n).$ 
\end{example}
A map $\Psi$ from $X \subset \A^{G}$ to $ Y \subset \mathcal{B}^G$ is \textbf{equivariant} if $ \Psi(g \cdot x) = g \cdot \Psi(x)$ for all $ g \in G$ and $x \in X$.
\begin{definition}
    A map $\Psi$ from a subshift $X \subset \A^{G}$ to a subshift $ Y \subset \mathcal{B}^G$ is called a \textbf{sliding block code}, if there exists a finite set $ F$ and a map $ \psi: L(X, F) \rightarrow \mathcal{B}$ such that $ \Psi(x)_g= \psi(x|_{g.F})$ for all $ x \in X$ and $ g \in G$.  The sliding block code is called a \textbf{factor} if it is surjective.
\end{definition}

 We will refer to $\psi$ as the \textbf{block map} of the sliding block code $\Psi$. The \textbf{coding radius} of the map $ \Phi$ is the smallest $r$ such that ball of radius $r$ around identity contains the set $ F$. We note that any sliding block code is continuous and equivariant.
A subshift $Y$ is called \textbf{sofic} if there exists a factor map from a subshift of finite type \textit{to} $Y$.
Here is an example of a subshift which is sofic but not of finite type \cite[Chapter 3]{LindMarcus}.
\begin{example}[Even shift]\label{Even shift}
 Let $ X \subset \{ 0,1\}^{\mathbb{Z}}$ be the set of all bi-infinite strings in letters $\{ 0,1\}$ such that the number of zeroes between any two successive occurrences of $1'$s is even. Then $ X$ is called an \textbf{even shift}. Here, the set of forbidden patterns which describes the subshift is $ \mathcal{F} = \{ 10^{m} 1: m \text{ is odd} \} $. Then $ X = X_\mathcal{F}$. 
 It is easy to see that there is no finite collection of forbidden patterns $ \mathcal{F}'$ such that $X = X_{\mathcal{F}'}$: For instance, assume that $\mathcal{F}' \subset A^n$ is such a collection. Then since all subpatterns of length $n$ in the pattern $10^{2n+1}1$ are globally allowed in $X$, $ 10^{2n+1}1 \notin \mathcal{F}'$, a contradiction. Hence, $X$ is not a subshift of finite type.
 \end{example}
 
\begin{definition}
    A subshift $ X \subset A^{G}$ is \textbf{topologically mixing} if for every pair of globally allowed patterns $ a \in A^F$ and  $b \in A^{F'}$ in $X$ there exists $ n \in \mathbb{N}$ such that for all $ g \in G$ with $ d(g,e) > n$, 
 we have $ g \cdot [a] \cap [b] \neq \phi $.
 
\end{definition}
Equivalently $ X$ is topologically mixing if for every pair of globally allowed patterns $ a$ and $b$, there exists $ n= n(a, b) \in \mathbb{N}$ such that for every $ g \in G$ with $ d(g,e) > n$, there exists an $ x = x(a, b, g)\in X$ such that $ x|_{g \cdot F} = a$ and $ x|_{F'}= b$.

\begin{definition}
    A subshift $X $ is \textbf{strongly irreducible} (SI) if there exists $k \in \mathbb{N}$ such that for any two subsets $F$ and $F'$ in $G$ with $d(F, F') > k$ and any pair of globally allowed patterns $ a \in A^F$ and $ b \in A^{F'}$, we have $[a] \cap [b] \neq \emptyset$, that is, there exists $ x \in X$ such that $ x|_{F} = a$ and $ x|_{F'} =b$. The smallest such $k$ is called the SI distance for $X$.
\end{definition}

For example, the even shift is strongly irreducible with SI distance $k=1$. Clearly, strong irreducibility implies topological mixing. For one-dimensional subshifts of finite type, topological mixing and strong irreducibility are equivalent. This is no longer the case for multidimensional subshifts of finite type. For example, let $ \mathcal{R} = \{ R_1, R_2\}$ where $ R_1$ is a $ m_1 \times n_1 $ box and $ R_2$ is a $m_2 \times n_2$ box in $ \mathbb{R}^2$. It follows from (\cite[Lemma 2.1]{Einsiedler}) that if $ gcd(m_1, m_2) = 1$ and $gcd(n_1, n_2) =1$, then the space of the tilings $ X(\mathcal{R})$ (see Example \ref{Rectangular tilings}) is topologically mixing. However, we can show that if $ n_1>1$ and  $m_2 > 1$, then  $ X(\mathcal{R})$ is not strongly irreducible. We illustrate this for the case of domino tilings ($ m_1 =n_2=1, m_2 = n_1=2$). The general case is similar.
\begin{figure}[!htbp]
    \centering
    \includegraphics[width=0.7\linewidth]{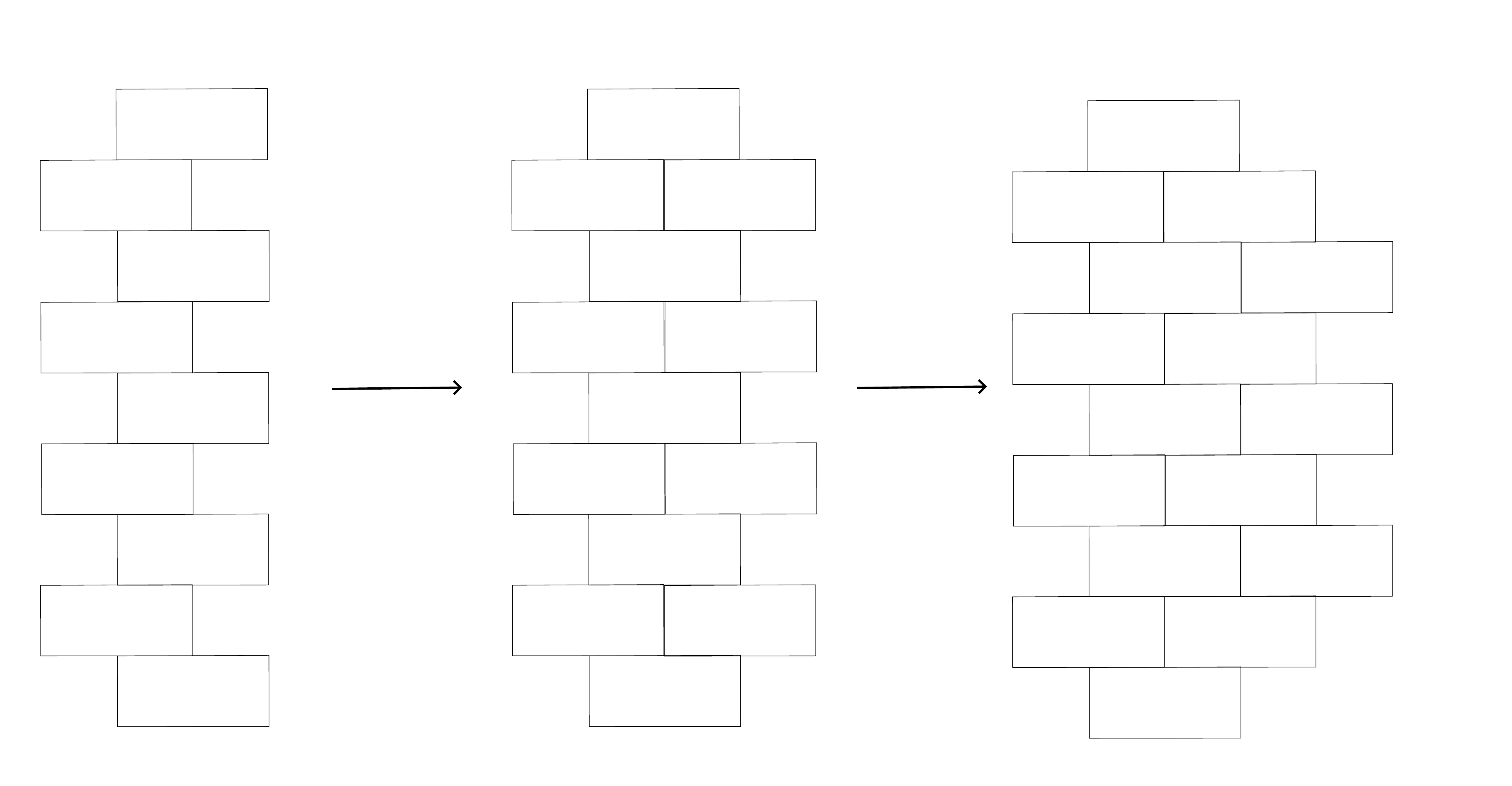}
    \caption{Domino tilings are not strongly irreducible}
    \label{Domino}
\end{figure}

Consider the pattern on the left-hand side of the Figure \ref{Domino}. If this pattern appears in some $ x \in X(\mathcal{R})$, then it forces the patterns to the horizontal distance $\lfloor \frac{n}{2} \rfloor$, where $n$ is the height of the pattern (we have shown the first two steps of pattern forcing). Since $ n $ can be made arbitrarily large, it follows that there is no $k$ such that we can put any two globally allowed patterns together in a tiling of $\mathbb{Z}^2$ whenever the distance between the two patterns is greater than $k$. Hence $X(\mathcal{R})$ is not strongly irreducible.

We also have the following theorem.
\begin{theorem}
\begin{enumerate}
    \item $\mathcal{X}_{q}^{d}$ is topologically mixing iff $q\geq 3$ (see \cite[Section 4.3]{Schmidt}).
    \item For $ 3 \leq q \leq d+1$,  $\mathcal{X}_{q}^{d} $ is not strongly irreducible (\cite{Mixing}).
\end{enumerate}
\end{theorem}
The following proposition links finitely dependent processes to strongly irreducible subshifts.
\begin{proposition}\label{FINDEP implies SI}
If $ (X_g)_{g \in G} \in \mathcal{A}^{G}$ is a shift-invariant finitely dependent process, then its support is a strongly irreducible subshift.    
\end{proposition}
\begin{proof}
    Let $\text{supp}(X)$ denote the support of $ (X_g)_{g \in G}$. Then clearly $\text{supp}(X)$ is a subshift of $ \mathcal{A}^{G}$. Assume that $ (X_g)_{g \in G}$ is $ k$-dependent. Let $ a \in \mathcal{A}^{S} $ and $ b \in \A^{T}$ be two globally allowed patterns in $\text{supp}(X)$ with $ d(S, T)> k$. Then $ \mathbb{P}(X_{S} = a) > 0$ and $ \mathbb{P}(X_{T} =b) > 0$. Hence, 
    \[ \mathbb{P}(X_{S} = a , X_{T} = b) = \mathbb{P}(X_S = a)\mathbb{P}(X_T=b) >0 .\]
    Thus, there exists $ x \in \text{supp}(X)$ such that $ x_{S} = a $ and $ x_{T} =b$. Hence, $\text{supp}(X)$ is strongly irreducible. 
    \end{proof}
    
Conversely, we prove that every strongly irreducible subshift on $\Z$ supports a dense set of shift-invariant measures. 
\onedimension*
Although we do not know if every strongly irreducible $\Z^d$-subshift supports a shift-invariant finitely dependent process for $ d \geq 2$, we prove the existence of shift-invariant finitely dependent processes on subshifts satisfying a stronger mixing property than strong irreducibility. The finite extension property is a strengthening of strong irreducibilty for subshifts of finite type which was introduced in (\cite{FEP}). 
\begin{definition}
    A subshift of finite type $ X \subset A^{G}$ satisfies the $k$-extension property if there exists a finite collection $ \mathcal{F} $ of forbidden patterns for which $ X = X(\mathcal{F})$ and such that any locally allowed pattern $ w$ supported on a set $S \subset G$ which extends to a locally allowed pattern on $ S \cdot \mathbf{B}(k) $, is globally allowed.  $ X$ has finite extension property (FEP) if it satisfies $ k$-extension property for some $k \in \N$.
\end{definition}
For example, it was proved in \cite{Mixing} that the subshift $\mathcal{X}_{q}^{d}$ satisfies the finite extension property if and only if $q \geq d+2$.  It is easily seen that any FEP subshift is strongly irreducible. However, there are strongly irreducible subshifts of finite type which do not satisfy FEP (see \cite[Theorem 8.16]{poirier2024contractiblesubshifts}, for example).  

We prove that the finite extension property is sufficient for the existence of shift-invariant finitely dependent processes on $\Z^d$. Additionally, we can make sure that such a process gives positive probabilities to all finite allowed patterns in $X$ (i.e.\! fully supported).
\FEPcase*
 Since the subshift of proper $k$-colorings of $ \Z^d$ is not strongly irreducible for $ k \leq d+1$ and has the finite extension property for $ k \geq d+2$, we have the following corollary.
\begin{corollary}
    There exists a fully supported shift-invariant finitely dependent process on proper $k$-colorings of $\Z^d$ iff $ k \geq d+2$.
\end{corollary}
   The question of the existence of shift-invariant finitely dependent processes on subshifts is intimately linked to questions in descriptive combinatorics (\cite{gao2023continuouscombinatoricsabeliangroup,Computerscience}). We briefly mention the connection here: Consider the free action of $ \Z^d$ (in general, of any countable group) on a standard Borel space $ \mathcal{S}$. For a set of generators $ F\subset \Z^d$, the Schreier graph for this action is the graph with the vertex set $ \mathcal{S}$, where $ x$ is connected to $y$ iff $ y = \sigma^{v} (x)$ for some $ v \in F$. One is often interested in the combinatorial properties (e.g. chromatic number) of this graph. A map $f: \mathcal{S} \rightarrow \{ 1,2, \cdots, k \}$ is a proper $k$-coloring of $\mathcal{S}$, if $ f (x) \neq f(y)$ whenever $ x$ is adjacent to $y$. The measurable/Borel/continuous chromatic number of $\mathcal{S}$ is the minimum $ k$ for which there exists a measurable/Borel/continuous proper $ k$-coloring of (Schreier graph) $\mathcal{S}$. 

   An important standard Borel space which admits a free action of $\Z^d$ is the free part of $ \{0,1\}^{\Z^d}$, which is the set of all elements in $ \{0,1\}^{\Z^d}$ which do not have any non-trivial period. 
\[ F(\{0,1\}^{\Z^d}) := \{ x \in \{0,1\}^{\Z^d}: \sigma^{v}(x) \neq x \text{ for all } v \in \Z^d\setminus\{\mathbf{0}\}\}.\]
We will denote $\{0,1\}^{\Z^d}$ by $2^{\Z^d}$. $F(2^{\Z^d})$ naturally inherits the subspace topology of $ 2^{\Z^d}$ under which it is a Polish space. If $d >1$ and $ \mathcal{S} = F(\{0,1\}^{\Z^d})$, then it was proved in \cite{gao2023continuouscombinatoricsabeliangroup} that there is no continuous proper $3$-coloring of $ \mathcal{S}$, however, as proved in \cite{GaoJacksonbreakthrough} there exists a continuous proper $ 4$-coloring of $\mathcal{S}$.  In other words, the continuous chromatic number of $\mathcal S$ is $4$.
In general, given a subshift of finite type $Y \subset \A^{\Z^d}$, we say that there exists a measurable/Borel/continuous coloring of $ \mathcal{S}$ which satisfies the constraints of $ Y$, if there exists a measurable/Borel/continuous map $ f: \mathcal{S} \rightarrow \A$ such that the element $ f(g \cdot x)_{g \in \Z^d} \in \A^{\Z^d}$ lies in $Y$ for all $ x \in \mathcal{S}$. The following theorem follows from \cite{Holroyd}, \cite{Bernshteyn} (see also \cite[Theorem 7.2 and Theorem 7.4]{Computerscience})
\begin{theorem}\label{continuous imples findep}
    Let $ X \subset \A^{\Z^d}$ be a subshift of finite type. If there exists a continuous equivariant map $\Psi: F(2^{\Z^d}) \rightarrow X$, then there exists a shift-invariant finitely dependent process supported on $X$.
\end{theorem}
For $d=1$, the converse of the above theorem also holds (\cite[Corollary 6]{Holroyd}). It is not known whether the converse of the above theorem is true for $ d >1$.
\subsection{Cocycles}\label{section:cocycle}\hspace*{\fill} \\
In this subsection, we introduce cocycles on subshifts and discuss some examples. If we wish to prove that there are no shift-invariant finitely dependent processes supported on a subshift $X$, then it is enough to show that there are no strongly irreducible subshifts in $X$. In Section \ref{NO SI SUBSHIFTS}, we will use the cocycles introduced in this subsection to prove that there are no strongly irreducible subshifts within some subshifts.

Here, We restrict to $\mathbb{Z}^d$ subshifts. We will denote the action of the element $ v $ of $\mathbb{Z}^d$ on $ A^{\mathbb{Z}^d}$ by $ x \rightarrow \sigma^{v}(x)$.
\begin{definition}
    Let $ X \subset A^{\mathbb{Z}^d}$ be a subshift and $H$ be a group. A $ H$-valued \textbf{cocycle} on $ X$ is a continuous function $c: X \times \mathbb{Z}^d \rightarrow H$ satisfying, \begin{equation}\label{cocycle} 
     c(x, v+ w) = c(x, v)\cdot c(\sigma^{v}(x), w).  
    \end{equation}
    \end{definition}
    \begin{remark}\label{remark: cocycle with respect to a subaction}
        In this paper $($see \Cref{ribbon tiling cocycle}$)$, we will also need to consider cocycles with respect to subactions instead of the full $\Z^d$ action. For instance, if $W$ is a subgroup of $ \Z^d$, then a function $c$ like in the definition above is called a cocycle with respect to the $W$-subaction if \Cref{cocycle} is satisfied for $v, w\in W$.
    \end{remark}
In the literature, cocycles (with respect to a finite index subgroup) are often referred to as height functions. Note that if $ c$ is a cocycle and $ c(\cdot, v)$ is constant for every $v \in \mathbb{Z}^d$, then $c$ is a homomorphism from $ \mathbb{Z}^d \rightarrow H$.

 \begin{example}

 There is a cocycle $c: \mathcal{X}_{3}^{d} \rightarrow \mathbb{Z}$ on proper $3$-colorings of $\mathbb{Z}^d$ which is defined as follows. Let $ x \in \mathcal{X}_{3}^{d}$ and let $e_i, i=1,2,\cdots, d$ be the standard unit vectors in $\Z^d$. Then define
\begin{equation}
c(x, e_i) = \begin{cases}
 1  & \text{ if } (x_{\mathbf{0}}, x_{e_i}) \in \{ (0,1), (1,2), (2,0) \}. \\
 -1 & \text{ if } (x_{\mathbf{0}}, x_{e_i}) \in \{ (0,2), (1,0), (2,1) \}. \\
\end{cases}    
\end{equation}
 
We then extend $c(x, \cdot) $ to $ \mathbb{Z}^d$ using the cocycle equation \eqref{cocycle}. In other words, to determine $ c(x, v)$ we consider a path from the origin to $ v$, we initialise $c(x, \mathbf{0})$ to $0$ and add $ 1$ to the cocycle value if we encounter configurations $ 0 - 1 \text{ $($a 1 next to a 0$)$ }$, $1 - 2$ , $2 - 0$ along an edge and subtract $1$ otherwise. Since the total change incurred after traveling around any pattern on a unit square in $ \Z^2$ is zero, the result is independent of the chosen path.

This cocycle was used in $($\cite[Proposition 12.2]{NoSIsubshifts}$)$ to show that there are no strongly irreducible subshifts in proper $3$-colorings of $ \mathbb{Z}^d$ for $d \geq 2$. This gives an alternate proof of the fact that there are no shift-invariant finitely dependent proper $3$-colorings of $ \mathbb{Z}^d$ for $d \geq 2$, which was proved in \cite{Schramm} using probabilistic ideas $($such as the analysis of variance of $ c(x, n e_1 )$ as $ n \rightarrow \infty)$. 
\end{example}

\begin{example}\label{tiling cocycle}
    Let $\R = \{ R_1, R_2\}$ where $R_1$ and $ R_2$ are two boxes with dimensions $ m \times 1$ and $ 1 \times n$ respectively. Let $X(\mathcal{R})$ be the set of all tilings of $ \mathbb{R}^2$ using $R_1$ and $R_2$. Conway and Lagarias $($\cite{Conway}$)$ and Thurston $($\cite{zbMATH04175886}$)$ introduced cocycles on tiling spaces to answer questions related to tileability of regions of the plane by a given set of tiles. The cocycle on $X(\mathcal{R})$ was defined by C. Kenyon and R. Kenyon \cite{zbMATH01256699} as follows:

Consider the free group $\mathbb{F}_2$ with generators $\{ h,v\}$. Our cocycle will take values in a quotient of this group. Here, $v$ denotes the change in cocycle value after traveling vertical edge of length 1 and $h$ denotes the change in cocycle value after traveling horizontal edge of length 1. To ensure that the cocycle values are defined without any ambiguity, we require that the total change in the value of the cocycle incurred after traveling along tiles $R_1$ or $R_2$ is identity. Hence, we must have 
\begin{equation}\label{tilingcocycle}
v  \cdot  h^{m}  \cdot  v^{-1}  \cdot  h^{-m}= e. \text{ and } v^{n}  \cdot  h  \cdot  v^{-n}  \cdot  h^{-1}= e.   
 \end{equation}  
 
To this end, consider the group defined by the relations \[ \Gamma = \langle h,v | v  \cdot  h^{m}  \cdot  v^{-1}  \cdot  h^{-m}, \text{ } v^{n}  \cdot  h  \cdot  v^{-n}  \cdot  h^{-1} \rangle .\]
The group $ \Gamma$ is called the tiling group associated with tiles $ R_1$ and $R_2$. 

We can achieve both of the above conditions of \Cref{tilingcocycle} simply by demanding: $h^{m}=e$ and $v^{n}=e$. Hence, we consider the group $G = \langle h,v | h^m, v^{n} \rangle$.
It is easy to see that $G \cong \mathbb{Z}/m \mathbb{Z}* \mathbb{Z}/n \mathbb{Z}$.

Let $x \in X(\mathcal{R})$ be a tiling and $ i \in \mathbb{Z}^2$. Let $\gamma$ be any path from the origin to $i$ along the tiling $x$. We define
 \begin{center}  $c_{\R}(x, i)$ := product of group elements in $G$ along $\gamma$.\end{center} 
See Figure \ref{rectangle height function} for an example. $ c_{\R}(x, i)$ is independent of the path $ \gamma$ chosen, thanks to the relation \eqref{tilingcocycle}. It is easy to check that $ c_{\R}$ satisfies the cocycle equation.
\begin{figure}[!htbp]
    \centering
    \includegraphics[width=0.4\linewidth]{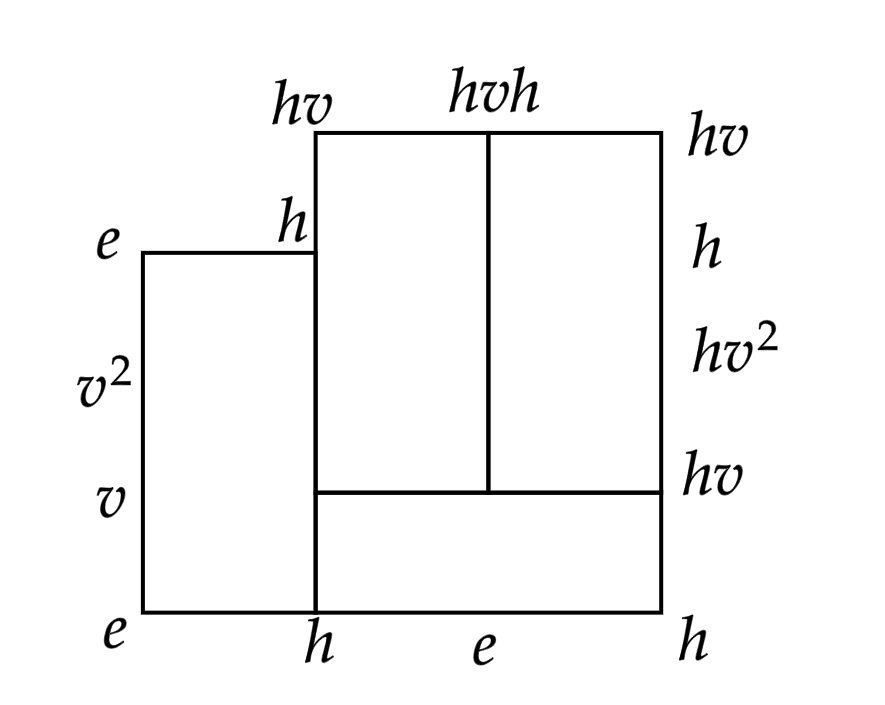}
    \caption{The cocycle for the $2 \times 1 $ and $ 1 \times 3$ boxes takes values in the group $\langle h, v | h^2 ,v^3  \rangle $}
    \label{rectangle height function}
\end{figure}
We consider this cocycle on $X(\mathcal{R})$ henceforth.

Here is a quick application of the cocycle introduced above.
\begin{lemma}\label{lemma: length or breadth is tileable}
If a box $R$ in $ \mathbb{R}^2$ can be tiled by $ \{R_1, R_2\}$ where $ R_1$ is a $ m \times 1$ box and $ R_2$ is a $ 1 \times n$ box, then either $ R$ can be tiled by $ R_1$ or $R$ can be tiled by $ R_2$.    
\end{lemma}
\begin{proof}
Let $l$ and $ b$ be horizontal and vertical sidelengths of $ R$. Since $ R$ can be tiled by $\{R_1, R_2\}$, the word corresponding to the boundary of $ R$ must be identity, i.e., $ h^lv^bh^{-l}v^{-b}= e$. Clearly, this implies that either $ h^l= e$ or $ v^b=e$ in $ \Z/m\Z*\Z/n\Z$. But then either $ l$ is a multiple of $ m$ or $ b $ is a multiple of $n$.     
\end{proof}
It is easy to see that the abelianization of the group $G$ is isomorphic to $\mathbb{Z}/m \mathbb{Z} \times \mathbb{Z}/n\mathbb{Z}$. There is a unique homomorphism $\pi: G \rightarrow \mathbb{Z}/m \mathbb{Z} \times \mathbb{Z}/n\mathbb{Z}$ which satisfies $\pi(h) = (\bar{1}, \bar{0}) $ and $ \pi(v) = (\bar{0}, \bar{1}).$ We will often refer to $ \pi(g)$ as \textbf{modulo class} of $g$.

The following observation will be useful in Section \ref{NO SI SUBSHIFTS}. 
\begin{lemma}
Let $ i = (a, b) \in \mathbb{Z}^2$. Then for any $x \in X(\R)$, $ \pi(c_{\R}(x, i)) = (a \Mod m, b  \Mod {n})$. In particular, $ \pi(c_{\R}(x, i))$ is independent of $ x.$
\end{lemma}
\begin{proof}

For any $x \in X(\mathcal{R})$, looking at the tile configuration near the origin, we see that $c_{\R}(x, e_1) = v^k   \cdot   h   \cdot   v^{-k}$ for some $ k= k(x) \in \{ 0,1, \cdots, n-1\}$. Hence $\pi(c_{\R}(x, e_1)) = ({1}\Mod {m},{0}\Mod {n})$. Similarly $\pi(c_{\R}(x, {e}_2)) = ({0}\Mod {m},{1}\Mod {n})$ for any $x \in X(\mathcal{R})$. By using the cocycle equation (\Cref{cocycle}), we can prove inductively that if $i = (a, b) \in \Z^2$, then $\pi(c_{\R}(x, i)) = (a  \Mod m, b  \Mod {n})$ for any $ x \in X(\R)$.
\end{proof}
In Section \ref{NO SI SUBSHIFTS}, we use the cocycle introduced above to prove the following theorem. 

\begin{restatable}{theorem}{NOSIINTILINGS}\label{theorem: no_SI_rectangular_tiling}
    Let $ \R = \{R_1, R _2, \cdots, R_k\}$ be a collection of boxes in $ \mathbb{R}^2$ with integer sidelengths. If there exist $m,n>1$ such that the $m\times 1$ and $1\times n$ boxes tile each element of $\R$, then there is no strongly irreducible subshift in $ X(\R)$.
\end{restatable}

In particular, this implies that $ X(\R)$ does not support any shift-invariant finitely dependent processes in this case.

By a \textbf{continuous tiling} of $ F(2^{\Z^2})$ by the set of tiles in $ \mathcal{R}$, we mean a continuous equivariant map from $\Psi: F(2^{\Z^2}) \rightarrow X(\R)$. Problem 2 in \cite[Section 4.5]{gao2023continuouscombinatoricsabeliangroup} asked if there exists a continuous tiling of $ F(2^{\Z^2})$ by tiles in $ \R = \{ R_1, R_2 \}$ where $ R_1$ is $ m\times m$ box and $ R_2$ is $ (m+1)\times(m+1)$ box, for an integer $ m \geq 2$. The following corollary answers this question in the negative.
\begin{restatable}{corollary}{NOCONTINUOUSTILINGS}\label{NO CONTINUOUS ON Tilings}
      Let $ \R = \{R_1, R _2, \cdots, R_k\}$ be a collection of boxes in $ \mathbb{R}^2$ with integer sidelengths. If there exist $m,n>1$ such that the $m\times 1$ and $1\times n$ boxes tile each element of $\R$, then there is no continuous tiling of $ F(2^{\Z^2})$ by $ X(\mathcal{R})$.
\end{restatable}

Note that by \Cref{lemma: length or breadth is tileable}, the condition in the above corollary is equivalent to the existence of a partition $\R = A \sqcup B$, $($where $A$ or $B$ are possibly empty$)$ such that the horizontal $($resp. vertical $)$ sidelength of each rectangle in $A$ $($resp. in $B$ $)$ is divisible by $m$ $($resp. by $n$ $)$. It will follow from \Cref{corollary: positive condition for CONTINUOUS} that if a collection of boxes $ \R$ in $\mathbb{R}^2$ does not satisfy the assumptions of \Cref{theorem: no_SI_rectangular_tiling}, then there exists continuous tiling of $F(2^{\Z^2})$ by $ X(\mathcal{R})$. Thus, we have the following characterization.
\characterisationboxes*
\end{example}
\begin{example}\label{Hom cocycle}
For $d \geq 2$, consider the homshift $ \textit{Hom}(\Z^d , H)$ where $ H = (V, E)$ is a countable undirected graph. We further assume that $ H$ is simple, i.e., it does not contain self-loops and multiple edges. Cocycles on $ \textit{Hom}(\Z^d, H)$ were used in \cite{gao2023continuouscombinatoricsabeliangroup}, \cite{entropyminimality} and \cite{chandgotia2025undecidabilityblockgluingclasses} to answer various questions related to these homshifts.

 For every edge $e=(x,y) \in E$, we denote the reverse edge $ (y,x) $ by $ \bar{e}$. Consider the free group $F(E)$ whose generators are the set of edges in $ E$, where the inverse of any edge $e$ is declared to be $ \bar{e}$, that is, $E$ is a symmetric set of generators for $ F(E)$. Our cocycle will take values in a quotient of this group. Intuitively, if we want to ensure that the cocycle value at a point $v \in \Z^d$ is independent of the path from origin to $ v$, it is enough to ensure that any word obtained by traveling around a pattern on a square in $\Z^d$ is identity.

More precisely, a cycle of length $n$ in $ H$ is a sequence $ (v_0, v_1,\cdots, v_n)$ of vertices in $H$ such that $ (v_i, v_{i+1})\in E $ for all $i$ and $ v_n= v_0$. Let $ e_i:=(v_{i-1}, v_i)$ for $ 1 \leq i\leq n $. Then the word corresponding to the $n$-cycle $ (v_0, v_1,\cdots, v_n)$ is $e_1e_2\cdots e_n$. We demand that every word corresponding to a $4$-cycle is equal to identity. To this end, let $\mathcal{P}_4 $ denote the collection of all words corresponding to a $4$-cycle in $H$. Let $ N$ be the smallest normal subgroup of $ F(E)$ that contains all words in $\mathcal{P}_4$.  Then we define $ \mathcal{G_{H}} = F(E)/N$. Let $ \pi: F(E) \rightarrow F(E)/N$ denote the natural quotient map.

For $ x \in Hom(\Z^d, H)$, we define
\[ c(x, e_i) := \pi((x_{\mathbf{0}},x_{e_i})).\]
This is well defined since $ (x_{\mathbf{0}}, x_{e_i})\in E$.

Now, if $ x \in Hom(\Z^d, H)$ and $v \in \Z^d$, then for any path $ \gamma$ from the origin to $v$ in $ \Z^d$, we define $c_{H}(x, v)$ as the product of $\pi(e)$ over edges along $\gamma$, in the group $\mathcal{G}_{H}$. The value $c_{H}(x, v)$ is independent of $\gamma$, since we have quotiented the group $ F(E)$ by the subgroup $N$.

A $ 4$-cycle $(v_0, v_1, v_2, v_3, v_0)$ is \textit{non-backtracking} if $ v_0 \neq v_2$ and $ v_1 \neq v_3$. If the graph $ H$ does not contain a non-backtracking $4$-cycle, then we call it four-cycle free. If $ H$ is four-cycle free, then $N$ is the trivial group and therefore, $\mathcal{G_{H}}= F(E)$.
\begin{remark}\label{remark: Hom cocycle is same as universal cover}
    The cocycle $($or height function $)$ used in \cite{entropyminimality} takes values in universal cover of the graph $H$. It is easy to see that when $ H$ is four cycle free graph, then we can naturally equip the universal cover with a group structure so that this cocycle is the same as that of $ c_{H}$.
\end{remark}

We will use this cocycle to prove the following theorem.
\begin{restatable}{theorem}{NOSIINHOM}\label{NO SI shifts in Hom}
    If $ H$ is a finite four-cycle free simple graph. Then for $d \geq 2$, $\textit{Hom}(\Z^d, H)$ does not contain any strongly irreducible subshifts.
\end{restatable}
\end{example}
\begin{example}\label{ribbon tiling cocycle}
    For an integer $ n \geq 3$, consider the subshift $X(ribbon, n)$ $($see \Cref{Ribbon tilings} $)$ of the tilings of $\mathbb{R}^2$ using ribbon tiles of order $n$. As before, the corners of the ribbon tiles are assumed to lie on the $ \Z^2$ lattice. In addition, recall that we color the unit square with the lower left corner at $ (i, j) \in \Z^2$ by color $ (i+j) \Mod{n}$. The following was introduced by Scott Sheffield in \cite{ScottRibbon}.
    \begin{figure}[!htbp]
        \centering
        \includegraphics[width=0.4\linewidth]{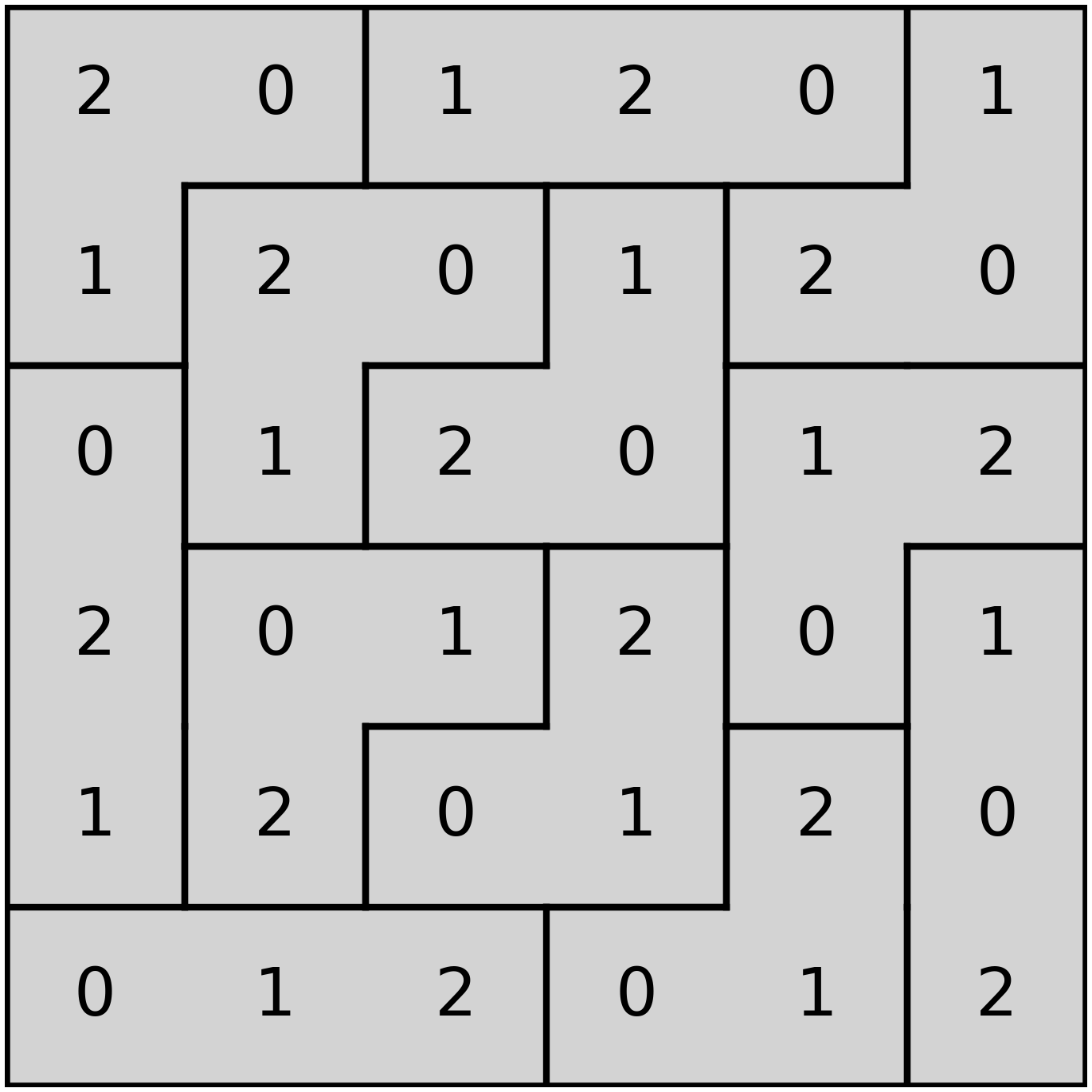}
        \caption{A pattern of ribbon tiling}
        \label{fig:placeholder}
        \vspace{-12pt}
    \end{figure}
    
    We first define a function $c_{rib}: X(ribbon, n) \times \Z^2 \rightarrow \Z^n$ $($whose values will be referred to as `heights' $)$ by specifying height changes across edges. To this end, let $ \mathbf{e}_{i, i+1}, i=0,1,\cdots, n-1$ denote the standard basis elements of $ \Z^n$ $($the choice of notation will become clear soon $)$. We say that an edge in $\Z^2$ has the type $(i,i+1)$, if it lies at the intersection of squares of color $i$ and $i+1$. Note that every edge of $\Z^2$ is of the type $(i,i+1)$ for some $i \in \Z/n\Z$.
    We associate the weight $\mathbf e_{i,i+1}$ with the directed edge of type $ (i, i+1)$ if it is oriented so that the square of color $ i+1$ on the left and the weight $ -\mathbf e_{i,i+1}$ if the square of color $ (i+1)$ is on the right of the edge. Note that if we go around any square of color $i$ in the clockwise direction, then the sum of weights we get is $2 \mathbf e_{i,i+1}-2\mathbf e_{i-1, i}$. Let $x \in X(ribbon, n)$ be a ribbon tiling. We say that an edge in $\Z^2$ lies on a tile in $x$ if it lies on the boundary of some tile in the tiling $x$. Otherwise, if the edge lies in the interior of some tile in $x$, then we say that the edge crosses a tile in $x$. We say that a path $\gamma$ in $ \Z^2$ lies entirely on a tiling $x$, if every edge in $ \gamma$ lies on a tile in $x$. For all $x\in X$, $v\in \Z^2$ and path $\gamma$ from the origin to $v$ lying on the tiling $x$ define
    \[ c_{rib}(x , v) := \text{sum of weights along the directed edges in $ \gamma$}.\]
    It follows that this is independent of the path $ \gamma$ because the sum of weights along the edges of a ribbon tile $($in either orientation $)$ is equal to the sum of the weights around squares of each color, which in turn equals $\mathbf0$.

    However, $ c_{rib}$ is not a cocycle since one needs to be `aware' of the location of an edge in $ \Z^2$ to determine the height change along that edge. This problem can be solved as follows:

    Let $ W$ be the subgroup of $ \Z^2$ defined by
    \[ W:= \{ (i,j): i+j = 0\Mod{n}\}.\]
    Note that if we translate an edge of type $(i, i+1)$ in $ \Z^2$ by a vector in $ W$, then the translated edge also has type $ (i, i+1)$. Therefore, we see that for any $ x \in X(ribbon, n)$ and $ w_1, w_2 \in W$,
    \[ c_{rib}(x, w_1 + w_2 )= c_{rib}(x, w_1) + c_{rib}(\sigma^{w_1}(x ), w_2).\]
    In other words, $ c_{rib}$ is a cocycle with respect to $W$ on $X(ribbon,n)$.

    We now note some observations about this cocycle $($with respect to $W$ $)$ which will be useful in Section \ref{NO SI SUBSHIFTS}. For $ u= (u_1, u_2)$ and $ v = (v_1, v_2)$ in $ \Z^2$, we say that the unit square $ s_{v}$ with its lower-left corner at $v$ is higher $($lower $)$ than the unit square $ s_u$, if $ v_1+v_2 > u_1+u_2$ $($$ v_1+v_2 < u_1+u_2$ $)$. Note that for every ribbon tile, one can go from the lowest square to the highest square by moving either upwards or rightwards increasing the color by $1$ at each step. Thus, if the lowest square has color $j+1$, then the highest square has color $(j+1 +(n-1))\Mod n= j\Mod n$. We say that a tile has type $(j, j+1)$ if the lowest and highest squares in the tile have colors $j+1$ and $j$, respectively. If a directed edge joining adjacent vertices $ u$ and $ v$ in $\Z^2$ has type $(i,i+1)$ with a square of color $(i+1)$ on its left and crosses a tile of type $(j, j+1)$ in a tiling $x$, then the cocycle value changes by $2\mathbf{e}_{j, j+1}-\mathbf{e}_{i, i+1}$ when going from $ u$ to $v$ $($i.e. $c_{rib}(x,v)-c_{rib}(x,u)= 2\mathbf{e}_{j, j+1}-\mathbf{e}_{i, i+1}$ $)$: This can be seen by just summing up the weights going anticlockwise around the squares from color $i+1$ to $j$ to get $2\mathbf{e}_{j,j+1}$ and then subtracting the weight of the edge from $u$ to $v$ because it is not on the boundary of the tile. This implies that for any $x\in X(ribbon,n)$, if the path $(0,0)\rightarrow (1,0) \rightarrow(1,-1)$ is contained in the tile of type $(j, j+1)$ in the tiling $x$, then
    \begin{equation}\label{equation:diagonal_descent}
      c_{rib}(x, (1,-1))= 2e_{j, j+1}. 
    \end{equation}

     Let $ P_i : \Z^n \rightarrow \Z$ denote the projection on $\mathbf e_{i, i+1}$-th coordinate.
    \begin{lemma}\label{ribbonlemma}
    If $ u = (u_1, u_2)$ and $ \sum_{i= 0}^{n-1}P_i(c_{rib}(x, u)) $ is independent of $ x$ and in fact equal to $ u_1 - u_2$. 
    \end{lemma}
    \vspace{-15pt}
    \begin{proof}
    Let $v'=v+(1,0)$. Then $(v,v')$ is a directed edge of type $(i, i+1)$ with $(i+1)$ on the left for some $i$. For such an edge, $c_{rib}(x, v')-c_{rib}(x, v)$ equals either $2\mathbf{e}_{(j, j+1)}- \mathbf{e}_{(i, i+1)}$ or $\mathbf{e}_{(i, i+1)}$ for some $j$ depending on whether the edge crosses a tile or not. Note that $2\mathbf{e}_{(j, j+1)}- \mathbf{e}_{(i, i+1)}= \mathbf{e}_{(i, i+1)}$ when $j=i$. Similarly, if $v'=v+(0,1)$ then the possible values that $c_{rib}(x, v')-c_{rib}(x, v)$ can take are $-2\mathbf{e}_{(j, j+1)}+\mathbf{e}_{(i, i+1)}$ for some $i$ and $j$. From this, the result follows.
    \end{proof}
We will use this to prove the lack of strongly irreducible subshifts in $ X(ribbon, n)$ for $ n \geq 3$ in \Cref{subsection: NO SI in ribbon tilings}. Notice that $ X(ribbon, 2)$ is the space of domino tilings, which also does not contain any strongly irreducible subshifts by \Cref{theorem: no_SI_rectangular_tiling}. Thus, we have the following theorem.

\begin{restatable}{theorem}{NOSIINRIBBON}\label{No SI subshifts in ribbon tilings}
   For $ n \geq 2$, $ X(ribbon, n)$ does not contain a strongly irreducible subshift. 
\end{restatable}
\end{example}
\subsection{Hyperbolic groups}\hspace*{\fill}\\
Let $ G$ be a finitely generated group with a finite generating set $S$. Then for any $ g, h \in G$, a \textbf{geodesic segment} $[g,h]$ joining $ g$ and $ h$ is a path from $ g$ to $h$ of the shortest length in the Cayley graph of $G$. Note that geodesic need not be unique and therefore $ [g,h]$ denotes any one choice among possibly many. 
\begin{definition}
    Let $I$ be a (possibly infinite) sub-interval of $\Z$ and $K, \epsilon > 0$. A map $\phi: I \rightarrow G$ is a $(K, \epsilon)$-\textbf{quasi-geodesic} if for any $ m, n \in I$, 
    \[ \frac{|m-n|}{K} - \epsilon \leq d(\phi(m), \phi(n)) \leq K|m-n| + \epsilon.\]
\end{definition}
For $x, y, z \in G$, a geodesic triangle $ xyz$ is a union of three geodesics $ [x, y], [y, z]$ and $[z,x]$.
\begin{definition}
    Let $ G$ be a finitely generated group with a finite generating set $S$. Then $ G$ is $ \delta$-\textbf{hyperbolic} with respect to $S$, if for any geodesic triangle in the Cayley graph of $G$, any side is contained in the union of the $ \delta$-neighborhoods of the other two sides. We say that $ G$ is \textbf{hyperbolic} if $G$ is $\delta$-hyperbolic with respect to a finite generating set $S$, for some $ \delta\geq 0$.
\end{definition}

For example, free groups are $ 0$-hyperbolic with respect to a free set of generators. If $ G $ is a finitely generated group, then for any $ g \in G$, the function $ n \rightarrow |g^n|$ is subadditive. Therefore, by Fekete's lemma, $\lim_{n \rightarrow + \infty} \frac{|g^n|}{n}$ exists and
\begin{equation}\label{eq: translation length}
\tau(g) := \lim_{n \rightarrow +\infty} \frac{|g^n|}{n} = \inf_{n \in \N} \frac{|g^n|}{n}. \end{equation}
$\tau(g)$ is called the \textbf{translation length} of $g$. Note that $ \tau(g^r) = |r| \tau(g)$ for all $ g \in G$ and $ r \in \Z$. If $ g$ is an infinite order (non-torsion) element in a hyperbolic group $ G$, then the map $ n \rightarrow g^n$ is a quasi-geodesic (\cite[Chapter III.$\Gamma$, Corollary 3.10]{Bridson}) and therefore $ \tau(g) > 0$. 
The following lemma says that in a hyperbolic group, a quasi-geodesic stays close to a geodesic.
\begin{lemma}[Morse lemma]\label{lemma: Morse lemma}
    Let $ G$ be a $\delta$-hyperbolic group with respect to a finite generating set $S$. Let $ \phi: [-k, k'] \rightarrow G$ be a $(K, \epsilon)$-quasi-geodesic for some $ k, k' \in \N$. Then there exists a constant $ A(K, \epsilon, \delta)$ such that the (image of) quasi-geodesic $\phi$ is contained in the $A$-neighborhood of the geodesic $ [\phi(-k), \phi(k')]$ (for any choice of geodesic joining $\phi(-k)$ and $\phi(k')$).
\end{lemma}
We can use the Morse lemma to prove that in a hyperbolic group, $ |g^n|$ grows linearly with bounded error.
\begin{lemma}\label{lemma: linear growth}
    Let $ G$ be a hyperbolic group and $ g \in G$ be an element of infinite order. Then there exists a constant $ A := A(g)$ such that for all $ n \in \Z$,
    \[ |(|g^n| -|n|\tau(g))|\leq A.\]
\end{lemma}
\begin{proof}
    By \Cref{eq: translation length}, we already have $ |g^n|\geq |n| \tau(g)$ for all $ n \in \Z$. To prove the upper bound, assume that $G$ is $ \delta$-hyperbolic with respect to the given set of generators. For fixed $ m, n \in \N$, consider the geodesic segment $\gamma=[1, g^{m+n}]$. The map $ i \rightarrow g^{i}$ is a $(K, \epsilon)$-quasi-geodesic from $ \Z \rightarrow G$ for some $ K, \epsilon >0$. Hence, Morse's lemma guarantees a constant $A(K, \epsilon, \delta)$ such that the quasi-geodesic $ 1 \rightarrow g \rightarrow g^2 \rightarrow\cdots \rightarrow g^{m+n}$ is contained in $ A$-neighborhood of $ \gamma$. In particular, there exists $ p \in \gamma$ such that $d(p, g^m) \leq A$. We have that
    \begin{eqnarray*}
    |g^{m+n}| &= &d(1, g^{m+n}) = d(1, p)+ d(p, g^{m+n})\\
    &\geq& (d(e, g^m) - d(p, g^m)) + (d(g^{m}, g^{m+n})- d(p, g^m))\geq |g^m| + |g^n| -2 A.
    \end{eqnarray*}

    Thus, for any $k \in \N$, 
    $$\frac{|g^{kn}|}{kn} \geq \frac{n |g^k| - 2nA}{kn}= \frac{|g^k|}{k}- \frac{2 A}{k}.$$
Letting $ n \rightarrow + \infty$, we get 
\[ |g^k|\leq k \tau(g) + 2A.\]
\end{proof}
\section{Finitely dependent m-net processes}\label{Section: m net processes}
The main tool that we use in our constructions of finitely dependent processes on subshifts is the finitely dependent $m$-net process which we describe in this section. Let $ G $ be a finitely generated group. We will abuse the notation and let $ G$ also denote the (bounded degree) right Cayley graph of $ G$ with respect to a finite and symmetric set of generators $S \subset G$.
\begin{definition}
For $ m \in \mathbb{N}$, an element $x \in \{ 0,1\}^{G}$ is called an $ \mathbf{m}$-\textbf{net} (with respect to distance $d$ on $G$) if the set $ \{ g \in G: x_g =1\}$ is a maximal $m$-separated set. In other words, it is
\begin{enumerate}
    \item $m$-covering i.e. for every $ g \in G$ there exists $h\in G$ such that $ d(g,h) \leq m$ and $x_h= 1$, and
    \item $ m$-separated i.e. for every $ g $ with $x_g = 1 $ and $ h $ such that $ d(g, h) \leq m$ we have $ x_{h}=0$.
\end{enumerate}

\end{definition}
A process $ (J_g)_{g \in G}$ is called an $\mathbf{m}$-\textbf{net process} if $ (J_g)_{g \in G} \in \{ 0,1\}^{G}$ is an $ m$-net almost surely. It essentially follows from the results of (\cite{Schramm}), \cite{timar2024finitelydependentrandomcolorings} and \cite{Onlinelocality} that there exist shift-invariant finitely dependent $m$-net processes for any $m \geq 1$. We will sketch the proof here for completeness.

 %The existence of automorphism invariant finitely dependent proper $q$-colorings on bounded degree graphs was proved in \cite{timar2024finitelydependentrandomcolorings} and \cite{Onlinelocality}. The following result follows from (\cite[Theorem 1]{timar2024finitelydependentrandomcolorings}).
%\begin{theorem}\label{Timar theorem}
%Let $ H$ be a transitive graph of maximal degree $\delta$. Then there exists a $ 4$-dependent automorphism invariant proper $4^{\frac{\delta(\delta+1)}{2}}$ coloring of $H$.     
%\end{theorem}
 A process $ (Y_g)_{g \in G}$ is called a range-$m$ proper $q$-coloring if $ Y_g \in \{ 0,1, \cdots, q-1\}$ and $ Y_g \neq Y_h$ for any $g$ and $h$ with $d(g,h) \leq m$. The following is an immediate corollary of \Cref{Timar theorem}.
\begin{proposition}\label{FINDEP RANGE M COLORING}
Let $ G $ be a finitely generated group whose Cayley graph has degree $d$. Then there exists a shift-invariant finitely dependent range-$m$ proper $q$-coloring for $q \leq exp(C \cdot m)$ where $ C$ depends only on $d$.
\end{proposition}
\begin{proof} Let $S$ be a finite symmetric set of generators of $ G$. Let $d(g,h)$ denote the distance between $g$ and $h$ in the right Cayley graph of $G$ with respect to $S$. Consider the graph $ G^{(m)}$ whose vertex set is $ G$ and there is an edge between two vertices $ g$ and $h$ iff $d(g,h) \leq m$. By \Cref{Timar theorem}, there exists a finitely dependent automorphism-invariant proper $q$-coloring $(Y_g)_{g \in V(G^{(m)})}$ of the graph $G^{(m)}$ for some $q$. Since $ V(G^{(m)}) = V(G) = G$ and translation by any element $g$ is an automorphism of the graph $ G^{(m)}$,  $ (Y_g)_{g \in G}$ is a shift-invariant finitely dependent range-$m$ proper $q$-coloring of $G$.
\end{proof}

The following theorem also follows from the proof of \cite[Corollary 11]{Schramm}. Here we give a proof for completion.

\begin{theorem}\label{m-net process}
    Let $G$ be an infinite finitely generated group. Then for every $ m \geq 1 $, there exists a shift-invariant finitely dependent $m$-net process on $G$. 
\end{theorem}
\begin{proof}
Let $ (Y_g)_{g \in G}$ be a shift-invariant finitely dependent range$-m$ proper $q$-coloring of $G$. We will construct the desired $m$-net process by means of a sequence of block factors of $(Y_g)_{g \in G}$. To this end, we will define processes $ (J_g^{(i)})_{g\in G}$ $ 1\leq i \leq q$ sequentially. Let 
    \begin{equation} J^{(1)}_{g} = \begin{cases}
    1 &  \text{ if $Y_g =1$ , }\\
       0 & \text{ otherwise.}
        
    \end{cases}
    \end{equation}
Having defined the processes $ (J^{(i)}_{g})$ for $ 1 \leq i \leq q-1$, we define
\begin{equation} J^{(i+1)}_{g} = \begin{cases}
    1 &  \text{ if $ J_{g}^{(i)} = 1$ or  ($J^{(i)}_{g\cdot h} = 0$ for all $ h $ with $|h| \leq m$ and $Y_g =i$) , }\\
       0 & \text{ otherwise.}
    \end{cases}
    \end{equation}
 Then each of the processes $ (J_{g}^{(i)})_{g \in G}, 1 \leq i \leq q$ is $ m$-separated almost surely and a block factor of $(Y_g)_{g \in G}$. It is easy to see that $ (J_g)_{g \in G} := (J_g^{(q)})_{g \in G}$ is an $m$-net process.
    \end{proof}

\section{When do $\Z$-subshifts support shift-invariant finitely dependent processes?}\label{Section: one dimension}
Holroyd and Liggett (\cite[Theorem 22]{Holroyd}) proved that there exists a shift-invariant finitely dependent process on any strongly irreducible subshift of finite type on $ \Z$. Here we prove that shift-invariant finitely dependent processes are dense amongst the space of invariant probability measures for strongly irreducible $\Z$-subshifts (not necessarily of finite type) and discuss when can we make sure that the process gives positive probabilities to all globally allowed patterns.

Let $X$ be a strongly irreducible subshift of $\A^{\mathbb{Z}}$ (not necessarily of finite type) for a finite alphabet $\A$. Let $k$ be the SI distance for $X$. For such a $k$ given any two globally allowed patterns $u, w \in L(X)$ and $k'\geq k$ there exists a globally allowed pattern $v$ of length $k'$ such that $uvw \in L(X)$. For any two globally allowed patterns $ u_0$ and $w_0$, let $ \mathcal{S}_{k}(u_0, w_0)$ be the set of all globally allowed patterns of length $k$ which join $ u_0$ and $w_0$ i.e.
\begin{equation}\label{connecting words}
\mathcal{S}_{k}(u_0, w_0) = \{ v \in \A^k: u_0 v w_0 \in L(X)\} 
\end{equation}
A pattern $v \in L(X)$ is called synchronizing if, for any $u, w \in L(X)$ satisfying $ uv \in L(X)$ and $ vw \in L(X)$, $uvw \in L(X)$. It is easy (\cite[Theorem 2.1.8]{LindMarcus}) to see that if $ X$ is a subshift of finite type for which the maximum length of forbidden pattern is $l$, then any globally allowed pattern of length $ (l-1)$ is a synchronizing pattern. The following proposition is well known \cite[Theorem 1]{SyncWord}. 
\begin{proposition}\label[proposition]{sync word}
    If $ X$ is a strongly irreducible subshift of $ \A^{\mathbb{Z}}$, then there exists a synchronizing pattern in $L(X)$.
\end{proposition}

Let $ \mathcal{M}_{\sigma}(X)$ denote the set of all shift-invariant probability measures on $X$. Every stationary stochastic process $ (X_v)_{v \in \Z}$ where $ X_v \in \A$, corresponds to a shift-invariant measure on $ \A^{\Z}$ defined by, 
\[ \mu([a]) = \mathbb{P}((X_1,X_2, \cdots, X_m)=a) \]
for every $ a \in \A^m$. Henceforth, we will switch between the stationary stochastic process and its corresponding shift-invariant measure for convenience. 
\onedimension*
In the following if $ p < q$ are integers, then we will denote the set $[p,q] \cap\Z $ simply by $ [p,q]$.
\begin{proof}
    Let $k$ be the SI distance of $X$. By Proposition \ref{sync word}, there exists a synchronizing pattern $w \in L(X)$. Let $length(w)=l$ and $ \text{Sync}_{l}$ denote the set of synchronizing patterns of length $ l$. 
    
    Recall that, in order to approximate $ \mu \in \mathcal{M}_{\sigma}(X)$ in weak-\textasteriskcentered{}, it is enough to produce for every  $r \in \mathbb{N}$ and $ \epsilon$, a measure  $\nu_r$ such that
    \begin{equation} \label{weak} |\nu_{r}([a]) - \mu([a])| < \epsilon    \cdot \end{equation}
    for all  $a \in \A^r \cap L(X)$. Now, let $\mu \in \mathcal{M}_{\sigma}(X)$. Take $n$ large enough so that $\frac{2k+l+r}{n} < \frac{\epsilon}{4}$.
   
    Before giving a formal construction of approximating measures $ \nu_r$ which we are going to construct, we will give a brief informal description. We first consider a stationary finitely dependent $n$-net process $\mathbf{J}=(J_v)_{v \in \mathbb{Z}}$ for $n$ chosen as above. Now we build the measure $ \nu_r$ using $J$ as follows: At each occurrence of $1's$ in $\mathbf{J}$, we will place the synchronizing pattern uniformly from the set $ \text{Sync}_{l}$. Within the gap between two consecutive occurrences of the synchronizing patterns, we will sample patterns according to the measure $\mu$, leaving gaps of size $k$ after the occurrence of one synchronizing pattern and before the occurrence of the next synchronizing pattern so that we can legally fill the gaps to obtain an element of $X$. The resulting measure $\nu_r$ will satisfy $\Cref{weak}$ and is finitely dependent.
    
    More precisely, since $X$ is strongly irreducible, for every globally allowed pattern $ a \in L(X)$ and $ w' \in \text{Sync}_{l}$, there exists $ v \in \A^k \cap L(X)$ such that $w'va \in L(X)$. Let $\mathcal{S}_{k}(u,v)$ be as defined before (see \Cref{connecting words}). 
   
    Let $(J_v)_{v \in \mathbb{Z}}$ be a shift-invariant finitely dependent $n$-net process. Recall that the gap between successive $ 1$'s in the $ n$-net lies in $ [n+1,2n+1]$.
     
    Let $ \mathcal{C} = \prod\limits_{t= n+1}^{2n+1} \A^{t-2k-l}$. Thus, each element of $\mathcal{C}$ is a tuple of patterns $ \mathbf{u} = (u_{n+1}, u_{n+2}, \cdots, u_{2n+1})$. For $n+1 \leq t \leq 2n+1$, let $\mathbf{u}_t$ denote the $t$-th component of $ \mathbf{u}$.  We can equip $ \mathcal{C}$ with the product measure $\eta := \prod\limits_{t= n+1}^{2n+1} \mu_{t-2k-l}$ where $ \mu_{i}$ denotes the marginal distribution of $ \mu$ on $ L(X, [1,i])$. Let $(Y_v)_{v \in \mathbb{Z}}$ be an iid process independent of $(J_v)_{v \in \mathbb{Z}}$, valued in $\mathcal{C}$ corresponding to the measure $ \eta$ (that is, each $ Y_v$ has distribution $ \eta$). Let $ \zeta$ be the uniform measure on (finite) set $ Sync_{l}$. Let $ (S_v)_{v \in \mathbb{Z}}$ be the iid process independent of $(J_v, Y_v)_{v \in \mathbb{Z}}$ valued in $Sync_{l}$ such that each $S_v$ has distribution $\zeta$. Finally, consider an iid process $ (F_v)_{v \in \Z}$ independent of $(J_v, Y_v, S_v)_{v \in \Z}$, where each $ F_v$ is a random function from  $L(X)\times L(X)$ to $L(X, [1,k])$ such that for any $ (a, b) \in L(X) \times L(X)$, $ F_v(a, b)$ is uniformly distributed on ${S}_{k}(a, b)$.
    Then we define a process $ (X_v)$ as a block factor of $(J_v, Y_v, S_v, F_v)_{v \in \mathbb{Z}}$ as follows: 
    
    If for some $ v \in \mathbb{Z}$ we have $ J_v =1$ and $ t >0$ is the smallest integer such that $ J_{v + t} = 1$, then define
    \begin{enumerate}
        \item $(X_{v}, X_{v+1}, \cdots, X_{v+l-1}) = S_v$.
        \item $(X_{v+l+k}, X_{v+l+k+1}, \cdots, X_{v+(t-k-1)}) = (Y_v)_{t}$.
        \item $ (X_{v+l}, X_{v+l+1}, \cdots, X_{v+l+k-1})= F(S_v, (Y_v)_{t}) $.
        \item $ (X_{v + t- k}, X_{v+t-k+1}, \cdots, X_{v+t-1}) = F((Y_v)_{t}, S_{v+t}).$
    \end{enumerate}

    Then the process $ (X_{v})_{v \in \mathbb{Z}}$ is stationary and finitely dependent because it is $ 4n+2$-block factor of the stationary finitely dependent process $(J_v, Y_v, S_v, F_v)_{v \in \mathbb{Z}}$.

    Let $\nu_r$ be the measure corresponding to $ (X_v)_{v \in \mathbb{Z}}$ and $\eta$ be the measure corresponding to $ (J_v)_{v \in \mathbb{Z}}$. Given a sample of the process $ J_v$, we say that $ i \in \mathbb{Z}$ is a \emph{good} position, if $ i$ lies in the union of intervals where the samples from the process $ Y_v$ are pasted but not in the last $r$ positions in any of these intervals, that is there exists $ v \in \mathbb{Z}$ and minimum $ t \in \mathbb{N}$ such that $ J_v = J_{v+t} =1$ and $ v+l+k+1 \leq i \leq v+(t-k)-r$. Then for any $a \in \A^r \cap L(X)$,
    \begin{align*}
        \mathbb{P}_{\nu_r}\bigl( (X_0, X_1, \cdots, X_{r-1})=a \bigr) = \text{ }&\mathbb{P}\bigl( (X_0, X_1, \cdots, X_{r-1})=a | 0 \text{ is good}\bigr) \mathbb{P}\bigl(0 \text{ is good})+ \\ &\mathbb{P}\bigl( (X_0, X_1, \cdots, X_{r-1})=a | 0 \text{ is not good}\bigr) \mathbb{P}\bigl(0 \text{ is not good}).
    \end{align*}
    where probabilities on the right hand side of the equation are w.r.t. the joint process $(J_v, S_v, Y_v, F_v)_{v \in \mathbb{Z}}$.  
    Now, $ \mathbb{P}\bigl((X_0, X_1, \cdots, X_{r-1})=a | 0 \text{ is good}\bigr) = \mu([a])$ and given our choice of $n $ and the fact that the process $(J_v)_{v \in \mathbb{Z}}$ is ergodic, we can use   
  the ergodic theorem to conclude that $\mathbb{P}\bigl(0 \text{ is not good})= \mathbb{P}_{\eta}\bigl(0 \text{ is not good}) < \frac{\epsilon}{2}$, hence we have
    \[ | \nu_{r}([a]) - \mu([a])| < \epsilon .  \]\end{proof}

\begin{remark}\label{remark: d bar approximation}
Although we do not prove this here, we can strengthen the conclusion of \Cref{onedimension}. We can use the finitely dependent processes constructed in the above proof to also show the entropy density of the finitely dependent processes in $ \mathcal{M}_{\sigma}(X)$, that is, given $\mu \in \mathcal{M}_{\sigma}(X)$, there exist shift-invariant finitely dependent process $ \mu_n$ such that $ h(\mu_n) \rightarrow h(\mu)$, where $ h(\nu)$ denotes the measure-theoretic entropy of the measure $ \nu$. 

 A process $ (Y_v)_{v \in \Z} \in B^\Z$ is a \textbf{factor} of $ (X_v)_{v \in \Z} \in A^{\Z}$, if there is a measurable map $ \phi: A^{\Z} \rightarrow B$ such that $ Y_v = \phi((X_{v+w})_{w \in \Z})$. Any factor of an iid (FIID) process is finitely determined and the set of processes in $\mathcal{M}_{\sigma}(X) $ which are FIID is closed under the $ \bar{d}$-metric (See \cite[Chapter 5]{Kalikow}). It follows from \cite{zbMATH00177172,Yinonfd} that the finitely dependent processes in $\mathcal{M}_{\sigma}(X) $ are FIID. Therefore, for any SI subshift $X$, the $ \bar{d}$-closure of the set of finitely dependent measures in $ \mathcal{M}_{\sigma}(X)$ is precisely the set of FIID processes in $ \mathcal{M}_{\sigma}(X)$. 
\end{remark}
\begin{remark}
A subshift which has a synchronizing word is called the synchronizing subshift. The above proof can also be made to work for synchronizing subshifts satisfying a weaker mixing condition called non-uniform specification (\cite{zbMATH03656251, nonuniformspecification}). For example, if $ S= \{ n_1, n_2, n_3,\cdots\} \subset \N $ where $ n_1 < n_2 < n_3 < \cdots$ is such that $ gcd(\{s+1: s \in S\}=1$ and $ n_{k+1} -n_{k} = o(n_k)$, then the $S$-gap shift satisfies these properties. 
\end{remark}
\begin{remark}
    We can use the above construction to construct a dense set of shift-invariant finitely dependent processes which are not block factors. Note that if $ w$ is any synchronizing pattern and $ uwv \in L(X)$ then $uwv $ is also a synchronizing pattern. Thus, by the strong irreducibility of $X$, we can find a synchronizing pattern $w'$ which contains all globally allowed letters in $ \A \cap L(X)$. The process we get by just pasting (deterministic pattern) $w'$ at locations of $1's$ and following the rest of construction as it is, cannot be a block factor of an iid: For, as shown in \cite[Proposition 5]{Schramm} any block factor of an iid gives a positive probability to the pattern $a^n$ for any $n $ and for any globally allowed letter $a$. But for a sufficiently large $n $, $ a^n$ does not appear in the process we constructed. 
\end{remark}

The finitely dependent process constructed above does not give positive probabilities to all globally allowed patterns in general. For example, if $ X $ is an even shift (see Example \ref{Even shift}), then any synchronizing pattern must contain symbol $1$. Hence, the finitely dependent process constructed above does not assign a positive probability to the patterns $ 0^{n}$ for large enough $n$.

Now we prove that if $X$ is a strongly irreducible sofic subshift, then there exists a finitely dependent process which gives positive probabilities to all globally allowed patterns. 

\begin{definition}
    A stochastic process which lies in subshift $ X$ almost surely is \textbf{fully supported} in $X$ if it gives positive probabilities to all finite globally allowed patterns in $X$.
\end{definition}
\begin{proposition}\label{fully supported process}
  Let $ X \subset \A^{\mathbb{Z}}$ be a strongly irreducible sofic subshift on $\mathbb{Z}$. Then there exists a fully supported shift-invariant finitely dependent process on $X$.    
\end{proposition}
\begin{proof}
The proof is a simple modification of the proof of \Cref{onedimension}. First note that a block factor of a fully supported shift-invariant finitely dependent process is a fully supported shift-invariant finitely dependent process. Hence, we can assume that $X$ is a subshift of finite type. We can further assume without loss of generality (\cite[Theorem 2.3.2]{LindMarcus}) that all forbidden patterns of $X$ are of length $2$ (i.e. $ X$ is a nearest neighnour subshift of finite type). Let $k$ be the SI distance for $X$. 

The main observation is that any globally allowed letter (patterns of length $1$) in $X\subset \A^\Z$ is synchronizing. Now as in the proof of \Cref{onedimension}, we can take a $n$-net process $(J_v)_{v\in \Z}$ for some $n>k$. At each appearance of $1$ in a sample of the process, place uniformly and independently a letter from $\A\cap L(X)$. Now uniformly choose a valid pattern between successive letters. It is easy to check that the support of the resulting process is entire subshift $X$.
\end{proof}

We do not know if there exists a fully supported shift-invariant finitely dependent process on every strongly irreducible subshift on $ \Z$.

\section{Finitely dependent processes on $\Z^d$-subshifts with the finite extension property}\label{Section: FINDEP On FEP}
In this section, we prove the existence of a dense set of fully supported shift-invariant finitely dependent processes on a $\Z^d$-subshift satisfying the finite extension property.

We first prove the existence of shift-invariant finitely dependent processes for the simpler case of subshifts satisfying the $ 0$-extension property (which is equivalent to the topological strong spatial mixing \cite[Proposition 2.12]{FEP}) since the proof is simpler and works for a general class of groups.
\subsection{Finitely dependent processes on subshifts with  $0$-extension property}\label{subsection: TSSM}\hspace*{\fill} 
\begin{theorem}\label{thm: TSSM}
    Let $G$ be an infinite finitely generated group. Let $ X \subset \A^{G}$ be a subshift of finite type which satisfies $0$-extension property. Then there exists a fully supported shift-invariant finitely dependent process on $X$.  
\end{theorem}
\begin{proof}
    As $X$ has the $0$-extension property, there exists a finite set $ \F$ of forbidden words such that $ X= X_{\F}$ and that every locally allowed configuration with respect to $ \F$ is globally allowed. Let $ k$ be the maximum diameter of a forbidden word in $ \F$. Let $ m > 10k$. Let $ (Z_g)_{g \in G}$ be a shift-invariant finitely dependent range-$m$ proper $ q$-coloring guaranteed by Proposition \ref{FINDEP RANGE M COLORING}. For $1 \leq i \leq q$, let $ W_i = Z^{-1}(i) := \{g \in G: Z_g = i\}$. Note that the sets $ \{W_i: 1 \leq i \leq q\}$ partition $G$ and since $ (Z_g)_{g \in G}$ is a range-$m$ proper $q$-coloring, $ W_i$ is an $m$-separated subset of $G$ for each $i$, almost surely. 

    We now define processes $ (Y_{g}^{(i)})_{g \in G}$, $ 1 \leq i \leq q$ in stages so that process $ (Y_{g}^{(i)})_{g \in G}$ will reveal the output of the desired process on the set $W_1 \sqcup \cdots \sqcup W_i$ and $(Y^{(i)}, Z)$ is a shift-invariant finitely dependent process. Finally, $(Y_{g}^{(q)})_{g \in G}$ will be our desired process.
    
    Let us first construct $Y^{(1)}$. Let $(a_g^{(1)})_{g\in G}$ be an iid uniform process with alphabet $\A\cap L(X)$. Choose a letter $* \notin \A$. Define
    $Y^{(1)}$ by
    $$Y_g^{(1)}= \begin{cases}
        a_g^{(1)}&\text{ if }g\in W_1\\
        *&\text{ otherwise.}
        \end{cases}$$
    Let us make some observations about this process. Since $W_1$ is $m$-separated for some $m>k$ and $k$ is the maximum diameter of the forbidden words, we have that $Y^{(1)}|_{W_1}$ is locally allowed and hence globally allowed. $(Y^{(1)}, Z)$ is shift-invariant and finitely dependent because it is a factor of a shift-invariant finitely dependent process.
    
    Now, suppose that for some $i; 1 \leq i \leq q-1$, the process $ Y^{(i)}$ has been defined such that $(Y^{(i)}, Z)$ is finitely dependent and $Y^{(i)}|_{W_1 \sqcup \cdots \sqcup  W_i}$ is locally allowed and hence globally allowed configuration. Let us construct $Y^{(i+1)}$. Consider the random subset $W_{i+1}$ and the selection of symbols $(\A_g)_{g\in G}$ given by
    $$\A_g=\begin{cases}
        \left\{Y^{(i)}_g\right\}&\text{ if }g\in \sqcup_{j=1}^iW_j\\
        \left\{a\in \mathcal A~:~ \exists x\in X\text{ for which }x_g=a\ \&\  x_h=Y^{(i)}_h\text{ for }h\in W_1\sqcup \ldots \sqcup W_i \right\}&\text{ if }g\in W_{i+1}\\
        \{*\}&\text{ otherwise.}
    \end{cases}$$
    First note that $\A_g$ is non-empty for all $g\in G$. This is obvious for $g\notin W_{i+1}$. For $g\in W_{i+1}$ this follows because $Y^{(i)}|_{W_1\sqcup \ldots \sqcup W_i}$ is globally allowed. Further, $(\A_g)_{g\in G}$ is shift-invariant and finitely dependent since it is a block factor of the shift-invariant finitely dependent process $(Y^{(i)},Z)$: To see this note that since $X$ has the zero extension property and $W_{i+1}$ is $m$-separated for $m>k$ it follows that for $g\in W_{i+1}$,
    $$\A_g=\{a\in \mathcal A~:~ \exists b\in \mathcal L(X, g.\B(k)) \text{ such that }b_g=a\ \&\ b_h=Y^{(i)}_h\text{ for }h\in (W_1\sqcup\ldots\sqcup W_i)\cap g.\B(k)\}.$$
    Since $m>k$ this also implies that if $x\in (\A\cup \{*\})^G$ satisfies $x_g\in \A_g$ for all $g\in G$, then by the $0$-extension property $x|_{W_1\sqcup\ldots W_{i+1}}$ is locally allowed and hence globally allowed. Now, define $Y^{(i+1)}_g$ as a uniform and independent choice of $\A_g$ for all $g\in G$. Consequently $(Y^{(i+1)}, Z)$ is a finitely dependent process such that $Y^{(i+1)}|_{W_1\sqcup\ldots \sqcup W_{i+1}}$ is globally allowed.
    
    It follows that the process $Y^{(q)}$ obtained in this way is a shift-invariant finitely dependent process on $X$. Now we will show that $(Y^{(q)}_{g})_{g \in G}$ is fully supported. Clearly, it is enough to show that for any $n > m$, the probability of any pattern $ a \in L(X, \B(n))$ w.r.t. $ (X^{(q)}_{g})_{g \in G}$ is positive. To this end, let $z \in \{ 1,2,\cdots, q\}^{\B(n)}$ be a range-$ m$ proper $q$-coloring of $ \B(n)$ (i.e. $z_g \neq z_h$ whenever $ g, h \in \B(n)$ are two elements such that $ d(g, h) \leq m$) which has a positive probability w.r.t. the process $(Y_g)_{g \in G}$. Then conditioned on $(Y_g)_{g\in \B(n)}=z$, the probability of $(Y^{(q)}_g)_{g\in \B(n)}=a$ is positive: For conditioned on the event $\{(Y_g)_{g \in \B(n)} = z \}$ (with respect to the joint distribution $(Y^{(q)}, Z)$), the conditional measure of $ Y^{(q)}|_{W_{i+1} \cap \B(n)}$ given the event $Y^{(q)}|_{W_{i} \cap \B(n)}= a|_{W_i \cap \B(n)}$ is a product measure and thus assigns positive probabilities to all patterns on $ W_{i+1} \cap \B(n)$ which are consistent with $a|_{W_i \cap \B(n)}$, in particular to the pattern $ a|_{W_{i+1} \cap \B(n)}$, for every $ 1 \leq i \leq q$. \end{proof}
\begin{remark}
    If the Cayley graph of a group $G$ has degree $d \in \N$, then the space of proper $d+1$ colorings of $ G$ has the $0$-extension property: Given any locally allowed pattern with support $ F \subset G$, we can extend the pattern to $G$ simply by filling the smallest available color at each $g \in G\setminus F$ sequentially (in some order). Thus, there exists a fully supported shift-invariant finitely dependent process on $ \mathcal{X}_{d+1}^{G}$.
\end{remark}

\subsection{Finitely dependent processes on subshifts with the finite extension property}\label{subsection: FEP}\hspace*{\fill} \\

Now we will prove the existence and density of shift-invariant finitely dependent processes on subshifts with the finite extension property for $ G= \Z^d$. In this subsection, the distance between two points will be the $ l_{\infty}$ distance unless stated otherwise.

We first need an auxiliary construction. For $ r \in \N$, we consider a graph $\Z_{r}^{d}$ obtained by connecting any two points $u$ and $ v$ of $ \Z^d$ by an edge if $|| u - v||_{\infty} \leq r$. A subset $ F \subset \Z^d$ is $r$-connected if any two elements of $ F$ can be connected by a path in $ \Z_{r}^{d}$ that lies entirely in $F$. The $r$-connected components of $ W \subset \Z^d$ will be referred to as $r$-components of $W$.

Given any $ r \in \N$, we will construct a shift-invariant finitely dependent coloring $ (Z_{v})_{v \in \Z^d}$, where \linebreak $ Z_{v} \in \{ 1,2,\cdots, D\}$ for a constant $ D \in \N$ to be specified later, with the following property: For $1 \leq i \leq D$, the monochromatic sets $ W_i := \{ v \in \Z^d: Z_{v} =i\}$  will consist of disjoint union of finite $r$-components. In addition, we will ensure that the $r$-components of $W_1$ occupy an arbitrarily large proportion of $ \Z^d$. 

We will construct such a coloring by first obtaining a tiling of $\Z^d$ using a finitely dependent $m$-net process for an appropriate $m$ and then by coloring the regions of each of these tiles using colors from $\{ 1,2,\cdots, D\}$ so that the parts of the tile with higher local complexity (explained later) are given a higher color (with respect to the usual order on $\{ 1,2,\cdots, D\}$).

To this end, let $ m \in \N$. By Theorem \ref{m-net process}, there exists a shift-invariant finitely dependent $ m$-net process $ (J_{v})_{v \in \Z^d}$. We first build a (random) covering of $ \Z^d$ with Voronoi cells whose centers lie in the set $\{v \in \Z^d: J_{v}=1\}$: A \textbf{Voronoi covering} around a subset $ S \subset \Z^d $ is a covering of $ \Z^d$ by subsets $(\mathcal{C}_{S}(v))_{v \in S}$ defined by 
\[ \mathcal{C}_{S}(v) := \{ w \in \Z^d: || w - v ||_{\infty} \leq || w-v'||_{\infty}, \text{for all } v' \in S\}.\]
That is, the cell $ \mathcal{C}_{S}(v)$ around $ v \in S $ consists of those points that are closer to $v$ than any other point in $S$. The subset $ \mathcal{C}_{S}(v)$ is called a \textbf{Voronoi cell} around $v$. Let $ S(\omega)$ be the random subset of $\Z^d$ given by $ S(\omega) := \{ v \in \Z^d: J_{v}(\omega) =1 \}$. Then $ \{ C_{S(\omega)}(v): v \in S(\omega)\}$ gives a random covering of $ \Z^d$.

Note that Voronoi covering is obtained by applying a block factor to the finitely dependent $m$-net process $J$. 

We will use some geometric observations about $ \Z^d$. Recall that $ x \in \{ 0, 1\}^{\Z^d}$ is an $m$-net configuration if $ \{ v \in \Z^d: x_{v} =1\}$ is an $m$-separated and $ m$-covering subset of $ \Z^d$.
\begin{lemma}\label{geometric lemma}
    Let $ x \in \{ 0, 1\}^{\Z^d} $ be an $ m$-net configuration for some $ m \in \N$. Then for all $ c> 0$ and $v\in \Z^d$ there exists $D$ which depends only on $c$ and $d$ such that for any $ m \geq 2$, $\B(v, \lfloor cm \rfloor)$ intersects at most $ D$ Voronoi cells in the Voronoi covering around the set $ \{ w \in \Z^d~:~ x_{w} =1\}$.
\end{lemma}
\begin{proof}
    The proof is a simple volume packing argument. Let $ S = \{ w \in \Z^d : x_{w}=1 \}$. Consider the Voronoi covering $ (C_{S}(w))_{w \in S}$ around the set $S$. Since for any $v \in \Z^d$, there exists an element of $S$ in the $m$-net at a distance at most $m$ from $v$, we see that the number of Voronoi cells intersecting $\B(v, \lfloor cm \rfloor)$ is at most equal to the number of elements in $ \B(v, \lfloor (c+1)m \rfloor) \cap S$. Since $ x$ is an $m$-net configuration, $\lfloor \frac{m}{2} \rfloor$ radius balls around any two elements of $S$ are disjoint. Hence if $ \B(v, \lfloor (c+1)m \rfloor)$ contains $ K$ number of $ 1$'s in $x$, then $ \B(v, \lceil (c+\frac{3}{2})m \rceil)$ contains at least $ K$ disjoint $ \lfloor \frac{m}{2} \rfloor$ radius balls. As the volume of $ l_{\infty}$ ball of radius $ t$ in $ \Z^d$ is $(2t+1)^d$, we obtain the following inequality by comparing volumes.
    \begin{align*}
     \Bigl(3 \left \lceil \left(c+\frac{3}{2} \right)m \right \rceil\Bigr)^d & \geq K \left \lfloor \Bigl(\frac{m}{2}\Bigr)\right \rfloor^d . \\
     \implies K & \leq        4^d \left \lceil 3\left( c+2 \right) \right \rceil^d
    \end{align*}
 Thus, $ \B(v, \lfloor cm \rfloor)$ intersects at most $ D= 4^d \left \lceil 3\left( c+2 \right) \right \rceil^d$ Voronoi cells in the Voronoi covering around the set $S$.     
\end{proof}
For any point $v \in S$, the Voronoi cell $ C_{S}(v)$ around $v$ is contained in the ball of radius $2m$ around $v$. Thus, the above lemma implies that $ C_{S}(v)$ intersects with at most $48^d$ neighboring Voronoi cells.
\begin{lemma}\label{lemma:invariance of Voronoi cells}
    Let $ x \in \{ 0, 1\}^{\Z^d} $ be an $ m$-net configuration for some $ m \in \N$. Let $ S= \{ w \in \Z^d~:~ x_{w} = 1\}$. There exists a constant $ K$ which depends only on $d$, such that for any Voronoi cell $ \mathcal{C}$ around a point $ v \in S$ and any $r \in 
    \N$, $ \frac{|\partial \mathcal{C}^{(r)}|}{|\mathcal{C}|} \leq \frac{Kr}{m}$.
\end{lemma}

\begin{proof}
    Let $v \in S$. Then the Voronoi cell $\mathcal{C}:= C_{S}(v)$ around $v$ is contained in $ \B(v, 2m)$ (ball in the $ l_{\infty}$ metric). Since each Voronoi cell intersects with at most $48^d$ neighboring cells, $\mathcal{C}$ is an intersection of at most $48^d$ halfspaces. Thus, $ \partial C$ has at most $ 48^d$ faces. Each face $F$ of $ \partial C$ is contained in the intersection of a hyperplane with $ \B(v, 2m)$. Thus, there exists $ K_1 = K_1(d)$ such that $ |F^{(r)}|\leq K_1 rm^{d-1}$. Therefore, $ |\partial C^{(r)}| \leq 48^d K_1 rm^{d-1}$. On the other hand, $ \mathcal{C}$ contains $ \B(v, \lfloor\frac{m}{2} \rfloor)$, thereby $ |\mathcal{C}| \geq K_2 m^d$ for some $K_2 = K_2(d)$. \end{proof}

We now construct the coloring $(Z_{v})_{v \in \Z^d}$ promised before. The proof of this lemma is similar to the proof of \cite[Lemma 6.4]{poirier2024contractiblesubshifts}.
\begin{lemma}\label{lemma: coloring of group}
Let $\epsilon > 0$. Then there exists $ D \in \N$ satisfying the following: For any $ r,l \in \N$, there exists $m \in \N$ such that any shift-invariant $ m$-net process $(J_{v})_{v \in \Z^d}$ has a block factor $(Z_v)_{v \in \Z^d}$ with $ Z_v \in \{1,2,\cdots, D\}$ for which
\begin{enumerate}
    \item There exists $M \in \N$ such that for any $ 1 \leq i \leq D$, the $r$-components of the sets $ W_i = \{ v \in \Z^d: Z_{v} =i\}$ have diameter less than $M$.
     \item In the Voronoi covering corresponding to the $ m$-net process $(J_{v})_{v \in \Z^d}$, the subset $ G_{l} := \{ v \in \Z^d: \B(v, l) \subseteq W_1\}$ occupies at least $ (1-\frac{\epsilon}{2})$ volume of each Voronoi cell i.e. $ \frac{|G_{l} \cap \mathcal{C}|}{\mathcal{|C|}} \geq (1 - \frac{\epsilon}{2})$ for each Voronoi cell $ \mathcal{C}$.
\end{enumerate}

\end{lemma}

\begin{proof}
  Let $K$ be as \Cref{lemma:invariance of Voronoi cells}. We set $c= \frac{\epsilon}{8K}$. Let $ D$ be as in \Cref{geometric lemma}. Choose $ m \in \N$ be large enough so that $ \frac{\lfloor cm \rfloor }{D} > r$ and $ \frac{l+1}{m} < \frac{\epsilon}{4K}$. Let $ (J_{v})_{v \in \Z^d}$ be a shift-invariant $m$-net process. Let $ S(\omega) = \{v \in \Z^d: J_{v}(\omega) =1\}$. Consider the Voronoi covering $ (\mathcal{C}_{S(\omega)}(v))_{v \in S(\omega)}$. Since $ rD < \lfloor cm \rfloor$, we know that the ball $ \B(v, rD) \subset \B(v, \lfloor cm \rfloor)$ intersects at most $D$ Voronoi cells. Define $ Z_{v}$ to be the largest $ k \in \{ 1, 2, \cdots, D \}$ such that the ball $\B(v, rk)$ intersects at least $ k$ of the Voronoi cells (in which case it will intersect exactly $ k$ Voronoi cells otherwise $ Z_{v} \geq k+1$). See Figure \ref{Colored Voronoi} for a schematic picture of this coloring. Clearly $(Z_{v})_{v \in \Z^d}$ is a block factor of $ (J_{v})_{v \in \Z^d}$.

  \begin{figure}[ht]
      \centering
      \includegraphics[width=0.5\linewidth]{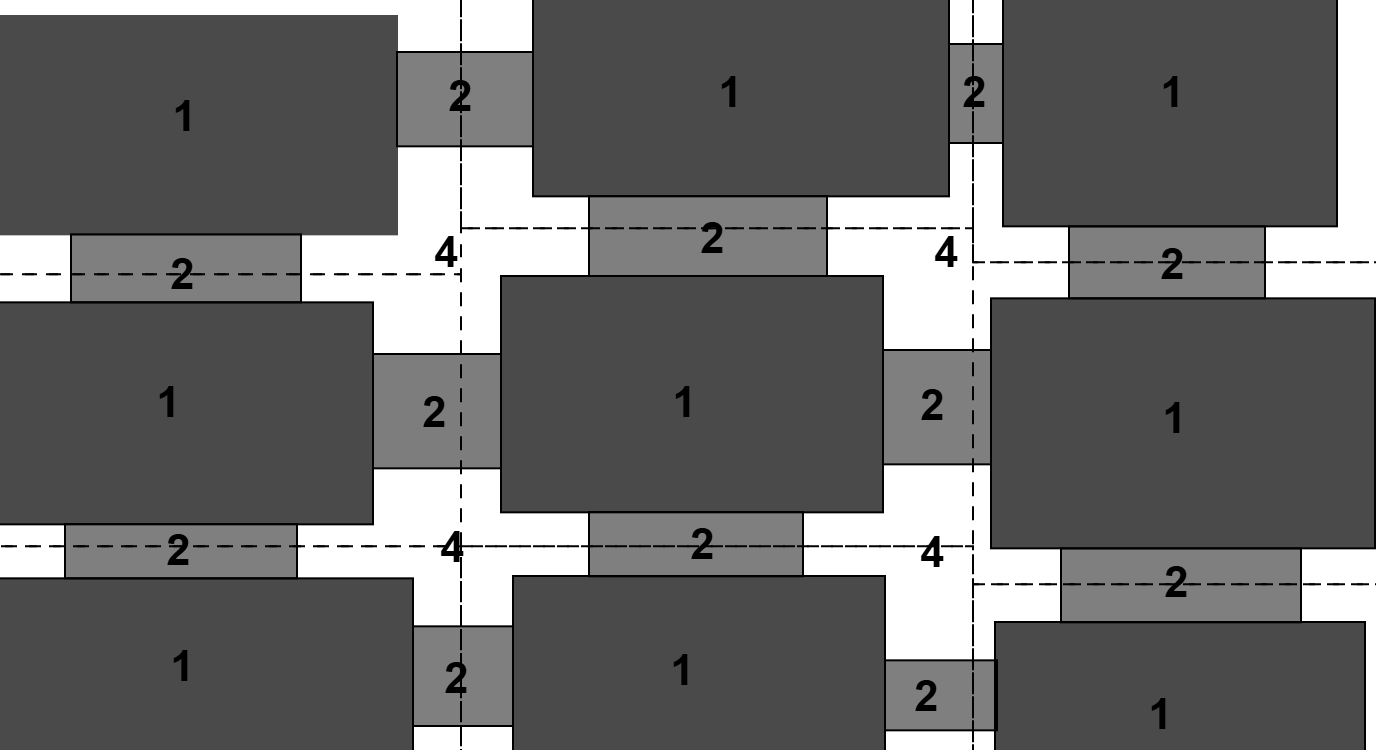}
      \caption{A schematic representation (not up to scale) of coloring $ Z_{v}$ for $r=1$. The dotted boxes are Voronoi cells. All points in a same colored polygons have same color label. Note that while the figure does not have any polygon labeled $3$, they may appear in other configurations.}
      \label{Colored Voronoi}
  \end{figure}
  
  \textit{Proof of $1$}): We first show that the $r$-components of the sets $ W_i = Z^{-1}(i)$ are finite a.s. Suppose on the contrary that $ v_1, v_2, \cdots $ is an infinite sequence of distinct elements such that $ Z_{v_j} = i$ and $ || v_j - v_{j+1}||_{\infty} \leq r$. Then $ \B(v_1, ri)$ intersects $ i$ Voronoi cells, say $ \mathcal{C}_1, \mathcal{C}_2, \cdots, \mathcal{C}_{i}$. Let $ j$ be the first index at which $ \B(v_j, ri)$ does not intersect one of the Voronoi cells of the collection $\{ \mathcal{C}_1, \mathcal{C}_2, \cdots, \mathcal{C}_{i}\}$. Let us say $ \B(v_j, ri)$ does not intersect with $\mathcal{C}_1$ but since $ Z_{v_j} =i$; $ \mathcal{C}_1$ must be replaced by another Voronoi cell, say $\mathcal{C}_{i+1}$ which intersects $ \B(v_j, ri)$. But then $ \B(v_{j-1}, r(i+1))$ intersects $ \mathcal{C}_1, \mathcal{C}_2, \cdots, \mathcal{C}_{i}$ and $\mathcal{C}_{i+1}$, contradicting the fact that the color at $ v_{j-1}$ is $ i$. This shows that the monochromatic $r$-component containing $v_1$ is contained in $ (\mathcal{C}_1  \cup \cdots \cup \mathcal{C}_{i})^{(r)} \subseteq \mathcal{C}_1^{(r)} \cup \cdots \mathcal{C}_{i}^{(r)}$. Each Voronoi cell is contained in a $ l_{\infty}$ ball of radius $ 2m+1$. Hence $ |\mathcal{C}_j^{(r)}| \leq (3(2m+1+r))^d $. Since $1 \leq i \leq D$, this implies that the cardinality of each monochromatic $r$-component is less than $ i\cdot (3(2m+1+r))^d \leq D(3(2m+1+r))^d :=M'$. Therefore, the diameter of each monochromatic $r$-component is bounded by at most $ M:= M'r.$ 

  \textit{Proof of $2$}): Notice that for any Voronoi cell $ \mathcal{C}_{S}(w)$, all points in $ \text{int}_{\B(2r+l+1)}(\mathcal{C}_{S}(w))$ are in $ G_l$, because if $ v' \in  \text{int}_{\B(2r+l+1)}(\mathcal{C}_{S}(w))$ then the ball of radius $ 2r$ around any point in $ v' + \B(l)$ intersects only one Voronoi cell (namely $\mathcal{C}_{S}(w)$) and hence all elements in $ v' + \B(l)$ are contained in $W_1$. Since 
  $$\frac{2r+l+1}{m} < \frac{2c}{D} + \frac{l+1}{m} \leq 2c+ \frac{l+1}{m}< \frac{\epsilon}{2K},$$
  it follows from Lemma \ref{lemma:invariance of Voronoi cells} that
  $$ \frac{|G_{l} \cap \mathcal{C}|}{|\mathcal{C}|} \geq \frac{|\mathcal{C} \setminus \mathcal{\partial C}^{(2r+l+1)}|}{|\mathcal{C}|}\geq 1- \frac{|\partial C^{(2r+l+1)}|}{|\mathcal{C}|}\geq 1- \frac{\epsilon}{2},$$
  that is, the elements in $ G_l$ occupy more than $1-\frac{\epsilon}{2}$ proportion of each Voronoi cell.
  \end{proof}

We now have the necessary ingredients to prove the main theorem of this section.
\FEPcase*
\begin{proof}

    Let $ X$ be a finite extension property subshift with the FEP constant $ g \in \N$ and finite forbidden list $ \mathcal{F}$. Since, $X$ is a strongly irreducible subshift, let $ k$ be the SI constant of $X$. Let $N$ be the maximum diameter of a word in $ \mathcal{F}$. 

     Let $ \mu \in \mathcal{M}_{\sigma}(X)$ and $ \epsilon >0$, $ l \in \N$ be given. We will first construct a shift-invariant finitely dependent process $ \nu$ such that
    \[ |\mu([\alpha]) - \nu([\alpha])| < \epsilon\]
    for all $ \alpha \in \A^{\B(l)}$ and later show how to modify the construction to make $\nu$ fully supported.

    Let $ D$ be as in \Cref{lemma: coloring of group}. Set $r = 2gD + k+1$. Let $ m$ be as guaranteed in \Cref{lemma: coloring of group} and $ (J_{v})_{v \in \Z^d}$ be a shift-invariant finitely dependent $m$-net process. 
    
    By \Cref{lemma: coloring of group}, there exists a coloring $ (Z_{v})_{v \in \Z^d}$ of $\Z^d$ using colors from $\{ 1,2, \cdots, D\}$ which is a block factor of $(J_{v})_{v \in \Z^d}$, such that for each $ 1\leq i \leq D$,  the $r$-components of the monochromatic sets $ W_i = Z^{-1}(i)$ have diameter less than $M$, for some $M$. Note that $r$-components of $ W_i$ are separated from each other by distance at least $r$. We now define processes $ (Y_{v}^{(i)})_{v \in \Z^d}$, $1 \leq i \leq D$ in stages so that the final process $ Y_{v}^{(D)}$ satisfies the desired properties.

    Step i): The $r$-components of $ W_1$ are separated from each other by distance at least $r$. Since $ r > 2gD+k$, the set $ \B(g(D-2)) + W_1$ consists of the (countable) disjoint union of finite subsets which are separated from each other by a distance greater than $k$. For the sake of definiteness and ease of reference, we write this as $\B(g(D-2)) + W_1 = \sqcup_{j =1}^{\infty} V_{j}^{(1)}$, where each $ V_{j}^{(1)} = \mathcal{C}_j + \B(g(D-2))$ where $\mathcal{C}_j$ is a $r$-component of $W_1$.
    
    We will paste samples from the measure $\mu$ on each $ V_{j}^{(1)}$. Since we need to perform this procedure locally, we randomly choose a center from $V_{j}^{(1)}$ and then sample from the marginal distribution of $ \mu$ on the shape $ V_{j}^{(1)}$ as `observed' from this center. To this end, consider an iid process $ U := (U_{v})_{v \in \Z^d}$ independent of the process $ (J_v, Z_v)_{v \in \Z^d}$ where each $ U_{v}$ is a uniform random variable on $[0,1]$.
    Within each finite set $V_{j}^{(1)}$, we first choose the center of $ V_{j}^{(1)}$ to be $ c_{j}^{(1)} \in V_{j}^{(1)}$ such that $ U_{c_{j}^{(1)}} = \min\{ U_{v} : v \in V_{j}^{(1)}\}$. Note that the centers are uniquely defined almost surely. Now we sample patterns on sets $ V_{j}^{(1)}$ according to the marginal distribution of the measure  $ \mu$ on ($ V_{j}^{(1)} - c_{j}^{(1)})$; doing this independently on different sets $ V_{j}^{(1)}$. We label each element in $\Z^d \setminus (\B(g(D-2) + W_1)$ with a fixed symbol not in $\A$, say \textasteriskcentered{}. As the diameter of each $V_{j}^{(1)}$ is uniformly bounded, this defines a shift-invariant finitely dependent process $ (Y_{v}^{(1)})_{v \in \Z^d}$. Since sets $ V_{j}^{(1)}$ are separated by distance at least $k$, the restriction of the process $Y^{(1)}$ on set $ \B(g(D-2))+ W_1$ is a globally allowed configuration. 

    Step ii): Again, $ W_2$ is a disjoint union of $ r$-components that are separated by the distance at least $ r$. Therefore, $ \B(g(D-2)) + W_2$ (and hence $ (\B(g(D-2))+W_2) \setminus (\B(g(D-2))+ W_1)$ is a disjoint union of finite subsets that are separated by distance at least $k$.
    For the ease of reference we write this as, \linebreak $ (\B(g(D-2))+W_2) \setminus (\B(g(D-2))+W_1) = \sqcup_{j=1}^{\infty} V_{j}^{(2)}$. Define the center of the set $V_{j}^{(2)} $ to be that element $ c_{j}^{(2)} \in V_{j}^{(2)}$ for which $ U_{c_j}^{(2)} = \min \{ U_v : v \in V_{j}^{(2)} \}$. Each $ V_{j}^{(2)}$ might be adjacent to some components $ V_{j'}^{(1)}$. Now, sample from the \textit{uniform} distribution on locally allowed patterns on $ V_{j}^{(2)}$ given the already existing patterns from previous step on the set $ (\B(N) + V_{j}^{(2)}) \setminus V_{j}^{(2)}$, making sure that no forbidden words appear, doing this independently on each $ V_{j}^{(2)}$. This gives a locally allowed configuration on $ \B(g(D-2)) + (W_1 \sqcup W_2)$ which may not be globally allowed. However, thanks to the $ g$-extension property of $ X$, the restriction of this configuration to the $ g$-interior, $ \text{int}_{\B(g)} (\B(g(D-2)) + (W_1 \sqcup W_2)) \supset \B(g(D-3)) + (W_1 \sqcup W_2)$ is a globally allowed configuration. Labeling the points in the complement of $\B(g(D-3)) + (W_1 \sqcup W_2)$ by \textasteriskcentered{} defines a shift-invariant finitely dependent process $(Y_{v}^{(2)})_{v \in \Z^d} $. We continue inductively until the Step $ D-1$. We obtain a shift-invariant finitely dependent process $ Y_{v}^{(D-1)}$ whose restriction on $ W_1 \sqcup W_2 \cdots \sqcup W_{D-1}$ is a globally allowed configuration.

    Step iii) : Finally, let $ W_{D} = \sqcup_{j=1}^{\infty} V_{j}^{(D)}$ be the decomposition of $ W_{D}$ into its $ r$-components. Assign the center $ c_{j}^{(D)}$ as before. Now each $ r$-component $V_{j}^{(D)}$ is surrounded by regions from $ W_1 \sqcup W_2 \sqcup \cdots \sqcup W_{D-1}$. We sample uniformly among the locally allowed patterns on each $ V_{j}^{(D)}$ so that no forbidden word appears in $ \B(N) + V_{j}^{(D)}$ (such patterns exist because the already assigned configuration on $ W_1 \sqcup \cdots \sqcup W_{D-1}$ is globally allowed). This gives a configuration on the whole $ \Z^d$ that does not contain a forbidden word and therefore must be a globally allowed configuration. This is our desired shift-invariant finitely dependent process $ (Y_{v}^{(D)})_{v \in \Z^d}$. 

    Now we show that the measure $\nu$ corresponding to the process $ (Y_{v}^{(D)})_{v \in \Z^d}$ indeed approximates $ \mu$ to the desired accuracy. Given a sample of the process $(Z_{v})_{v \in \Z^d}$, we say that $ v \in \Z^d$ is in \textit{good} position if $ v + \B(l)$ is contained in $ W_{1}$. It follows from \Cref{lemma: coloring of group} that the elements in good position occupy more than $(1-\frac{\epsilon}{2})$ proportion of each Voronoi cell in the Voronoi covering corresponding to the $m$-net process $(J_{v})_{v \in \Z^d}$. Consequently, if $ \eta$ is the measure corresponding to the process $ (J_{v})_{v \in \Z^d}$, then $ \mathbb{P}_{\eta}(\mathbf{0} \text{ is good}) > (1 - \frac{\epsilon}{2})$ by the pointwise ergodic theorem. This implies the desired weak-\textasteriskcentered{} approximation by a similar argument as in the proof of Theorem \ref{onedimension}.

    To make sure that the process is also fully supported, we perform the same procedure except that in Step 1, on each set $V_{j}^{(1)}$, we independently sample according to a mixture $ t \mu|_{(V_{j}^{(1)}- c_{j}^{(1)})} + (1-t) \beta_{j}^{(1)}$, where $0 < t <1$ and $\beta_{j}^{(1)}$ is the uniform distribution on globally allowed patterns on $ V_{j}^{(1)}$. The rest of the steps are carried out as above, where we sample uniformly on the allowed patterns on $ V_{j}^{(i)}$, $2 \leq i \leq D$, which are consistent with pre-existing patterns created in steps $ 1, 2, \cdots, i-1$. The proof that the resulting process is fully supported is similar to the argument in \Cref{thm: TSSM}. Furthermore it is easy to see that, if $ t$ is sufficiently close to $1$, then the resulting process also weak-\textasteriskcentered{} approximates $\mu$ to desired accuracy. \end{proof}
\begin{remark}
    As in \Cref{remark: d bar approximation}, it follows that for any $\Z^d$ subshift $X$ satisfying FEP, the $\bar{d}$-closure of the set of finitely dependent processes $\mathcal{M}_{\sigma}(X)$ is precisely the set of FIID processes in $ \mathcal{M}_{\sigma}(X)$. 
\end{remark}
\begin{remark}
    We can use the above proof to show that there exists a fully supported shift-invariant process on any $ G$ subshift with FEP, where $G$ is a countable group with polynomial growth. However, the proof of density of such processes requires that the Voronoi cells in the Voronoi covering are approximately invariant in the F\o lner sense (analog of \Cref{lemma:invariance of Voronoi cells}). We do not know whether Theorem \ref{FEPcase} holds when $ G$ is any countable amenable group. 
\end{remark}
\begin{remark}\label{tower ffiid on fep}
    In \cite[Corollary 11]{Schramm} it was proved that there is a $m$-net process which is a finitary factor of an iid whose coding radius admits a tower function decay. The above proof constructs a block factor from the product of an $m$-net process and an iid process onto an FEP subshift. Thus, there exists a fully supported finitary factor of iid with tower function decay on any factor of an FEP subshift. 
\end{remark}
\begin{remark}
    A similar construction of a (topological) sliding block code from the subshift of $ m$-net configurations to FEP subshifts appears in \cite{BLAND_2025}. However, the map there need not be surjective. We can take the product of the $m$-net subshift with an appropriate full shift and get a factor onto an FEP subshift using the above proof. 
\end{remark}

\section{Finitely dependent processes on certain rectangular tilings}\label{Section: rectangular tilings}

Let $\mathcal{R} =\{R_1$, $R_2$, $\cdots$, $R_k$\} be the set of boxes (rectangles) in $\mathbb{R}^2$ with integer side lengths. Let $ X(\mathcal{R})$ be the corresponding tiling space. In this section, we give a sufficient condition on the side lengths of $ R_i$ which implies that the set of stationary finitely dependent processes supported on $ X(\mathcal{R})$ is weak-\textasteriskcentered{} dense in the set of shift-invariant measures on $ X(\mathcal{R})$.

Let $H$ be the set of horizontal sidelengths that appear in $R_i, i=1, 2, \cdots, k$ and $ V$ be the set of vertical lengths that appear in $R_i$. As an example, if $ R_1, R_2, R_3, R_4$ are boxes of dimensions $ m \times (m+1)$, $ m \times (m+2)$, $ (m+3) \times (m+1)$ and $ (m+3) \times (m+2)$, then $H= \{m, m+3\}$ and $ V= \{m+1, m+2\}$. For each $ h \in H$, let $ V_h$ denote the set of vertical lengths of (all possible) boxes which have $ h$ as horizontal side lengths. So in the above example, $V_m= V_{m+3}= \{ m+1, m+2\}$. Similarly, for each $ v \in V$, let $H_v$ denote the set of horizontal lengths of boxes (among $R_i, i=1,2, \cdots, k$) whose vertical length is $v$. 

Let $ \mathbf{a}$ be a globally allowed pattern of tiles in $X(\mathcal{R})$, which is supported on a subset $ F \subset \mathbb{Z}^2$. Let $ \mathcal{E}(F) $ denote the set of edges in the subgraph obtained by restricting the (undirected) Cayley graph of $ \mathbb{Z}^2$ to $F$. We say that an edge $ e \in \mathcal{E}(F)$ is \textit{closed} in pattern $\mathbf{a}$, if $ e$ lies on the boundary of one of the boxes in $\mathbf{a}$ and is \textit{open} otherwise. For a box $R$, a globally allowed pattern $\bf a$ on $R$ is said to have a closed boundary if all edges joining any two adjacent vertices on $\partial R$ are closed in ${\bf a}$.

 We will use the following lemma. Recall that for $F \subset \Z^2$,  the notation $F^{(N)} = \{\mathbf{v} \in \Z^2 : d_{\infty}(\mathbf{v}, F)\leq N \}$. For a subset $ A \subset \mathbb{N}$, let gcd($A$) denote the greatest common divisor of elements of $A$.
 \begin{lemma}\label{lemma: extending configurations}
     Let $ \R$ be a collection of boxes such that $ gcd(V_h) = gcd(H_v) = 1$ for all $ h \in H$ and $v \in V$. Then there exists $ N \in \mathbb{N}$, such that any globally allowed pattern $ \mathbf{a}$ supported on a rectangular shape $ F \subset \mathbb{Z}^2$ extends to a globally allowed pattern $\tilde{\mathbf{a}}$ on the box $ F^{(N)}$ which has a closed boundary.
 \end{lemma}
 \begin{proof}
Since $ gcd(V_h) = gcd(H_v) = 1$ for all $ h \in H$ and $ v \in V$, there exists $ N_0 \in \mathbb{N}$ such that any $n \geq N_0$ can be written as a non-negative integer combination of elements from any one set $ V_h$ or $H_v$ for $h \in H$ and $ v \in V$. Let $D$ be the maximum diameter of a box in $\R$ with respect to the $l_{\infty}$ norm. Let $N= N_0 + D$.

Let $\mathbf{a}$ be a globally allowed pattern supported on a box $ F= [x,y] \times [z,w]$ for some integers $ x, y, z$ and $w$. There might be some edges on $ \partial F$ which are open in pattern $\mathbf{a}$. We will show that $\mathbf{a}$ extends to a pattern $ \mathbf{\tilde{a}}$ on box $F^{(N)}$ such that all the edges on $\partial F^{(N)}$ are closed in the pattern $ \mathbf{\tilde{a}}$. Consider first the top side of the pattern $\mathbf{a}$. We can extend $\mathbf{a}$ to a tiling $ \mathbf{a}_1$ of $F_1 = [x, y] \times [z, w+N]$ such that all edges on the top side of box $F_1$ are closed in the tiling $\mathbf{a}_1$. To this end, first complete (in any allowed way) the incomplete boxes (if there are any) present on the top side of pattern $\mathbf{a}$ to obtain a pattern $\mathbf{a}_c$. On the top of any box of horizontal sidelength $h$ bordering the top side of $\mathbf{a}_c$, pile up appropriate combinations of tiles from $ V_h$ to reach the closed boundary at height $(w+N)$. We can do this because $gcd(V_h)=1$ for all $ h \in H$. \Cref{a,a_1} illustrate this procedure.
\begin{figure}[ht]
\centering
\begin{minipage}{0.25\textwidth}
\centering
\includegraphics[width=\linewidth]{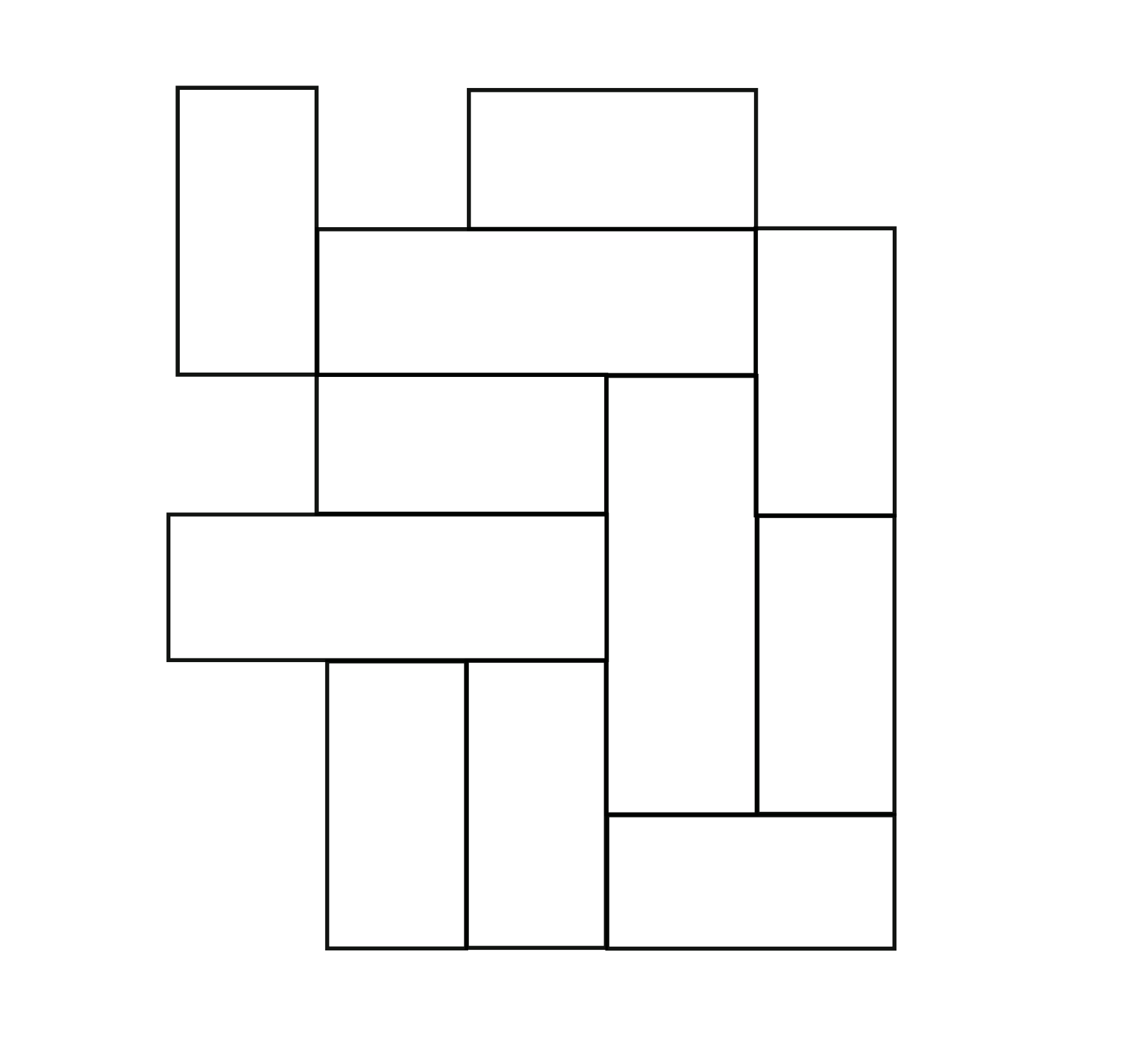}
\caption{$\mathbf{a}$}
\label{a}
\end{minipage}
\begin{minipage}{0.25\textwidth}
\centering
\includegraphics[width=\linewidth]{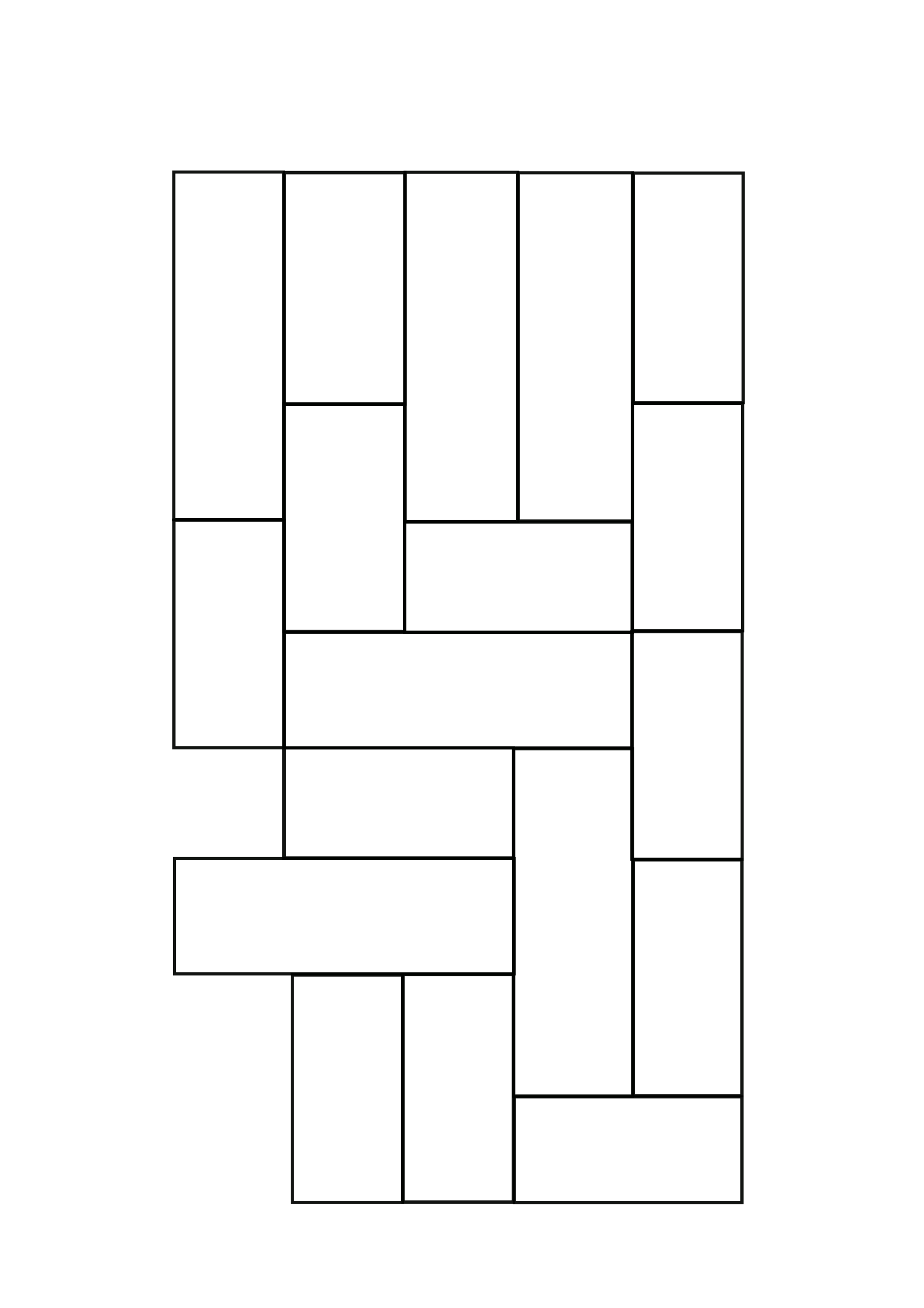}
\caption{$\mathbf{a}_1$}
\label{a_1}
\end{minipage}

\vspace{0.5cm}

\begin{minipage}{0.3\textwidth}
\centering
\includegraphics[width=\linewidth]{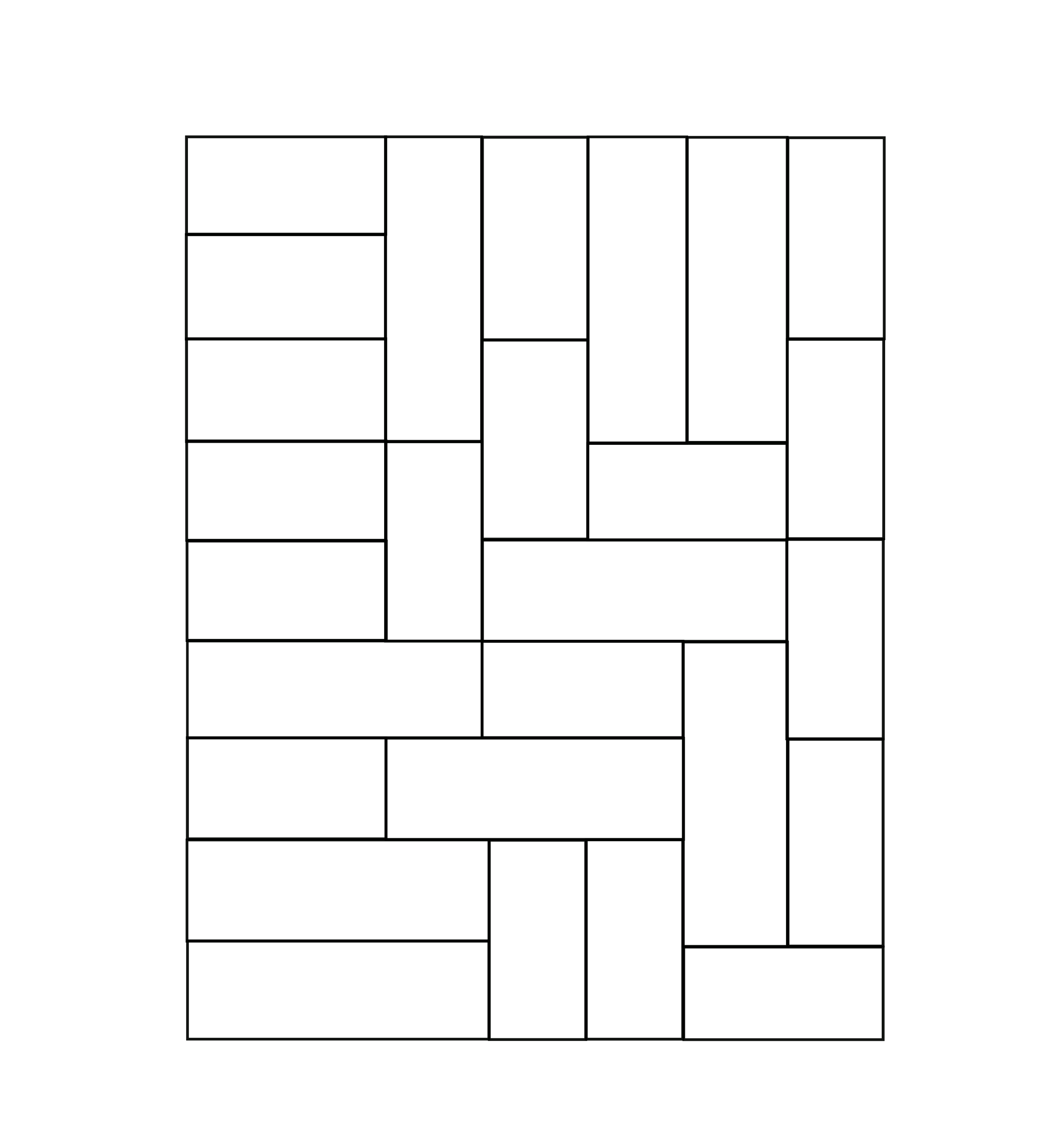}
\caption{$\mathbf{a}_2$}
\label{a_2}
\end{minipage}
\caption{This figure illustrates the steps of the proof of \Cref{lemma: extending configurations} giving an instance of the patterns $\mathbf{a}, \mathbf{a}_1$ and $\mathbf{a}_2$.}
\end{figure}
If there are some open edges on the left side of the pattern $\mathbf{a}_1$, then repeat the above procedure in a similar way on the left side of $\mathbf{a}_1$ and extend the left side of $ \mathbf{a}_1$ to obtain a pattern $\mathbf{a}_2$ on $ F_2=[ x-N, y] \times [z, w+N]$ such that all the edges on the left side of the box $F_2$ are closed in the pattern $\mathbf{a}_2$ (see \Cref{a_2}). If necessary, repeat the same procedure on the bottom and right sides of $ \mathbf{a}_2$ to obtain a pattern $\mathbf{\tilde{a}}$ on the box $ F^{(N)}= [x-N, y+N] \times [z-N, w+N]$ such that all the boundary edges of $F^{(N)} $ are closed in the pattern $\mathbf{\tilde{a}}$. \end{proof}
To prove the main theorem of this section we will need the following theorem which follows from \cite[Theorem 3.1]{GaoJacksonbreakthrough} and its corollary.
\begin{theorem}\label{theorem: perfect tiligs}
    Let $m \in \N$. Let $ \R$ be a collection of all boxes in $\mathbb{R}^d$ such that sidelength of each side is either $m$ or $m+1$. Then there exists a continuous tiling of $ F(2^{\Z^d})$ by $ X(\R)$.
\end{theorem}
Together with \Cref{continuous imples findep} this implies the following.
\begin{corollary}\label{corollary: perfect tilings}
    Let $ m \in \N$. Let $ \R' = \{ R_1, R_2, R_3, R_4\}$ be boxes in $ \mathbb{R}^2$ whose dimensions are $ m \times m$, $ m \times (m+1)$, $ (m+1)\times m$ and $ (m+1) \times (m+1)$ respectively. Then there exists a shift-invariant finitely dependent process supported on $ X(\R')$. 
\end{corollary}

\begin{theorem}\label{rectangle}
    Let $ \R = \{R_1, R_2, \cdots, R_k \}$ be boxes in $\mathbb{R}^2$ with integer sidelengths satisfying $ gcd(V_h) = gcd(H_v) = 1$ for all $ h \in H$ and $v \in V$. Let $ X(\mathcal{R})$ denote the corresponding tiling space. Then the set of shift-invariant finitely dependent processes is dense in $ \mathcal{M}_{\sigma}(X(\mathcal{R}))$. 
\end{theorem}
\begin{proof}

 Let $ \mu \in \mathcal{M}_{\sigma}(X(\mathcal{R}))$ be an invariant measure which we wish to approximate to accuracy $ \epsilon$. For $N$ specified by \Cref{lemma: extending configurations}, choose $m>N$ such that $ \frac{N}{m} =C\epsilon$ where $ C$ will be specified later. Now, there exists a shift-invariant finitely dependent process $(J_v)_{v \in \mathbb{Z}^2}$ which is supported on $X(\R')$ where $ \R'$ is as specified in \Cref{corollary: perfect tilings} for this choice of $m$. Now, within each tile $ \mathcal{T}$ in the random tiling $ (J_v)_{v \in \mathbb{Z}^2}$, sample patterns according to the measure $ \mu$ in $int_{\B(N)}(\mathcal{T}) = \{ \mathbf{w} \in \mathcal{T} \cap \Z^2: d_{\infty}(\mathbf{w}, \partial \mathcal{T}) \geq N\}$. By \Cref{lemma: extending configurations}, we can, within each tile $ \mathcal{T}$, extend these central patterns to a pattern on the whole $ \mathcal{T}$ which has a closed boundary. We can simultaneously and independently perform this procedure in all the tiles $ \mathcal{T}$ in the random tiling $ (J_v)_{v \in \mathbb{Z}^2}$, since all the extended configurations agree on the boundary whenever they intersect. This gives a shift-invariant finitely dependent process $ (Z_{v})_{v \in \mathbb{Z}^2}$ which is supported on $X(\mathcal{R})$. As in the proof of \Cref{FEPcase}, we can choose $ C$ such that the resulting process $\epsilon$-approximates $ \mu$ in weak-\textasteriskcentered{}. \end{proof}
 \begin{remark}
     The construction does not yield fully supported finitely dependent processes. This would follow immediately if these tiling spaces have FEP but this is open. See also \Cref{section: Mixing properties of tiling spaces}. 
 \end{remark}
\section{Cocycles on strongly irreducible subshifts}\label{Section: cocycle rigidity}
In this section, we prove that cocycles defined on strongly irreducible subshifts valued in certain groups must be small perturbations of a homomorphism in some sense. In the next section, we will use this to prove that there are no strongly irreducible subshifts on certain subshifts.

In general, the expectation is that the 1-cohomology of strongly irreducible subshifts should not be too complicated. In Theorem \ref{triviality of cocycles}, we prove that this is indeed the case for cocycles taking values in a certain class of groups.

We will make repeated use of the following observation about the existence of periodic points in strongly irreducible subshifts on $\Z$. However, we note that Hochman (\cite{hochman}) showed that there are strongly irreducible subshifts on $ \Z^2$ which do not contain any fully periodic point. An element $ x= (\cdots x_{-2}x_{-1}x_0x_1x_2\cdots)$ of $ X$ is a periodic point of period $ N$ if $ x_{i}= x_{i+N}$ for all $ i \in \Z$.
\begin{proposition}
    Let $X \subset A^{\Z}$ be a strongly irreducible subshift. Then there exists $N_0$ such that $X$ contains a periodic point of period $ N$ for every $N \geq N_0$.
    \label[proposition]{proposition: periodic in SI}
\end{proposition}

\begin{proof}
We know from Proposition \ref{sync word} that there exists a synchronizing pattern $w $ in $ L(X)$. Let $ k$ be the SI constant of $ X$ and $ |w|$ be the length of the pattern $w$. Then, by the strong irreducibility of $ X$, for every $ N\geq N_0= k+|w|$, there exists a pattern $ u \in L(X)$ of length $ N-|w|$ such that $ wuw$ is an allowed pattern in $ X$. Since $ w$ is a synchronizing pattern, $wuwu$ is also an allowed word. Continuing in inductive manner, we see that $x_N := \cdots wuwuwu \cdots$ is a periodic point of period $N$ in $X$. \end{proof}

Let $G$ be a finitely generated group with a symmetric set of generators $S$ and $ c: X \times \Z^d \rightarrow G$ be a continuous cocycle. Since $X$ is compact, we see that
$$\sup_{x\in X, 1\leq i \leq d}\{ |c(x,\pm\mathbf e_i)|\}:=L < \infty$$
and in fact the supremum is achieved for some $x_0 \in X$ and $ 1 \leq i_0 \leq d$. We refer to $L$ as the Lipschitz constant of the cocycle $c$.
The following proposition follows easily from the ideas in \cite[Section 9]{NoSIsubshifts}.
\begin{proposition}\label{Proposition: SI_subshift_cocycle}
    Let $ X \subset \A^{\Z^d}$ be a strongly irreducible subshift for $ d \geq 2$. Let $G$ be a finitely generated group and $c:X\times \Z^d\to G$ be a continuous cocycle. For all $x\in X$, consider the map $c_x:\Z^d\to G$ given by $c_x(v)=c(x,v)$. Then there exists a constant $K$ such that
    \begin{enumerate}
            \item For all $x,y\in X$ and $u\in \Z^d$, there are group elements $g_1, g_2\in G$ such that $|g_1|, |g_2|<K$ and $c_x(u)=g_1c_y(u)g_2$.
        \item For all $x\in X$ and $v,w\in \Z^d$, there are group elements $g_1, g_2\in G$ such that $|g_1|, |g_2|<K$ such that $c_x(v+w)=c_x(v)g_1c_x(w)g_2$.
    \end{enumerate}
\end{proposition}
In this proof, we will use what is called a higher block presentation of a subshift: Given a subshift $X\subset A^{\Z^d}$ and a natural number $r>1$, consider the subshift $X^{(r)}$ on the alphabet $A'=A^{[-r,r]^d}$ defined as follows: Let $ \B(v, r) = v + [-r, r]^d$ and
\[ X^{(r)} = \{ \left(x|_{\B(v, r)} \right)_{v \in \Z^d}: x \in X \}\]
  The subshift $X^{(r)}$ is called a higher block presentation of the subshift $X$. It is easy to see that these subshifts are isomorphic, meaning that there is a continuous equivariant map from $ X$ to $ X^{(r)}$ which admits a continuous equivariant inverse. In fact, $ \Psi: X \rightarrow X^{(r)}$ given by $ \Psi(x)_{v} := x|_{\B(v, r)}$ is such a map. We will use the well-known fact that if $c: X\to G$ is a continuous map into a topological space $G$, then we can define a corresponding continuous map $\tilde c: X^{(r)}\to G$ given by $\tilde c(x)=c(\psi^{-1}(x))$. This can be used to transfer cocycles from $X$ to $X^{(r)}$ and vice-versa.
\begin{proof}
For each $i$, $ c(\cdot, e_i): X \rightarrow G$ is a (uniformly) continuous map and therefore there exists a radius $R$ such that for any $ 1 \leq i \leq d$, $ c(x, e_i) = c(y, e_i)$ whenever $ x|_{\B(R)} = y|_{\B(R)}$. By considering a  higher-block presentation of $X$ if necessary, we can assume without loss of generality that the value of $ c(x, e_i)$ depends only on $ x_{\mathbf{0}}$ for each $ 1 \leq i \leq d$, that is, for any two elements $ x, y \in X$ with $ x_{\mathbf{0}} = y_{\mathbf{0}}$, we have $ c(x, e_i) = c(y, e_i)$ for each $ 1 \leq i \leq d$.

Let us begin by proving the first part. Let $L=\max_{x\in X, 1\leq i \leq d}(|c(x,\pm\mathbf e_i)|)$ and $k$ be the SI constant of $X$. Choose a set $U$ consisting of a shortest path connecting the origin $0$ and $u$ approximating the straight line joining them in $\mathbb R^d$. Then there exists a vector $v$ such that $\|v\|_1\leq 2(k+d+1)$ and $d(U, U+v)>k$. Then there exists $z\in X$ such that $z|_{U}=x|_{U}$ and $z|_{U+v}= y|_U$. Since $U$ is a connected set containing the origin and $u$ it follows then that
$$c(x, u)= c(z, u)= c(z, v)c(\sigma^v(z), u)c(\sigma^{u+v}(z),-v)=c(z,v)c(y, u)c(\sigma^{u+v}(z),-v).$$
The first equality holds since $x|_U=z|_{U}$. Taking $K=2L(k+1)$, the required inequality follows.

For the second equality, note that $c(x, v+w)=c(x,v)c(\sigma^v(x),w)$ and that it follows from the first, taking $y=\sigma^v(x)$. \end{proof}
\begin{remark}\label{cocycles for subgroups}
    The above proposition holds even if $c$ is a continuous cocycle with respect to a subgroup $ W \leq \Z^d$ of finite index.
\end{remark}
It follows from the above proposition that if $ c$ is a continuous cocycle (with respect to a finite index subgroup $ W \leq \Z^d$) on a strongly irreducible subshift $X$, then  
\begin{equation} D:=\sup_{x,y \in X, v \in \Z^d}|(|c(x, v)| -|c(y, v)| )|< \infty. \label{bounded difference} \end{equation}
\begin{definition}
    Given a strongly irreducible subshift $X$, a cocycle $c$ and a configuration $x\in X$, the \textbf{slope} of $x$ in the direction $v$ for $c$ is given by
\[ \Phi(v) := \limsup_{n \in \N} \frac{|c(x,nv)|}{n||v||_1}.\]
\end{definition}
\begin{remark}
    Note that $\Phi(v)$ is independent of the configuration $x$. Indeed, since $X$ is a strongly irreducible subshift, by \Cref{bounded difference} 
\[ |c(x, (n+m)v)| \leq |c(x, nv)|+ |c(\sigma^{nv}x, mv)| = |c(x, nv)|+ |c(x, mv)|+D\]
for each $ x$. A variant of Fekete's lemma implies then that the limit superior in the definition of slope is actually a limit. Thus, by \Cref{bounded difference}, $ \Phi(v)$ is independent of $x$.
\end{remark}
Note that all cocycles introduced in Section \ref{section:cocycle} were valued in the free product of cyclic groups or direct products of such groups (this includes free groups and $\Z^d$). With this in mind, we say that a finitely generated group $G$ has \textbf{CIC} property if the centraliser of any infinite order element in $G$ is a cyclic subgroup.

For example, all torsion-free Gromov hyperbolic groups (\cite[Chapter III$.\Gamma$, Corollary 3.10]{Bridson}) (in particular free groups) and free products of two cyclic groups (\cite[Chapter 4, Exercise 28]{Combinatorialgrouptheory}) satisfy the CIC property. 

Let $C_{G}(g)$ denote the centraliser of $g$ in $ G$. We say that $ g \in G$ is \textbf{primitive} if $ C_{G}(g) = \langle g \rangle$.
\begin{lemma}\label{conjuagate powers}
    Let $ G$ be a finitely generated group which satisfies the $CIC$ property. If $g$ and $ h$ are two elements of infinite order such that $g^m = \eta h^n \eta^{-1}$ for some $m, n \in \Z \setminus\{ 0\}$ and $ \eta \in G$. Then there exists a primitive element $ g_0 \in G$ and $ r, s \in \Z$ such that $ g = g_0^r$ and $ h = \eta^{-1} g_{0}^{s} \eta$ and $ mr = ns$.
\end{lemma}
\begin{proof}
     We have
    \[ g^m = \eta h^n \eta^{-1}.\]
     Now, since $g^m$ has infinite order, the centraliser $ C_{G}(g^m)$ is a cyclic subgroup. Let $ g_0$ be the generator of $ C_{G}(g^m)$. Note that $ g $ and $ \eta h \eta^{-1}$ commute with $ g^m$. Hence, there exist $ r $ and $ s$ such that $ g = g_0^r$ and $ \eta h \eta^{-1} = g_0^s$, that is, $ h = \eta^{-1} g_0^{s}\eta$. Clearly, we have $mr=ns$. But since $ g= g_{0}^{r}$, $$ C_{G}(g_0) \subset C_{G}(g) \subset C_{G}(g^m) = \langle g_0 \rangle \subset C_{G}(g_0).$$
    Therefore, $ g_0$ is primitive.
    \end{proof}
    
In the following lemma, $ d$ is the Cayley graph metric on $G$. 

\begin{lemma}\label{bounded distance}
    Let $ G$ be a finitely generated group that satisfies the CIC property and $ g\in G$ be a primitive element of infinite order. If for some $ h \in G$,
    \[ \sup_{n \in \N} d(g^{n}, hg^{n})< \infty,\]
    then $ h = g^{t}$ for some $ t$.
\end{lemma}
\begin{proof}
    Suppose that $ \sup_{n \in \N} d(g^{n}, hg^{n})< \infty$. Then $ \{ g^{-n}hg^{n}: n \in \N \}$ is a finite subset of $ G$. By the pigeonhole principle, there exist $ i \neq j \in \N$ such that
    \[ g^{-i} h g^{i} = g^{-j}h g^{j}.\]
    Thus, $ h$ commutes with $ g^{i-j}$ and therefore lies in $ C_{G}(g^{i-j})$. Since $ g^{i-j}$ has infinite order, $ C_G(g^{i-j})$ is cyclic. Let $C_G(g^{i-j})=\langle g_0\rangle$. Then $\langle g\rangle= C_G(g)\subset C_G(g^{i-j})=\langle g_0 \rangle \subset C_{G}(g_0)$. Hence, $g =g_0^{r}$ for some $r$. Therefore, $C_{G}(g_0)\subset C_G(g)$. Thus, it follows that $C_G(g^{i-j})= C_G(g)=\langle g \rangle$ and there exists a $t$ such that $h=g^t$. \end{proof}
We will also use the following observation.
\begin{lemma}\label{lemma: quasimorphism}
    If $ (b_v)_{v \in \Z^d}$ is such that $ b_v \in \Z$, $b_\mathbf{0}=0$ and there exists a homomorphism $ \phi :\Z^d \rightarrow \Z$ and $ K < \frac{1}{4}\min\{|\phi(e_i)|: \phi(e_i) \neq 0\}$ such that       
    \begin{equation}\label{eq: Lipschitz}
    |(|b_v - b_w| - |\phi(v-w)|)| \leq K \end{equation} 
    for every $ v, w \in \Z^d$. Then, there exists $ \epsilon_{i} \in \{ -1, 1\}$ such that the homomorphism 
    \[ \theta(v_1, v_2, \cdots, v_d) := \sum _{i=1}^{d}{\epsilon_i} \phi(v_i e_i)\]
    satisfies $ |b_v -\theta(v)| \leq 2Kd$ for all $ v \in \Z^d$.
\end{lemma}
\begin{proof}
    We claim that for a fixed $v \in \Z^d$ and for every $1 \leq i \leq d$, there exists $ \epsilon_{i} \in \{-1,1\}$ (independent of $ v$) such that $ |(b_{v+ke_i}-b_{v}) -\epsilon_i\phi(ke_i)| \leq 2K$ for all $ k \in \Z$, from which the lemma follows by the triangle inequality.
    
    Since $b_{\mathbf{0}} =0$, it follows from \Cref{eq: Lipschitz} that $ |(|b_{w}| -|\phi(w)|)| \leq K$ for all $w \in \Z^d$. Hence, the claim follows easily when $ \phi(e_i) =0$: Indeed we get for all $v\in \Z^d$ and $\epsilon_i\in \{\pm 1\}$,
    $$|(b_{v+ke_i}-b_v)-\epsilon_i \phi(ke_i)|=||(b_{v+ke_i}-b_v)|-|\phi(ke_i)||\leq K.$$ Thus the claim follows by choosing $\epsilon_i=1$.
    Therefore, assume that $\phi(e_i) \neq 0$. Suppose there exists $m \neq n  \in \Z \setminus \{ 0\}$ such that $ |(b_{v+me_i} - b_v) -\phi(me_i)| \leq K$ and $|(b_{v+ne_i}- b_v) +\phi(ne_i)| \leq K$.
    Then by triangle inequality we see that 
    $$|(|b_{v+me_i} -b_{v+ne_i}| - |\phi({me_i}+{ne_i})|)|\leq |b_{v+me_i}-b_{v+ne_i} - \phi({me_i}+{ne_i})| \leq 2K.$$ But from the assumption on $K$ we have that
    \[ ||\phi({me_i}+{ne_i})|-|\phi({me_i}-{ne_i})||\geq 2|\phi(e_i)|\min(|m|, |n|) \geq 4K\]
    which implies $ |(|b_{v+me_i} - b_{v+ne_i}| - |\phi({me_i} -{ne_i})|)| \geq 2K$ leading to a contradiction.
    Thus, there exists $ \epsilon_i(v) \in \{-1, 1\} $ such that $$ |(b_{v+ke_i}-b_{v}) - \epsilon_i(v) \phi(ke_i)| \leq 2K$$ for all $ k \in \Z $. But the function $ v \rightarrow b_v $ is Lipschitz and hence in particular $ \epsilon_{i}(v) = \epsilon_{i}(w)$ whenever $ v$ and $ w $ are neighbors in $\Z^d$. It follows that $ \epsilon_{i}(v)$ is independent of $v$. \end{proof}
We will now prove the main theorem of the section.
    
\cocycles*

\begin{proof}
     As in the proof of \Cref{Proposition: SI_subshift_cocycle}, considering a  higher-block presentation of $X$ if necessary, we can assume without loss of generality that the value of $ c(x, e_i)$ depends only on $ x_{\mathbf{0}}$ for each $ 1 \leq i \leq d$, that is, for any two elements $ x, y \in X$ with $ x_{\mathbf{0}} = y_{\mathbf{0}}$, we have $ c(x, e_i) = c(y, e_i)$ for each $ 1 \leq i \leq d$.

    If $ \{c(x, v): x \in X, v \in \Z^d \} $ is a finite subset of $ G$, then the conclusion of the theorem is automatic. Hence, we assume that $ \{c(x, v): x \in X, v \in \Z^d \}$ is infinite. We divide the proof into three steps.
    
    \textbf{Step 1:} Let $ k$ denote the SI distance of $X$. First, we will prove that there exists a primitive element $ g_0 \in G$, $ t_1 \in \Z$ and constant $ C_1$ such that for every $x \in X$ and $ n \in \Z$, 
    \[ c(x, ne_1) = \alpha_1 g_{0}^{t_1 n} \widetilde{\alpha_1}  \]
    with $ \max(|\alpha_1|, |\widetilde{\alpha_1}|) \leq C_1$.

    If $ \sup\{ |c(x, ne_1)| : n \in \Z\} < \infty$ for some (and hence every) $ x \in X$, then the conclusion follows by taking $ t_1=0$. Hence, we assume that $ \sup\{ |c(x, ne_1)|: n \in \Z \}= \infty $ for all $ x$. 

    We will crucially use the fact that any one-dimensional strongly irreducible subshift contains periodic points of large enough periods (Proposition \ref{proposition: periodic in SI}). The restriction of $ X$ to $\Z e_1$ is one-dimensional strongly irreducible subshift. Let $N_1$ be as given in Proposition \ref{proposition: periodic in SI} for the subshift $ X|_{\Z e_1}$. Fix an integer $p > N_1$. Then there exists $y_1, y_2 \in X$ such that $ y_{1}|_{\Z e_1}$ is a periodic point of period $ p$ and ${y_2}|_{\Z e_1}$ is a periodic point of period $ p+1$. Let $ g_1 = c(y_1, pe_1)$ and $ g_2 = c(y_2, (p+1)e_1)$.

     Let $ m >k$. Then, by the strong irreducibility of $X$, there exists $ z \in X$ such that 
    \[ (z)_{ne_1}= (y_1)_{ne_1} \text{ and } (z)_{me_2 + ne_1} = (y_2)_{ne_1} \]
for all $ n \in \Z$. 

As there are only finitely many allowed patterns in $X$ which can occur on segments of the kind $ [v, v+me_2]$ where $v$ varies over $\Z^d$, there exist $ l < l'$ such that for $ v_1 = lp(p+1)e_1$ and $ v_2 = l'p(p+1)e_1$ the pattern $z|_{[v_1, v_1+me_2]}$ is the same as the pattern $ z|_{[v_2, v_2+me_2]}$. Let $ \eta := c(\sigma^{v_1}z, me_2) = c(\sigma^{v_2}z, me_2) .$
Let $ v_3 = v_2+ m e_2$ and $ v_4= v_1+ m e_2$.
We can evaluate the difference between the cocycle values between $ v_1$ and $ v_3$ in two different ways:
\begin{align*}
    c(\sigma^{v_1}z, v_3 -v_1) & = c(\sigma^{v_1}z, v_2 -v_1)\cdot c(\sigma^{v_2}z, v_3 - v_2) \\
    &= c(\sigma^{v_1}z, v_4 - v_1)\cdot c(\sigma^{v_4}z, v_3- v_4)
\end{align*}
But by construction,
\begin{align*}
c(\sigma^{v_1}z, v_2 -v_1) &= \prod_{j=0}^{(l'-l)(p+1)-1} c(\sigma^{v_1+jpe_1}(z), pe_1)
\\
&=g_1^{(l'-l)(p+1)}
\end{align*} 
where we used the fact that $ c(x, e_1)$ just depends on $x_{\mathbf{0}}$.

Similarly, $c(\sigma^{v_4}z, v_3 - v_4)= g_{2}^{(l'-l)p} $. Therefore, we have
\[ g_1^{(l'-l)(p+1)} = \eta  g_{2}^{(l'-l)p} \eta^{-1}.\]
Since $\sup\{ |c(z, ne_1)|: n \in \Z \} = \sup \{ |c(\sigma^{me_2}z, ne_1)|: n \in \Z \} = \infty$, the elements $ g_1$ and $ g_2$ are of infinite order. Hence, according to Lemma \ref{conjuagate powers}, there exists a primitive element $g_0 \in G$ and $ r,s \in \Z \setminus \{0\}$ such that $g_1 =  g_0^{r}$ and $ g_2 = \eta^{-1} g_{0}^{s} \eta$ and $ r (l'-l)(p+1)= s(l'-l)p$. Since $ (l'-l)$ is non zero, we have 
\[ r(p+1) = sp.\]

As $ gcd(p, p+1) =1$, this in turn gives us $ r= t_1p$ and $ s= t_1(p+1)$ for some $ t_1 \in \Z \setminus \{0\}$. Thus, we see that $ g_1= g_{0}^{t_1p}$ and $ g_2 = \eta^{-1} g_{0}^{t_1(p+1)} \eta$. Therefore, for all $ k \in \Z$,
 \[ c(z, kpe_1) = g_{1}^{k} = (g_{0}^{t_1})^{kp} .\]
 
Let $ L$ denote the Lipschitz constant of the cocycle $c$. Then for all $n \in \Z$ there exists $\alpha_1'\in G$ such that $|\alpha_1'|\leq pL$ and
$$ c(z, ne_1) = g_{0}^{t_1n}{\alpha}_1'.$$
This implies by Proposition \ref{Proposition: SI_subshift_cocycle} that there exists a constant $C_1$ such that for all $x\in X$ and $n\in \Z$ there exist $\alpha_1, \widetilde{\alpha_1}\in G$ for which $ |\alpha_1|, |\widetilde{\alpha_1}|\leq C_1$ and
\[ c(x, ne_1) = \alpha_1 g_{0}^{t_1n}\widetilde{\alpha}_1.\]

A useful consequence of the above argument is the following: If $q>N_1$ and $y|_{\Z e_1}$ have period $q$, then for any $ m >k$ consider an element $ z' \in X$ with $ z'_{ne_1} = (y_1)_{ne_1}$ and $ z'_{me_2 + ne_1} = y_{ne_1}$ for all $ n \in \Z$. Then we know that for any $ l \in \Z$, 
\[ c(z', (lpq)e_1) = g_{0}^{t_1 plq}.\]
We can argue as above and conclude that there is $ \eta \in G$ with $ |\eta| \leq C_1$ such that 
\begin{equation}
c(y,qe_1)=\eta g_0^{t_1q}\eta^{-1}.\label{equation:conjugation_when_periodic}
\end{equation}

\textbf{Step 2:}
Now we will show that there exists $g_0 \in G$ and a constant $C$ such that for every coordinate direction $ e_i$, there exist $ t_i \in \Z$ such that for all $ x \in X, n \in \Z$, there exist $\alpha_i, \widetilde{\alpha}_i \in G$ with $\max(|\alpha_i|, |\widetilde{\alpha}_i|) \leq C$ satisfying,
\begin{equation} c(x, ne_i) =  \alpha_{i} g_{0}^{t_i n} \widetilde{\alpha_i} . \label{i direction}\end{equation}
Without loss of generality, assume that  $\sup\{ |c(x, ne_1)|: n \in \Z \}= \infty $ for some (and hence all) $x \in X$. Then by the previous step, there exists a primitive element $ g_0 \in G$ of infinite order, $ t_1 \in \Z \setminus \{0\}$ and a constant $ C_1$ such that for every $ x\in X$ and $ n \in \Z$,
\[ c(x, ne_1) = \alpha_1 g_{0}^{t_1 n} \widetilde{\alpha_1}.\]
for some $\alpha_1, \widetilde{\alpha_1} $ with $|\alpha_1|, |\widetilde{\alpha_1}| \leq C_1 $.
 
Let $2\leq i\leq d$. If $ \{ |c(x, ne_i)| : n \in \Z\}$ is bounded in $ G$ for some (and hence all) $ x \in X$, then we can simply take $ t_i=0$. So assume that $ \sup \{ |c(x, ne_i)| : n \in \Z\} = \infty$ for all $ x \in X$. 

As $X|_{\Z e_j}$ is a strongly irreducible subshift for all $j$, \Cref{proposition: periodic in SI} implies that there exists $N_j>k$ such that $ X|_{\Z e_j}$ contains a periodic point for every period $m \geq N_j$. Fix $ m > N_i$. Let $ p \geq N_1$ and $y_1 \in X $ be as in Step 1 ($y_1|_{\Z e_1}$ is periodic of period $p$) and $ y_3 \in X$ be such that $ (y_{3})|_{\Z e_i}$ is a periodic point of period $ m$. Let $ g_3 = c(y_3, me_i)$. From Step 1, it follows that there exists a primitive element $ h_0 \in G$ and $ t_{i} \in \Z \setminus \{ 0\}$ such that $ g_3= h_{0}^{t_{i}m}$. Note that since the set $\{ g_3^n: n \in \Z\}$ is unbounded, $g_3 \neq e$. We will prove that $ h_0 = \beta^{-1} g_{0} \beta$ for some $ \beta \in G$. This will imply \Cref{i direction} by Proposition \ref{Proposition: SI_subshift_cocycle}.

Since $m,p > k$, due to the strong irreducibility of $X$, there exists $\bar{z}\in X$ such that
\[ \bar{z}_{kme_i + ne_1} = (y_1)_{ne_1}\]
for all $ k \in \Z$ and $ n \in \N\cup\{0\}$ and also 
\[ \bar{z}_{-pe_1 + n'e_i} = (y_3)_{n'e_i}\]
for all $ n' \in \Z$.

Now $ c(\sigma^{lme_i}\bar{z}, kpe_1) = c(\sigma^{(l+1)me_i}\bar{z}, kpe_1) = c(y_1, kpe_1) = g_{0}^{t_1kp}$ for every $l \in \Z$ and the distance between $ c(\sigma^{lme_i}\bar{z}, kpe_1)$ and $ c( \sigma^{lme_i}\bar{z}, me_i + kpe_1)$ is uniformly bounded by $ mL$. As $ g_0$ is a primitive element of infinite order, it follows from Lemma \ref{bounded distance} that for all $l \in \Z$, there exists an integer $ r_l$ such that 
\[ c(\sigma^{lme_i}\bar{z}, me_i) = g_{0}^{r_{l}}.\]
By the pigeonhole principle, there exist integers $ l < l'$ such that
\[ c(\sigma^{lme_i}\bar{z}, -pe_1) = c(\sigma^{l'me_i}\bar{z}, -pe_1) =: \beta.\]
Note that $|\beta| \leq pL$. Now, consider the box with endpoints $ v_1 =lme_i - pe_1$, $ v_2 = lme_i$, $ v_3= l'm e_i$ and $ v_4 = l'me_i - pe_i$. 
We evaluate the difference between the cocycle values at $v_4$ and $v_2$ in two different ways.
\begin{align*}
    c(\sigma^{v_2}\bar{z}, v_4-v_2) &= c(\sigma^{v_2}\bar{z}, v_3 - v_2)\cdot c(\sigma^{v_3}\bar{z}, v_4 - v_3)\\
    &= c(\sigma^{v_2}\bar{z}, v_1 - v_2) \cdot c(\sigma^{v_1}\bar{z}, v_4 -v_1).
\end{align*}
We know that $c(\sigma^{v_2}\bar{z}, v_3 - v_2) = g_{0}^{r'}$ for some $r'$ and $ c(\sigma^{v_1}\bar{z}, v_4 - v_1) = g_3^{(l'-l)}$. Thereby
\[ h_{0}^{t_{i}m(l'-l)}=g_{3}^{l'-l} = \beta^{-1} g_{0}^{r'} \beta. \]
Since $ g_0$ is primitive, $ \beta^{-1} g_0 \beta$ is also primitive. But $h_0$ is also primitive. Hence \Cref{conjuagate powers} implies that
$$ h_0 = \beta^{-1} g_{0} \beta.$$

\textbf{Step 3:}  We complete the proof of the theorem assuming $ d=2$. The proof for $ d >2$ follows by an induction argument, which we will sketch at the end.

We know from Step 2, that there exists a primitive element $ g_0 \in G$, $ t_1, t_2 \in \Z$ such that 
\begin{equation}
c(x, ne_i) = \alpha g_{0}^{t_i n} \widetilde{\alpha} \label{directed cocycles}
 \end{equation}
for $ i=1,2$, where $ |\alpha|, |\widetilde{\alpha}| \leq C$.

If either $ t_1 =0$ or $ t_2 =0$, then we are done by Step 1 and 2. Hence, assume that $ t_1 \neq 0$ and $t_2 \neq 0$. For the constants $A(g)$ as in \Cref{lemma: linear growth}, let $ m >\max\{ N_2, k, \frac{4}{|t_1|\tau(g_0)}(4C + \max \{2A(\delta g_0 \delta^{-1}): |\delta| \leq 2C \})\}$. We will show that there exists $ \tilde{z} \in X$ and a constant $C'$ such that for each $ v \in m\Z^2$, we will have 
\[ c(\tilde{z}, v) = h_1'\theta(v)h_2'\]
for some $ h_i'= h_i'(x, v)$ satisfying $ |h_i'|\leq C'.$ By \Cref{Proposition: SI_subshift_cocycle}, this is sufficient to prove the theorem.

By \Cref{proposition: periodic in SI}, there exists $ y \in X$, such that $ y|_{\Z e_2}$ is a periodic point of period $m$. Let $ h= c(y, me_2)$. It follows from \Cref{equation:conjugation_when_periodic} that $h = \delta g_{0}^{t_2m}\delta^{-1} $ for some $ \delta \in G$.

Since $ X$ is strongly irreducible, there exists $ \tilde{z} \in X$ which contains copies of the vertical configuration $ y|_{\Z e_2}$ at all points $lme_1$, where $l$ varies over $ \Z$, that is, for all $ l , n\in \Z$ 
\[ \tilde{z}_{(lme_1+ne_2)}= (y)_{ne_2}.\]
 Since $g_0$ is primitive, \Cref{bounded distance} implies that for every $ l\in \Z$,
\[ c(\tilde{z}, lme_1) = \delta g_{0}^{b_l} \delta^{-1}\]
for some $ b_l \in \Z$. But \Cref{directed cocycles} implies that 
\[ c(\tilde{z}, lme_1) = \delta_1 g_{0}^{t_1 lm} \delta_2\]
where $ |\delta_1|, | \delta_2| \leq C.$
Thus we have, 
\[ g_0^{b_{l}} = \delta'g_{0}^{t_1lm}\delta'^{-1}\delta_{3} \]
for some $ \delta', \delta_3$ with $ |\delta'|, |\delta_3| \leq 2C$. Thus, we have 
\[ |(|g_0^{b_{l}}| - |\delta'g_{0}^{t_1lm}\delta'^{-1}|)|\leq 2C.\]

By \Cref{lemma: linear growth}, this implies that there exists a constant $K =\frac{1}{\tau(g_0)} (4C + 2\max \{A(\delta g_0 \delta^{-1}): |\delta| \leq 2C \})$ such that
\[|(|b_{l}| - |t_1lm|)| \leq K.\]
for all $l \in \Z$. Clearly, $ b_0= 0$. Thanks to our choice of $ m$, we have $K < \frac{1}{4}|t_1 m|$. Thus, by \Cref{lemma: quasimorphism}, there exists $ \epsilon \in \{ -1, 1\}$ such that $ |b_{l} - \epsilon t_1 lm| \leq 2K$. Let $ \Delta(l) = b_l-\epsilon t_1lm$.

Let $ \theta: \Z^2 \rightarrow G$ be the homomorphism such that $\theta(e_1) = g_0^{\epsilon t_1} $ and $ \theta(e_2) = g_{0}^{t_2}$. Then, $$ c(\sigma^{lme_1}\tilde{z}, kme_2) =c(y, kme_2) = \delta g_{0}^{t_2km} \delta^{-1}. $$ Thus, we see that
\begin{align*} \label{hard}
    c(\tilde{z}, lme_1+kme_2 ) &= c(\tilde{z}, lme_1)\cdot c(\sigma^{lm e_1}\tilde{z}, kme_2) \\
    & = (\delta g_{0}^{b_l}\delta^{-1})(\delta g_{0}^{t_2 km}\delta^{-1}) \\
    &= \delta g_{0}^{\epsilon t_1lm +  t_2 km + \Delta(l)} \delta^{-1}\\
    &= (\delta g_{0}^{\epsilon t_1 lm + t_2 km}\delta^{-1})( \delta g_{0}^{\Delta(l)}
    \delta^{-1}) \\
    &= (\delta \theta(lme_1+kme_2)\delta^{-1}) ( \delta g_{0}^{\Delta(l)}
    \delta^{-1}).
\end{align*}
Now, since $|\Delta(l)| \leq 2K$, we see that $ C + | \delta g_{0}^{\Delta(l)} \delta^{-1}|$ is uniformly bounded by a constant (say) $C'$ which is independent of $l$ or $k$. It follows then that for all $ v \in m\Z^2$, there exist $ h'_i \in G$ with $ |h'_i|\leq C'$ such that
\[ c(\tilde{z}, v)= h_1'\theta(v)h_2'\]
as desired. From here we can proceed as in step 1 to prove the statement for all $v\in \Z^2$ and $x\in X$.

If $ d >2$, then we proceed by induction. 
 Assume without loss of generality that $t_1 \neq 0$. If $ t_i = 0$ for all $ i \neq 1$, then the conclusion of the theorem follows easily from Step 1. So assume that $ t_i \neq 0$ for some $ i \neq 1$. 
 Let $W_i$ be the $(d-1)$-dimensional hyperplane in $ \Z^d$ spanned by coordinate vectors $ \{ e_j~:~ j \neq i\}. $ 
 
 The restriction of $ X$ to $ W_i$ is also a strongly irreducible subshift.
By induction, there exists a homomorphism $ \phi: W_i \rightarrow G$ and a constant $ C_i>0$ such that for all $ w \in W_i$ and $ x \in X$, 
\[ c(x, w) = h_1 \phi(w)h_2 \]
where $ |h_1|, |h_2| \leq C_i$.

Let $c_i = \min\{ |\phi(e_j)|: j \neq i \text{ with }  \phi(e_j)\neq 0\}$. Choose $ m$ greater than $ \max \{ N_i, k,  \frac{4}{|c_i| \tau(g_0)}(4C_i + \max \{2A(\delta g_0 \delta^{-1}): |\delta| \leq 2C_i \})$.
Let $ y \in X$ be such that $ y|_{\Z{e_i}}$ is a periodic point of period $m$. Then by the strong irreducibility of $ X$, there exists $ \tilde{z} \in X$ such that for all $ w \in mW_i$ and $ n \in \Z$, 
\[ \tilde{z}_{w+ne_i} = y_{ne_i}.\]
The conclusion now follows by arguing as in the $ d=2$ case. 
 \end{proof}

The function $c_{rib}$ on the space $ X(ribbon, n)$ of ribbon tilings (\Cref{ribbon tiling cocycle}) is a cocycle with respect to a proper subgroup of $ \Z^2$ taking values in $ \Z^n$. However, note that $\Z^n$ is not a CIC group for $n \geq 2$. Hence, we will need the following corollary in \Cref{subsection: NO SI in ribbon tilings}. 
\begin{corollary}\label{ribbon tiling slope}
   Let $ X \subset \A^{\Z^d}$ be a strongly irreducible subshift for $ d \geq 2$. Let $G = G_1\times G_2\times\cdots\times G_r$, where each $G_i$ is a hyperbolic group with the CIC property and $ c: X \times \Z^d \rightarrow G$ is a cocycle with respect to the subgroup $W \leq \Z^d$ of finite index. Then there exists a homomorphism $ \theta: W \rightarrow G$ and a constant $ C$ such that for all $ x \in X\text{ and } w \in W$, there exist $h_1= h_1(x, w)$ and $h_2 = h_2(x, w)$ with $ \max(|h_1|, |h_2|) \leq C$ such that
   $$c(x, w ) = h_1 \theta(w) h_2.$$
\end{corollary}
\begin{proof}
    Considering the natural projection maps from $G$ to $ G_i$, we can assume without loss of generality that $r=1$, i.e.  $G$ is a hyperbolic group satisfying the CIC property. 
    
    Let $ \{\tilde{e}_1, \tilde{e}_2, \cdots, \tilde{e}_d \} $ be a basis of $W$. Then there exists unique isomorphism $ \Psi: \Z^d \rightarrow W$ such that $ \Psi(e_i)= \tilde{e}_1$. Consider the map $ c': X \times \Z^d \rightarrow G$ defined by 
    \[ c'(x, v) := c(x, \Psi(v)).\]
    Since, $c$ is a cocycle with respect to the $W$, $c'$ satisfies the equation:
    $$c'(x, v+w)= c'(x,v)c'(\sigma^{\Psi(v)}(x), w).$$
    Thus, it follows that $c'$ is a cocycle for a different subshift: Consider the fundamental domain $D$ for $\Z^d/W$ which is the parallelopiped in $ \Z^d$ with the endpoints in $\{ \sum_{i=1}^{d} \epsilon_i  \tilde{e}_i : \epsilon_{i} \in \{0,1\} \}$. Then the subshift
    $$X^W=\{y\in (A^D)^{\Z^d}~:~\text{ there exists $x\in X$ such that $x_{v+w}=(y_{\Psi^{-1}(v)})_w$ for $v\in W, w\in D$}\}$$
     is also strongly irreducible since $X$ is also strongly irreducible. The conclusion now follows easily from \Cref{triviality of cocycles}.   
\end{proof}
\begin{remark}
    In \Cref{triviality of cocycles}, we do not provide any information about $ h_i(x,v)$ other than the fact that they are uniformly bounded. A cocycle $c$ on $ X$ is said to be cohomologous to a homomorphism (or trivial) if there exists a continuous function $ h: X \rightarrow G$ and a homomorphism $\theta: \Z^d \rightarrow G $ such that $ c(x, w) = h(x) \theta(w) h(\sigma^{w}(x))^{-1}$ for all $ x \in X \text{ and } w \in \Z^d$. It is unclear to us whether continuous cocycles on a strongly irreducible subshift taking values in a hyperbolic group are trivial.
\end{remark}

\section{Non existence of strongly irreducible subshifts in certain subshifts}\label{NO SI SUBSHIFTS}

We will now apply \Cref{triviality of cocycles} to prove that there are no strongly irreducible subshifts (and therefore no shift-invariant finitely dependent processes) in various examples.

The following lemma says that if $ X$ is a topologically mixing $ \Z^d$ subshift, then any cocycle on $ X$ with some `structure' must be unbounded.

\begin{lemma}\label{mixing implies infinite}
    Let $ X$ be a topologically mixing $ \Z^d$ subshift and $ c: X \times \Z^d \rightarrow G$ be a cocycle, where $G$ is any finitely generated group. Suppose that there is an infinite subset $C\subset \Z^d$ such that for every $ x \in X$ and $ v \in C$, $ c(x, v)$ is an infinite order element. Then $\{ c(x, v) : x \in X, v \in \Z^d\}$ is an unbounded subset of $ G$.
\end{lemma}
\begin{proof}
For the sake of contradiction, assume that the cocycle is bounded. Then $ S=\{c(x, v)~:~ x \in X, v\in \Z^d\}$ is a bounded subset of $G$. For every $x\in X$, let $S_x=\{c(x, v)~:~v\in \Z^d\}$. Then the subsets $ S_x \subset S$ are partially ordered by inclusion. Choose $ x \in X$ such that $ S_x$ is maximal. Consider a finite box $[-n,n]^d\subset \Z^d$ such that
$$\{c(x, v)~:~v\in [-n,n]^d\}=S_x.$$  
Since the subshift $X$ is topologically mixing, we see that there exist $ v \in C$ and $ z \in X$ such that
$$z|_{[-n,n]^d}=z|_{v+[-n,n]^d}=x|_{[-n,n]^d}.$$
Since $c$ is a cocycle we have $S_x\cup (c(z,v)\cdot S_x)\subset S_z$. Then we have that $c(z, v)\cdot S_x = S_x $ due to maximality of $ S_x$. Let $ g= c(z, v)$. Since $ e = c(x, \textbf{0}) \in S_x$, $ g \cdot e= g$ lies in $ S_x$. But then $ g\cdot g= g^2$ also lies in $S_x$. Repeating the same argument we have that $g^n\in S_x$ for all $n\in \N$. However, since $v$ lies in $ C$, $ \{ g^n, n \in \N\}$ is an unbounded subset of $S_x$, which contradicts the boundedness of $ S_x$.   
\end{proof}

The cocycle in the next two examples takes values in the free product of cyclic groups. Let $ G = C_1*C_2*\cdots *C_n $ be a free product of groups $C_i$. The groups $C_i$ are called free factors of $G$. Every element $ g \neq e$ in $G$ can be written (\cite[Theorem 4.1]{Combinatorialgrouptheory}) in the form $ g= g_1 g_2\cdots g_r$, where $ g_i \neq e$ for each $i$, each $ g_i$ belongs to one of the free factors of $G$ and $ g_i, g_{i+1}$ are not in the same free factor of $G$ for any $i$. This representation of $g$ is called a reduced representation. The element $ g \neq e$ is called \textbf{cyclically reduced} if either $r=1$ or $ g_1$ and $g_r$ belong to different free factors of $G$. Every element in $ G$ is conjugate to a cyclically reduced element (\cite[Theorem 4.2]{Combinatorialgrouptheory}).

\subsection{Homshifts into four cycle free graphs}\hspace*{\fill} \\
Let $ H= (V, E)$ be any finite undirected graph. We assume further that $H$ is a four cycle free simple graph (See \Cref{Hom cocycle}). Let $ X \subset \textit{Hom}(\Z^d, H)$ be any strongly irreducible subshift for $ d \geq 2$. The cocycle $c_{H}$ introduced in Example \ref{Hom cocycle} restricts to a cocycle on $X$ which we also denote by $c_{H}$. Then for $1 \leq i \leq d$, we have
\[ |c_{H}(x, ne_i)| \leq |n|\]
for every $ x \in X$ and $ n \in \Z$. Therefore,
\[ \Phi(e_i) = \lim_{n \rightarrow + \infty} \frac{|c_{H}(x, ne_i)|}{n} \leq 1.\]
It follows from the ideas in \cite{entropyminimality} that equality cannot be achieved. Here we sketch a proof for completeness.

The following lemma follows from the proof of \cite[Lemma 8.1]{entropyminimality} and \Cref{remark: Hom cocycle is same as universal cover}.
\begin{lemma}\label{lemma: rigidity given slope 1}
    Let $ H$ be a four cycle free simple graph. Let $ X \subset \textit{Hom}(\Z^d, H)$ be a subshift where $ d \geq 2$. Suppose that for some $1 \leq i \leq d$, $|c_{H}(x, re_i)|=r$ for all $ x \in X$ and $ r \in \N$. Then for any $k \in \N$, any two allowed patterns $ a, b \in L(X, \B(k))$ which agree on the boundary of $ \B(k)$ (i.e. $ a|_{\partial \B(k)} = b|_{\partial \B(k)})$ must agree on entire $ \B(k)$ (i.e. $ a=b$).
\end{lemma}
 A subshift $ X$ is called \textbf{frozen} if for any two $ x, y \in X$ for which the set $ \{ v \in \Z^d : x_v \neq y_v\}$ is finite, we have $x=y$. A subshift is called trivial if it consists of the constant configuration $ \{ a\}^{\Z^d}$. Clearly, any non-trivial strongly irreducible subshift $X$ is not frozen.
\begin{lemma}\label{slope}
    Let $ H$ be a four cycle free simple graph. Let $ X \subset \textit{Hom}(\Z^d, H)$ be a strongly irreducible subshift where $ d \geq 2$. Then $ \Phi(e_i) <1$ for every $ 1 \leq i \leq d$.
\end{lemma}
\begin{proof}
 Since $ H$ does not contain any self-loops, $X$ cannot be a trivial subshift. Suppose that $ \Phi(e_i) = 1$ for some $ i$. Then we will show that for every $x \in X$ and for all $ r \in \N$,
 \begin{equation}\label{rigid configuration}
     |c_{H}(x, re_i)| = r.
 \end{equation} 
 If not, then there exists $ x_0 \in X$ such that $ |c_{H}(x_0, re_i)| < r$. Let $ k$ be the strong irreducibility constant of $X$. Then due to the strong irreducibility of $ X$, there exists $ y_0 \in X$ such that for every $ l \in \Z$, 
\[ y_0|_{[(2lk)e_i, (2lk+r)e_i]} = x_{0}|_{[\textbf{0}, re_i]}\]
where $ [me_i, ne_i]$ denotes the segment in the direction $ e_i$ joining the points $ me_i$ and $ne_i$. 

It follows that 
\[ \Phi(e_i) = \lim_{n \rightarrow +\infty} \frac{|c_{H}(y_0, ne_i)|}{n} <1\]
contrary to our assumption.

Now it follows from \Cref{rigid configuration} and \Cref{lemma: rigidity given slope 1} that for any $k \in \N$, if $a$ and $ b$ are two allowed patterns in $ L(X, \B(k))$ with $ a|_{\partial \B(k)} = b|_{\partial \B(k)}$, then $ a =b$. This implies that $ X$ is a frozen subshift which is a contradiction.     \end{proof}

As $ H$ is four cycle free, the group $ \mathcal{G}_{H} = F(E)$ (see Example \ref{Hom cocycle}) is a free group. Hence, if $ g$ is any cyclically reduced element in $ \mathcal{G}_{H}$, then $ |g^n| = n |g|$. Hence, the translation length (see \Cref{eq: translation length}) of $ g$ is, $ \tau(g) = |g| \in \N$. Since every non identity element is conjugate to a cyclically reduced element and $\tau(hgh^{-1}) = \tau(g)$ for any $ h \in G_{H}$, $ \tau(g) \in \N$ for every $ g \neq e$. 
\NOSIINHOM*
\begin{proof}
    Suppose for the contradiction that $ X \subset \textit{Hom}(\Z^d, H)$ is a strongly irreducible subshift. Since $ \mathcal{G}_{H}$ is a free group, its Cayley graph is a tree. Therefore, if $ v \in \Z^d$ is odd (i.e. $ v \notin 2\Z^d$), then $ c(x, v) \neq e$. Since every non-identity element in $ \mathcal{G}_{H}$ is of infinite order, Lemma \ref{mixing implies infinite} implies that $\{c_{H}(x, v) : x \in X, v \in \Z^d \}$ is an unbounded subset of $ G$. But then for some $ 1 \leq i \leq d$, $ \{c_{H}(x, ne_{i}) : x \in X, n \in \Z \}$ is unbounded.

    Since $ \mathcal{G}_{H}$ is a free group, Theorem \ref{triviality of cocycles} implies that there is a primitive element $ g_0 \in \mathcal{G}_{H}$, $ t \in \Z$ and a constant $ C$ such that for all $ x \in X$ and $ n \in \Z$, there exist $ h_1 = h_1(x, ne_i)$ and $ h_2 = h_2(x, ne_i)$ with $\max({|h_1|, |h_2|})\leq C$ satisfying,
    \[ c_{H}(x, ne_i)= h_1 g_{0}^{t n} h_2.\] 
     As $ \{c_{H}(x, ne_i): n \in 
    \Z\}$ is unbounded in $ \mathcal{G_{H}}$ for every $ x \in X$, we see that $ t\neq 0$ and $ \tau(g_0) \in \N$.

    Thus, the slope in direction $ e_i$ is, 
    \[ \Phi(e_i)= \lim_{n \rightarrow \infty} \frac{|c_{H}(x, ne_i)|}{|n|} = |t|\tau(g_0).\]
    Since $ \tau(g_0) \in \N$, $\Phi(e_i)$ is a positive integer which contradicts Lemma \ref{slope}.    
\end{proof}

\begin{corollary}
    If $ H$ is a finite four-cycle free simple graph, then for $d \geq 2$, $\textit{Hom}(\Z^d, H)$ does not support any shift-invariant finitely dependent processes.
\end{corollary}
\subsection{Space of tilings by boxes}\hspace*{\fill} \\
Let $\R = \{ R_1, R_2\}$ where $ R_1$ is a $ m \times 1$ box and $ R_2$ is a $ 1 \times n$ box in the plane where $ m$ and $n$ are integers greater than 1. Then cocycle $ c_{\R} $ defined on the tiling space $X(\R)$ (see Example \ref{tiling cocycle}) takes values in the group $ G = \Z/m\Z * \Z/n\Z = \langle h, v| h^m, v^n \rangle $. Let $ d(\cdot, \cdot)$ denote the distance in the Cayley graph of $ G$ with respect to generators $h$ and $v$. We will use yet another metric $D(\cdot, \cdot)$ on $ G$ for convenience. This metric was used in \cite{Remila} to give necessary and sufficient conditions for tileability of regions by boxes in $ \R$. Every element $ g \in G$ can be represented as $ g = x_1^{a_1}x_2^{a_2}\cdots x_{k}^{a_k}$ where each $ x_i $ is either $ h$ or $ v$ and $ x_i \neq x_{i+1}$ for all $ 1 \leq i \leq k-1$ . Such a representation is reduced if $ a_i \notin {m\Z}$ whenever $ x_i =h$ and $ a_j \notin n \Z$ whenever $x_j = v$. We define $ D(g,h):=k$ as the number of factors in the reduced representation of $gh^{-1}$. It is easy to see that $D$ is a metric.
\begin{lemma}\label{tiling slope}
    For all $ x \in X$, 
    \begin{equation}\label{bound on slope}  
    \limsup_{r \rightarrow \infty} \frac{D(c_{\R}(x, re_i), e)}{r} \leq 2
    \end{equation} 
    for $ i =1,2$.
\end{lemma}
\begin{proof}
For every $x\in X$ we know that $c_{\mathcal R}(x, e_1)=v^jhv^{-j}$. Thus it follows that
$$c_{\mathcal R}(x, ne_1)=\prod_{i=0}^{n-1}c_{\mathcal R}(\sigma^{ie_1}(x),e_1)=v^{j_1}hv^{j_2}\ldots hv^{j_{n+1}}$$
for some $j_1, j_2, \ldots , j_n$. This implies $D(c_{\R}(x, re_1), e) \leq 2r+1$ for all natural number $r$ which implies the claim. The analogous statement in $e_2$ direction follows similarly. \end{proof}

\begin{theorem}
   Let $m,n>1$. $\R= \{R_1, R_2\}$ where $R_1$ and $R_2$ are boxes of dimensions $m\times 1$ and $1\times n$ respectively. There is no strongly irreducible subshift contained in $X(\R)$.
\end{theorem}
For the following proof recall the notation in \Cref{tiling cocycle}.
\begin{proof}
    Suppose, on the contrary, that $ X \subset X(\R)$ is a strongly irreducible subshift. Let $ k$ be the SI distance of $X$. 

    Since every $g \in G$ is conjugate to a cyclically reduced element, we see that $g$ has finite order in $G$ if and only if $ g$ is a conjugate of $ h^i$ for some $i$ or $v^j$ for some $j$. Let $w = (w_1, w_2)\in \Z^2$ be such that $ w_1 \neq 0 \Mod{m}$ and $ w_2 \neq 0 \Mod{n}$. Then for every $x \in X$, we have $ \pi(c_{\R}(x, w)) = (a \Mod{m}, b \Mod{n})$ where $ a \neq 0 \Mod{m}$ and $ b \neq 0 \Mod{n}$. Thus, $ c_{\R}(x,w)$ has infinite order in $G$ for all $x \in X$. Therefore, by Lemma \ref{mixing implies infinite}, $\{ c_{\R}(x, w'): x \in X, w' \in \Z^2\}$ is an unbounded subset of $G$. Without loss of generality, assume that $ \{ c_{\R}(x, ne_1) : n \in \Z\}$ is unbounded for some (and hence all) $ x$ in $X$.
    
     The subshift $ X|_{\Z e_1}$ is strongly irreducible. Let $ N_0 \geq k$ be as guaranteed by Proposition \ref{proposition: periodic in SI}. Choose $ p \geq N_0$ such that $ p = 1 \Mod{n}$. Then there exists $y \in X $ such that $ y|_{\Z e_1}$ is a periodic point of period $ p$. As $ p \geq k$, there exists $ z \in X$ such that $ z|_{n e_1} = y|_{ne_1}$ and $ z|_{pe_2 + ne_1} = y|_{ ne_1}$. By the proof of \Cref{triviality of cocycles} (more specifically look at \Cref{equation:conjugation_when_periodic}), there exists a primitive element $g \in G$, $ t \in \Z$ and $ \alpha \in G$ such that
     \[ c_{\R}(z, pe_1) = \alpha g^{tp} \alpha^{-1}.\]
     Now, there exists a cyclically reduced element $g_0 \in G$ such that $ g = \beta g_0 \beta^{-1}$ for some $ \beta \in G$. Therefore,
     \[ c_{\R}(z, pe_1) = (\alpha \beta) g_{0}^{tp} (\beta\alpha)^{-1}.\]
     Let the modulo class of $ g_{0}^{t} $ be $ \pi(g_{0}^{t})= (a \Mod{m}, b \Mod{n})$. Then since $ p = 1 \Mod{n}$,
     \[ (p \Mod{m}, 0\Mod{n}) = \pi(c_{\R}(z, pe_1))= \pi((g_{0}^t)^p)= (pa \Mod{m}, pb \Mod{n}) = (pa \Mod{m}, b \Mod{n}).\]
     Therefore $ b= 0 \Mod{n}$. Furthermore, as $\{ g_{0}^{tn}: n \in \Z\} $ is unbounded, $ g_{0}^{t}$ is not equal to $ h^i$ or $ v^j$ for any $ i, j$. This implies that the minimal representation of $ g_{0}^{t}$ has at least three factors. Hence, $ D(g_{0}^{t}, e) \geq 3$. But since $g_{0}$ is cyclically reduced, $D(g_{0}^{tr},e) = r D(g_{0}^{t},e).$ Thereby,
     \[ \limsup_{r \rightarrow +\infty} \frac{D(c_{\R}(z, rp),e)}{rp} \geq 3\]
     which contradicts \Cref{tiling slope}. 
\end{proof}
\NOSIINTILINGS*

\begin{proof}
    Let $ \R' = \{R_1', R_2'\}$ where $ R_1'$ is a $ m \times 1$ box and $ R_2'$ is a $ 1 \times n$ box. Then each box in $\R$ is tileable by tiles in $ \R'$. Hence, the tiling space $ X(R)$ is a subshift of $ X(\R')$. But by Theorem \ref{theorem: no_SI_rectangular_tiling}, there are no strongly irreducible subshifts in $ X(\R')$. \end{proof}
  
\begin{corollary}\label{corollary: NO FINDEP on tilings}
     Let $\R$ be a set of boxes which satisfies the assumptions of Theorem \ref{theorem: no_SI_rectangular_tiling}. Then there is no shift-invariant finitely dependent process supported on $X(\R)$.
\end{corollary}

\begin{corollary}
     Let $\R$ be a set of boxes which satisfies assumptions of Theorem \ref{theorem: no_SI_rectangular_tiling}. Then there is no continuous tiling of $ F(2^{\Z^2})$ by $ X(\mathcal{R})$.
\end{corollary}
\begin{proof}
    This follows from \Cref{continuous imples findep} and \Cref{corollary: NO FINDEP on tilings}.
\end{proof}
\subsection{Ribbon tilings}\label{subsection: NO SI in ribbon tilings}\hspace*{\fill} \\
If $ n=1$, then the space $ X(ribbon, 1) = \{ \square\}$ is a singleton set consisting of the tiling of $\Z^2$ by the unit square and is trivially strongly irreducible. However, we prove that this is not the case for $ n \geq 2$.  
\NOSIINRIBBON*
\begin{proof}
    If $n=2$, then $ X(ribbon, n)$ is actually the space of domino tilings. Hence, from Theorem \ref{theorem: no_SI_rectangular_tiling} we know that there are no strongly irreducible subshifts in $ X(ribbon,2)$. Therefore, we assume that $ n \geq 3$.

 Suppose, on the contrary, that $ X(ribbon, n) $ contains a strongly irreducible subshift $X$. Let $v_0 = (1, -1) \in \Z^2$. Then $ X|_{\Z v_0}$ is a strongly irreducible subshift. Let $ N_0$ be as guaranteed by Proposition \ref{proposition: periodic in SI} and $ p \geq N_0$ be an odd integer. Then there exists $ y \in X$, such that $ y|_{\Z v_0}$ is a periodic point of period $ p$. Let $ c_{rib}(y, pv_0) = (a_0, a_1, \cdots, a_{n-1}).$

    Consider the subgroup $ W = \{ (i_1, i_2) \in \Z^2: i_1+i_2 = 0 \Mod{n}\} \leq \Z^2$. Recall that $ c_{rib}$ is a cocycle with respect to $W$. Hence, by Corollary \ref{ribbon tiling slope}, there exists a homomorphism $ \theta: W \rightarrow \Z^n $ such that for all $ x \in X$ and $ w \in W$,
    \begin{equation}\label{fh}
       ||c_{rib}(x, w) - \theta(w)||_{1}\leq C  
    \end{equation}
    where $ C$ is independent of $ x$ and $w$. Let $ \theta(v_0) = (r_0, r_1, \cdots, r_{n-1})$.
    
    Note that
    \begin{align*}
        \overline{\Phi}(v_0) :=\lim_{k \rightarrow +\infty} \frac{c_{rib}(x, kv_0)}{k}
    \end{align*}
    is independent of $x \in X$ (see \Cref{cocycles for subgroups} and discussion following it). Since $c_{rib}$ is a cocycle with respect to $W$ and $n v_0 \in W$ for all $ n \in \Z$, we see that
    \begin{align*}
        \overline{\Phi}(v_0) :=\lim_{k \rightarrow +\infty} \frac{c_{rib}(y, kpv_0)}{kp} =\lim_{k \rightarrow +\infty} \frac{kc_{rib}(y, pv_0)}{kp}= \frac{1}{p}(a_0, a_1, \cdots, a_{n-1}).
    \end{align*}
     But \Cref{fh} implies $ \overline{\Phi}(v_0) = \theta(v_0) = (r_0, r_1, \cdots, r_{n-1})$. Thus, we see that $ a_j = r_j p$ for every $ 0 \leq j \leq n-1.$ 

    By \Cref{equation:diagonal_descent},
    $$ c_{rib}(x, u+v_0) - c_{rib}(x, u) = 2 \mathbf{e}_{j, j+1}$$ for some $j$ which depends on $x$ and $ u$. It follows that all the coordinates $ a_j$ of $ c_{rib}(y, p v_0)$ are non-negative even numbers. But since $a_j = r_jp$ and $ p$ is odd, this implies that $ r_j$ is a non-negative even number for each $ j $, $ 0 \leq j \leq n-1$. But by Lemma \ref{ribbonlemma},
    \[ \sum_{j=0}^{n-1} r_j p = \sum_{j=0}^{n-1} a_j = \sum_{j=0}^{n-1} P_j(c_{rib}(y, pv_0)) = p-(-p)= 2p. \]
    Therefore, $ \sum_{j=0}^{n-1} r_j = 2$.  Thus we must have $ r_j= 2$ for some $j$ and $ r_l = 0$ for all $ l \neq j$.
    
     We will show that this implies for all $x \in X$ that 
    $$ P_j(c_{rib}(x, v_0)) = 2.$$ 
    Note that $ P_j(c_{rib}(x, v_0)) \leq 2$ for all $ x \in X$. Assume that there exists $ x_0 \in X$ such that $ P_j(c_{rib}(x_0, v_0)) < 2$. The value $ c_{rib}(x_0, v_0)$ can be deduced by looking at the finite pattern $ x_0|_{\B(2n)}$. Let $ k$ be the strong irreducibility constant in $ X$. Let $ M = 2k+4n$. Then, by the strong irreducibility of $X$, there exists $ y_0 \in X$ such that the restriction of $ y_0$ to the subset $ Mlv_0 + \B(2n)$ is equal to $x_0|_{\B(2n)}$ for every $ l \in \Z$. But the subadditivity of the function $ t \rightarrow P_j(c_{rib}(x, tv_0))$, implies that
    \[ r_j =P_j(\theta(v_0)) = \lim_{t \rightarrow + \infty} \frac{P_j(c_{rib}(x_0, tv_0))}{t} <2\]
    which is a contradiction.

    Consider the path $ p'= (\mathbf0, e_1, v_0)$ in  $\Z^2$ where $ \mathbf0$ denotes origin. Note that the edges $ (\mathbf0,e_1)$ and $ (e_1, v_0)$ are of type $ (n-1, 0)$ where the square of the color $0$ is on the left. We will use \Cref{equation:diagonal_descent} in the following. We have two cases to consider.
    
    Case 1) $j =n-1$: Choose any $ x \in X$. Since $P_j(c_{rib}(x, v_0)) =2$, the edges $ (\mathbf0, e_1)$ and $ (e_1, v_0)$ must lie on the tiling $x$. But since $\sigma^{v}(x) \in X$ as well, we see that the edges $(v,v+e_1) $ and $ (v+e_1,v +v_0)$ also lie on the tiling $ x$ for every $ v \in \Z^2$, that is, every edge in $\Z^2$ lies on a tiling $x$. This is a contradiction, since $ n \geq 2$ and some edge in $\Z^2$ must cross some tile in $x$.

    Case 2) $ j \neq n-1$: Choose any $ x \in X$. Then $ P_j(c_{rib}(x, v_0))=2$ implies that one of the edges $ (\mathbf0, e_1) $ and $(e_1, v_0)$ must cross a tile of type $(j , j+1)$ in the tiling $x$ and the other edge lies on the tiling $x$. But since $\sigma^{v}(x) \in X$ as well, we see that for every $ v \in \Z^2$, one of the edges $ (v, v+e_1)$ and $(v+e_1, v+v_0)$ must cross the tiling $ x$. This is a contradiction, since for every ribbon tiling $x$, there exists some $v$ (namely a point near a corner of any of the ribbon tiles) for which both the edges $ (v, v+e_1)$ and $ (v+e_1, v+v_0)$ lie on the tiling $x$. \end{proof}

    \begin{corollary}\label{No FINDEP on ribbon tilings}
        For $ n \geq 2$, there is no shift-invariant finitely dependent process supported on $ X(ribbon, n)$.
    \end{corollary}

\section{Open questions}\label{Section: open directions}
We list some open questions below. 
\subsection{Finitely dependent processes and mixing properties}\label{Question: Finitely dependent processes and mixing properties}\hspace*{\fill} \\
In \Cref{onedimension}, we proved that shift-invariant finitely dependent processes are dense for every strongly irreducible $\Z$-subshift. However, we do not know if every strongly irreducible $\Z$-subshift supports a shift-invariant finitely dependent process with full support. 

For example, for $S =\{n_1, n_2, \cdots,\} \subset \N \cup \{ 0\}$ where $ n_i < n_{i+1}$, the corresponding $S$-gap subshift is 
$$ X(S) : = \{x \in \{0,1\}^\Z~:~ \text{ the gap between consecutive $1$'s lies in $S$}\}.$$
Then $X(S)$ is not sofic if the sequence $(n_{i+1} - n_i)_{i \in \N}$ is not eventually periodic (\cite{LindMarcus}) and $X(S)$ is strongly irreducible if and only if $ \sup\{|n_{i+1}-n_i|: i \in \N\} < \infty$ and $ gcd\{s+1: s \in S\} =1$ (\cite{S-gap}). We do not know if there is a fully supported shift-invariant finitely dependent process on such $X(S)$.

\stronglyirreducible*

\subsection{Finitely dependent processes on homshifts}\hspace*{\fill} \\
 It follows from \cite[Lemma 5.2]{Dismantlable} that if $ H$ is a finite dismantleable graph then $ Hom(\Z^d, H)$ has the finite extension property and therefore supports shift-invariant finitely dependent processes by \Cref{FEPcase}. We have proved here that if $ H$ is a four cycle free simple graph then $ Hom(\Z^d, H)$ does not support any shift-invariant finitely dependent processes.
 \begin{question}
     \begin{enumerate}
         \item Is there a natural characterization of graphs $H$ for which there exist shift-invariant finitely dependent processes supported on \textit{Hom}($ \Z^d, H$)?
         \item Is it decidable whether given a graph $H$ if there are shift-invariant finitely dependent processes on \textit{Hom}$(\Z^d, H)$? 
     \end{enumerate}
 \end{question}
 We note that the problem of determining the graphs for which there exists a continuous coloring of $ F(2^{\Z^d})$ by the subshift \textit{Hom}$(\Z^d, H)$ is undecidable (\cite[Theorem 4.4.1]{gao2023continuouscombinatoricsabeliangroup}).
\subsection{Mixing properties of tiling spaces}\hspace*{\fill} \\
\label{section: Mixing properties of tiling spaces}The study of tiling spaces (even by objects as simple as boxes) of $\mathbb R^d$ for $d\geq 3$ is extremely complicated.  
\begin{question}
    \begin{enumerate}
\item A $d-$dimensional domino is a $ d$-dimensional box whose one side has length $2$ and all other sides have length 1 with its vertices in $\Z^d$. Does there exist a strongly irreducible subshift contained in the space of domino tilings of $\mathbb{R}^d$ for $d\geq 3$?
    \item In general, is there a natural characterization of the set of boxes $\R$ for which the space of tilings of $\mathbb{R}^d$ by the boxes in $\R$ contains a strongly irreducible subshift?
    
    We know that if the collection $ \R$ satisfies the condition of \Cref{thm:eventuallyuniversaltiles}, then there exists a shift-invariant finitely dependent process on the tiling space $X(\R)$. However, if $ p_1, p_2, \cdots, p_d$ are prime numbers and $ \R = \{ R_1, R_2, \cdots, R_d\}$ is a collection of bars where $ R_i$ is a bar whose sidelength in $ e_i$ direction is $ p_i$, then does $ X(\R)$ contain a strongly irreducible subshift?
    \end{enumerate}
\end{question}

\subsection{Finitely dependent processes over general groups}\hspace*{\fill} \\
 Let $ G$ be a countable amenable group. Then one can directly use the Rokhlin lemma for countable amenable groups (\cite{RokhlinLemmaAmenableGroups}) to show that the block factors of iid processes are weak-\textasteriskcentered{} dense in the shift-invariant measures on the full shift $ \{ 0,1\}^{G}$ (or any nearest neighbor subshift of finite type with a ``safe symbol" \cite[Definition 3.7]{TSSM}). Briefly, given sufficiently invariant Rokhlin towers for a non-trivial iid measure $ \mu $ on $ A^G$ for some finite set $ A$, we can arbitrarily approximate the bases of the towers by finite disjoint union of cylinder sets. We paste patterns from the measure we want to approximate on the towers based on the cylinder sets while coloring the error set (in the Rokhlin lemma and in the approximation by cylinder sets) with $ 0$'s (or safe symbols). We don't know the answer to the following question.
 \begin{question}
     For a countable amenable group $G$, does any $G$-subshift which satisfies the finite extension property support a dense set of shift-invariant finitely dependent processes? Is there a fully supported one?
 \end{question}
 The question becomes even more interesting when $G$ is replaced by the free group, where we do not know the answer for the density question even for the full shift. We heard the following question from Yinon Spinka which is a special case of Question 4 in \cite{Symmetrization}.
 \begin{question}
     Is there a finitely dependent process on the set of $m$-nets on $\mathbb R$?
 \end{question}

 \subsection{Finitely dependent Markov random fields over subshifts}\hspace*{\fill} \\
 It was shown by Holroyd and Liggett (\cite{Holroyd}) that no shift-invariant (or stationary) finitely dependent process supported on proper $q$-colorings of $ \Z$ can be a stationary Markov chain. Note that shift-invariant finite-valued Markov random fields are Markov chains (\cite{MarkovRandomFieldMarkovChains}). This leaves the question for higher-dimensions open.
 \begin{question}
     \begin{enumerate}
     \item Are there shift-invariant finitely dependent Markov random fields supported on proper $q$-colorings of $ \Z^d$ for $d \geq 2$ and $q \geq 4$?
     \item Are there shift-invariant finitely dependent Gibbs measures corresponding to nearest neighbor interactions supported on proper $q$-colorings of $ \Z^d$ for $ d \geq 2$ and $q \geq 4$?
     \item Is there a characterization of graphs $H$ for which there exists a shift-invariant finitely dependent Markov random field/Gibbs measure on \textit{Hom}$(\Z^d, H)$? 
 \end{enumerate}
 \end{question}
 For the last question, the following example was mentioned to the first author by Alexander Holroyd. Let $H_1, H_2$ be graphs such that \textit{Hom}$(\Z, H_i)$ have shift-invariant finitely dependent processes on them. Let $H$ be the Cartesian product of the graphs given by vertices $(v,w)$ where $v,w$ are vertices of $H$ and $((v_1, w_1), (v_2, w_2))$ form an edge in $H$ if either $(v_1, v_2)$ forms an edge in $H_1$ and $w_1=w_2$ or $(w_1, w_2)$ forms an edge in $H_2$ and $v_1=v_2$. Then it follows that there is a shift-invariant finitely dependent Markov random field on \textit{Hom}$(\Z^2, H)$.

\subsection{Finitely dependent processes and periodic points}\hspace*{\fill} \\
An element $ x \in \A^{\Z^d}$ is called periodic if the subgroup $\{ v \in \Z^d: \sigma^{v}x=x\}$ has a finite index in $ \Z^d$. The twelve tiles theorem in \cite{gao2023continuouscombinatoricsabeliangroup} implies that if there is a continuous coloring of $F(2^{\Z^d})$ by a subshift of finite type $ X$, then $ X$ must contain periodic points.
\begin{question}
    If a subshift $ X \subset \A^{\Z^d}$ is the support of a shift-invariant finitely dependent process, then does $X$ contain periodic points? Are periodic points dense in $X$?
\end{question}

This is true in one dimension since any one dimensional strongly irreducible subshift has dense set of periodic points. Given that there are strongly irreducible $\Z^2$-subshifts without any periodic points (\cite{hochman}), an affirmative answer to this question will negatively answer Question \ref{question:stronglyirreducible}. 

\subsection{Finitely dependent processes invariant under the subactions}\hspace*{\fill} \\
We learned about this question from Deepak Dhar. While in this paper we focus on shift-invariant finitely dependent processes, the question of finitely dependent processes invariant under a sub-action is largely ignored. For instance, though there are no shift-invariant finitely dependent processes on the space of domino tilings, there are plenty of finitely dependent processes which are invariant under translations by the $2\Z\times 2\Z$. As an example, note that the delta measure on a periodic configuration invariant under $2\Z\times 2\Z$ is finitely dependent. It follows in fact from standard well-known results (see \cite{zbMATH01566263}) that the space of finitely dependent processes on domino tilings of $\Z^2$, invariant under full rank subgroups is dense in the space of $\Z^2$-ergodic probability measures. It is however unclear whether for any $k \in \N$, the 
finitely dependent processes on domino tilings invariant under the $k\Z\times k\Z$ sub-actions are dense in the space of ergodic measures.

\subsection{Structure of finitely dependent processes} \label{section:structure}\hspace*{\fill} \\
While we know ways to construct shift-invariant finitely dependent processes, we know very little about their structure. As discussed in the introduction, it was earlier conjectured and later refuted that they are block factors of iid processes. We suggest the following modification.
\begin{question}\label{question: structure of findep}
    Is every shift-invariant finitely dependent process a block factor of the product of such a process on $m$-nets for some $m$ and an iid process?
\end{question}
A question along these lines was asked by Tom Meyerovitch. Note that shift-invariant finitely dependent processes are finitary factors of iid processes \cite{zbMATH00177172,Yinonfd}.

\bibliographystyle{amsalpha}
\bibliography{mybibliography}

@article {Holroyd,
    AUTHOR = {Holroyd, Alexander E. and Liggett, Thomas M.},
     TITLE = {Finitely dependent coloring},
   JOURNAL = {Forum Math. Pi},
  FJOURNAL = {Forum of Mathematics. Pi},
    VOLUME = {4},
      YEAR = {2016},
     PAGES = {e9, 43},
      ISSN = {2050-5086},
       DOI = {10.1017/fmp.2016.7},
       URL = {https://doi.org/10.1017/fmp.2016.7},
}

@article {Symmetrization,
    AUTHOR = {Holroyd, Alexander E.},
     TITLE = {Symmetrization for finitely dependent colouring},
   JOURNAL = {Electron. Commun. Probab.},
  FJOURNAL = {Electronic Communications in Probability},
    VOLUME = {29},
      YEAR = {2024},
     PAGES = {Paper No. 44, 12},
      ISSN = {1083-589X},
   MRCLASS = {60C05 (05C15)},
  MRNUMBER = {4806804},
       DOI = {10.1214/24-ECP600},
       URL = {https://doi.org/10.1214/24-ECP600},
}

@misc{arXiv:2510.11969,
 author = {Chandgotia, Nishant and Gangloff, Silv{\`e}re and de Menibus, Benjamin Hellouin and Oprocha, Piotr},
 title = {On the cohomology of homshifts},
 year = {2025},
 howpublished = {Preprint, {arXiv}:2510.11969 [math.{DS}] (2025)},
 url = {https://arxiv.org/abs/2510.11969},
 arXiv = {arXiv:2510.11969}
}

@incollection{zbMATH01256699,
 author = {Kenyon, Claire and Kenyon, Richard},
 title = {Tiling a polygon with rectangles},
 booktitle = {33rd annual symposium on Foundations of computer science (FOCS). Proceedings, Pittsburgh, PA, USA, October 24--27, 1992},
 pages = {610--619},
 year = {1992},
 publisher = {Washington, DC: IEEE Computer Society Press},
 zbMATH = {1256699},
 Zbl = {0915.05039}
}

@article { FEP,
    AUTHOR = {Brice\~no, Raimundo and McGoff, Kevin and Pavlov, Ronnie},
     TITLE = {Factoring onto {$\mathbb{Z}^d$} subshifts with the finite
              extension property},
   JOURNAL = {Proc. Amer. Math. Soc.},
  FJOURNAL = {Proceedings of the American Mathematical Society},
    VOLUME = {146},
      YEAR = {2018},
    NUMBER = {12},
     PAGES = {5129--5140},
      ISSN = {0002-9939,1088-6826},
       DOI = {10.1090/proc/14267},
       URL = {https://doi.org/10.1090/proc/14267},
}

@article {Schramm,
    AUTHOR = {Holroyd, Alexander E. and Schramm, Oded and Wilson, David B.},
     TITLE = {Finitary coloring},
   JOURNAL = {Ann. Probab.},
  FJOURNAL = {The Annals of Probability},
    VOLUME = {45},
      YEAR = {2017},
    NUMBER = {5},
     PAGES = {2867--2898},
      ISSN = {0091-1798,2168-894X},
       DOI = {10.1214/16-AOP1127},
       URL = {https://doi.org/10.1214/16-AOP1127},
}

@article {Aaronson,
    AUTHOR = {Aaronson, Jon and Gilat, David and Keane, Michael and de Valk,
              Vincent},
     TITLE = {An algebraic construction of a class of one-dependent
              processes},
   JOURNAL = {Ann. Probab.},
  FJOURNAL = {The Annals of Probability},
    VOLUME = {17},
      YEAR = {1989},
    NUMBER = {1},
     PAGES = {128--143},
      ISSN = {0091-1798,2168-894X},
       URL =
              {http://links.jstor.org/sici?sici=0091-1798(198901)17:1<128:AACOAC>2.0.CO;2-6&origin=MSN},
}

@article {Yinonfd,
    AUTHOR = {Spinka, Yinon},
     TITLE = {Finitely dependent processes are finitary},
   JOURNAL = {Ann. Probab.},
  FJOURNAL = {The Annals of Probability},
    VOLUME = {48},
      YEAR = {2020},
    NUMBER = {4},
     PAGES = {2088--2117},
      ISSN = {0091-1798,2168-894X},
       DOI = {10.1214/19-AOP1417},
       URL = {https://doi.org/10.1214/19-AOP1417},
}

@article {Schmidt,
    AUTHOR = {Schmidt, Klaus},
     TITLE = {The cohomology of higher-dimensional shifts of finite type},
   JOURNAL = {Pacific J. Math.},
  FJOURNAL = {Pacific Journal of Mathematics},
    VOLUME = {170},
      YEAR = {1995},
    NUMBER = {1},
     PAGES = {237--269},
      ISSN = {0030-8730,1945-5844},
       URL = {http://projecteuclid.org/euclid.pjm/1102371116},
}

@article {Einsiedler,
    AUTHOR = {Einsiedler, Manfred},
     TITLE = {Fundamental cocycles of tiling spaces},
   JOURNAL = {Ergodic Theory Dynam. Systems},
  FJOURNAL = {Ergodic Theory and Dynamical Systems},
    VOLUME = {21},
      YEAR = {2001},
    NUMBER = {3},
     PAGES = {777--800},
      ISSN = {0143-3857,1469-4417},
       DOI = {10.1017/S0143385701001389},
       URL = {https://doi.org/10.1017/S0143385701001389},
}

@article {Mixing,
    AUTHOR = {Alon, Noga and Brice\~{n}o, Raimundo and Chandgotia, Nishant
              and Magazinov, Alexander and Spinka, Yinon},
     TITLE = {Mixing properties of colourings of the {$\mathbb{Z}^d$} lattice},
   JOURNAL = {Combin. Probab. Comput.},
  FJOURNAL = {Combinatorics, Probability and Computing},
    VOLUME = {30},
      YEAR = {2021},
    NUMBER = {3},
     PAGES = {360--373},
      ISSN = {0963-5483,1469-2163},
       DOI = {10.1017/s0963548320000395},
       URL = {https://doi.org/10.1017/s0963548320000395},
}

@article {NoSIsubshifts,
    AUTHOR = {Chandgotia, Nishant and Meyerovitch, Tom},
     TITLE = {Borel subsystems and ergodic universality for compact {$\mathbb
              Z^d$}-systems via specification and beyond},
   JOURNAL = {Proc. Lond. Math. Soc. (3)},
  FJOURNAL = {Proceedings of the London Mathematical Society. Third Series},
    VOLUME = {123},
      YEAR = {2021},
    NUMBER = {3},
     PAGES = {231--312},
      ISSN = {0024-6115,1460-244X},
       DOI = {10.1112/plms.12398},
       URL = {https://doi.org/10.1112/plms.12398},
}

@article {Burton,
    AUTHOR = {Burton, Robert M. and Goulet, Marc and Meester, Ronald},
     TITLE = {On {$1$}-dependent processes and {$k$}-block factors},
   JOURNAL = {Ann. Probab.},
  FJOURNAL = {The Annals of Probability},
    VOLUME = {21},
      YEAR = {1993},
    NUMBER = {4},
     PAGES = {2157--2168},
      ISSN = {0091-1798,2168-894X},
       URL =
              {http://links.jstor.org/sici?sici=0091-1798(199310)21:4<2157:O1PAF>2.0.CO;2-B&origin=MSN},
}

@article {Computerscience,
    AUTHOR = {Greb\'{\i}k, Jan and Rozho\v{n}, V\'{a}clav},
     TITLE = {Local problems on grids from the perspective of distributed
              algorithms, finitary factors, and descriptive combinatorics},
   JOURNAL = {Adv. Math.},
  FJOURNAL = {Advances in Mathematics},
    VOLUME = {431},
      YEAR = {2023},
     PAGES = {Paper No. 109241, 56},
      ISSN = {0001-8708,1090-2082},
       DOI = {10.1016/j.aim.2023.109241},
       URL = {https://doi.org/10.1016/j.aim.2023.109241},
}

@article {nonuniformspecification,
    AUTHOR = {Pavlov, Ronnie},
     TITLE = {On intrinsic ergodicity and weakenings of the specification
              property},
   JOURNAL = {Adv. Math.},
  FJOURNAL = {Advances in Mathematics},
    VOLUME = {295},
      YEAR = {2016},
     PAGES = {250--270},
      ISSN = {0001-8708,1090-2082},
       DOI = {10.1016/j.aim.2016.03.013},
       URL = {https://doi.org/10.1016/j.aim.2016.03.013},
}

@article {AaronsonMarkov,
    AUTHOR = {Aaronson, Jon and Gilat, David and Keane, Michael},
     TITLE = {On the structure of {$1$}-dependent {M}arkov chains},
   JOURNAL = {J. Theoret. Probab.},
  FJOURNAL = {Journal of Theoretical Probability},
    VOLUME = {5},
      YEAR = {1992},
    NUMBER = {3},
     PAGES = {545--561},
      ISSN = {0894-9840,1572-9230},
       DOI = {10.1007/BF01060435},
       URL = {https://doi.org/10.1007/BF01060435},
}

@article {Conway,
    AUTHOR = {Conway, John H. and Lagarias, Jeffrey C.},
     TITLE = {Tiling with polyominoes and combinatorial group theory},
   JOURNAL = {J. Combin. Theory Ser. A},
  FJOURNAL = {Journal of Combinatorial Theory. Series A},
    VOLUME = {53},
      YEAR = {1990},
    NUMBER = {2},
     PAGES = {183--208},
      ISSN = {0097-3165,1096-0899},

       DOI = {10.1016/0097-3165(90)90057-4},
       URL = {https://doi.org/10.1016/0097-3165(90)90057-4},
}

@book{Combinatorialgrouptheory,
 author = {Magnus, Wilhelm and Karrass, Abraham and Solitar, Donald},
 title = {Combinatorial group theory. {Presentations} of groups in terms of generators and relations.},
 edition = {Reprint of the 1976 second edition},
 isbn = {0-486-43830-9},
 year = {2004},
 publisher = {Mineola, NY: Dover Publications},
 zbMATH = {2146477},
 Zbl = {1130.20307}
}

@article{BLAND_2025,
   title={Homomorphisms from aperiodic subshifts to subshifts with the finite extension property},
   ISSN={1469-4417},
   url={http://dx.doi.org/10.1017/etds.2025.16},
   DOI={10.1017/etds.2025.16},
   journal={Ergodic Theory and Dynamical Systems},
   publisher={Cambridge University Press (CUP)},
   author={Bland, Robert and Mcgoff, Kevin},
   year={2025},
   month=apr, pages={1–17} }

@book{gao2023continuouscombinatoricsabeliangroup,
 author = {Gao, Su and Jackson, Steve and Krohne, Edward and Seward, Brandon},
 title = {Continuous combinatorics of abelian group actions},
 fseries = {Memoirs of the American Mathematical Society},
 series = {Mem. Am. Math. Soc.},
 issn = {0065-9266},
 volume = {1573},
 isbn = {978-1-4704-7479-9; 978-1-4704-8392-0},
 year = {2025},
 publisher = {Providence, RI: American Mathematical Society (AMS)},
 doi = {10.1090/memo/1573},
 zbMATH = {8075988}
}

@article{timar2024finitelydependentrandomcolorings,
 author = {Tim{\'a}r, {\'A}d{\'a}m},
 title = {Finitely dependent random colorings of bounded degree graphs},
 fjournal = {Annales de l'Institut Henri Poincar{\'e}. Probabilit{\'e}s et Statistiques},
 journal = {Ann. Inst. Henri Poincar{\'e}, Probab. Stat.},
 issn = {0246-0203},
 volume = {62},
 number = {1},
 pages = {747--750},
 year = {2026},
 language = {English},
 doi = {10.1214/26-AIHP1653},
 url = {projecteuclid.org/journals/annales-de-linstitut-henri-poincare-probabilites-et-statistiques/volume-62/issue-1/Finitely-dependent-random-colorings-of-bounded-degree-graphs/10.1214/26-AIHP1653.full},
 zbMATH = {8172793}
}

@misc{poirier2024contractiblesubshifts,
 author = {Leo Poirier and Ville Salo},
 title = {Contractible subshifts},
 year = {2024},
 howpublished = {Preprint, {arXiv}:2401.16774},
 url = {https://arxiv.org/abs/2401.16774},
 arXiv = {arXiv:2401.16774}
}

@book {LindMarcus,
    AUTHOR = {Lind, Douglas and Marcus, Brian},
     TITLE = {An introduction to symbolic dynamics and coding},
    SERIES = {Cambridge Mathematical Library},
   EDITION = {Second},
 PUBLISHER = {Cambridge University Press, Cambridge},
      YEAR = {2021},
     PAGES = {xix+550},
      ISBN = {978-1-108-82028-8},
       DOI = {10.1017/9781108899727},
       URL = {https://doi.org/10.1017/9781108899727},
}

@incollection {SyncWord,
    AUTHOR = {Bertrand, Anne},
     TITLE = {Specification, synchronisation, average length},
 BOOKTITLE = {Coding theory and applications ({C}achan, 1986)},
    SERIES = {Lecture Notes in Comput. Sci.},
    VOLUME = {311},
     PAGES = {86--95},
 PUBLISHER = {Springer, Berlin},
      YEAR = {1988},
      ISBN = {3-540-19368-5},
       DOI = {10.1007/3-540-19368-5\_9},
       URL = {https://doi.org/10.1007/3-540-19368-5_9},
}

@misc{Onlinelocality,
      title={Online Locality Meets Distributed Quantum Computing}, 
      author={Amirreza Akbari and Xavier Coiteux-Roy and Francesco d'Amore and François Le Gall and Henrik Lievonen and Darya Melnyk and Augusto Modanese and Shreyas Pai and Marc-Olivier Renou and Václav Rozhoň and Jukka Suomela},
      year={2024},
      eprint={2403.01903},
      archivePrefix={arXiv},
      primaryClass={cs.DC},
      url={https://arxiv.org/abs/2403.01903}, 
}

@misc{hochman,
 author = {Michael Hochman},
 title = {Irreducibility and periodicity in $\mathbb{Z}^{2}$ symbolic systems},
 year = {2024},
 howpublished = {Preprint, {arXiv}:2401.02273},
 url = {https://arxiv.org/abs/2401.02273},
 arXiv = {arXiv:2401.02273}
}

@misc{chandgotia2025undecidabilityblockgluingclasses,
 author = {Chandgotia, Nishant and Gangloff, Silv{\`e}re and de Menibus, Benjamin Hellouin and Oprocha, Piotr},
 title = {Undecidability of the block gluing classes of homshifts},
 year = {2025},
 howpublished = {Preprint, {arXiv}:2507.21342 [math.{DS}] (2025)},
 url = {https://arxiv.org/abs/2507.21342},
 arXiv = {arXiv:2507.21342}
}

@article {entropyminimality,
    AUTHOR = {Chandgotia, Nishant},
     TITLE = {Four-cycle free graphs, height functions, the pivot property
              and entropy minimality},
   JOURNAL = {Ergodic Theory Dynam. Systems},
  FJOURNAL = {Ergodic Theory and Dynamical Systems},
    VOLUME = {37},
      YEAR = {2017},
    NUMBER = {4},
     PAGES = {1102--1132},
      ISSN = {0143-3857,1469-4417},
       DOI = {10.1017/etds.2015.88},
       URL = {https://doi.org/10.1017/etds.2015.88},
}

@article {Remila,
    AUTHOR = {R\'emila, Eric},
     TITLE = {On the structure of some spaces of tilings},
   JOURNAL = {SIAM J. Discrete Math.},
  FJOURNAL = {SIAM Journal on Discrete Mathematics},
    VOLUME = {16},
      YEAR = {2002},
    NUMBER = {1},
     PAGES = {1--19},
      ISSN = {0895-4801,1095-7146},
       DOI = {10.1137/S0895480100373297},
       URL = {https://doi.org/10.1137/S0895480100373297},
}

@article{zbMATH04175886,
 author = {Thurston, William P.},
 title = {Conway's tiling groups},
 fjournal = {American Mathematical Monthly},
 journal = {Am. Math. Mon.},
 issn = {0002-9890},
 volume = {97},
 number = {8},
 pages = {757--773},
 year = {1990},
 doi = {10.2307/2324578},
 zbMATH = {4175886},
 Zbl = {0714.52007}
}

@article {ScottRibbon,
    AUTHOR = {Sheffield, Scott},
     TITLE = {Ribbon tilings and multidimensional height functions},
   JOURNAL = {Trans. Amer. Math. Soc.},
  FJOURNAL = {Transactions of the American Mathematical Society},
    VOLUME = {354},
      YEAR = {2002},
    NUMBER = {12},
     PAGES = {4789--4813},
      ISSN = {0002-9947,1088-6850},
       DOI = {10.1090/S0002-9947-02-02981-1},
       URL = {https://doi.org/10.1090/S0002-9947-02-02981-1},
}

@incollection{zbMATH00177172,
 author = {Smorodinsky, Meir},
 title = {Finitary isomorphism of {{\(m\)}}-dependent processes},
 booktitle = {Symbolic dynamics and its applications. Proceedings of an AMS conference in honor of Roy L. Adler, held at Yale University, New Haven, Connecticut, USA, July 28-August 2, 1991},
 isbn = {0-8218-5146-2},
 pages = {373--376},
 year = {1992},
 publisher = {Providence, RI: American Mathematical Society},
 zbMATH = {177172},
 Zbl = {0768.28009}
}

@article{zbMATH03656251,
 author = {Marcus, Brian},
 title = {A note on periodic points for ergodic toral automorphisms},
 fjournal = {Monatshefte f{\"u}r Mathematik},
 journal = {Monatsh. Math.},
 issn = {0026-9255},
 volume = {89},
 pages = {121--129},
 year = {1980},
 doi = {10.1007/BF01476590},
 url = {https://eudml.org/doc/177959},
 zbMATH = {3656251},
 Zbl = {0422.28013}
}

@article{zbMATH01566263,
 author = {Cohn, Henry and Kenyon, Richard and Propp, James},
 title = {A variational principle for domino tilings},
 fjournal = {Journal of the American Mathematical Society},
 journal = {J. Am. Math. Soc.},
 issn = {0894-0347},
 volume = {14},
 number = {2},
 pages = {297--346},
 year = {2001},
 doi = {10.1090/S0894-0347-00-00355-6},
 zbMATH = {1566263},
 Zbl = {1037.82016}
}

@article{zbMATH07206761,
 author = {Holroyd, Alexander E. and Hutchcroft, Tom and Levy, Avi},
 title = {Mallows permutations and finite dependence},
 fjournal = {The Annals of Probability},
 journal = {Ann. Probab.},
 issn = {0091-1798},
 volume = {48},
 number = {1},
 pages = {343--379},
 year = {2020},
 doi = {10.1214/19-AOP1363},
 zbMATH = {7206761},
 Zbl = {1456.60081}
}

@book {Bridson,
    AUTHOR = {Bridson, Martin R. and Haefliger, Andr\'e},
     TITLE = {Metric spaces of non-positive curvature},
    SERIES = {Grundlehren der mathematischen Wissenschaften [Fundamental
              Principles of Mathematical Sciences]},
    VOLUME = {319},
 PUBLISHER = {Springer-Verlag, Berlin},
      YEAR = {1999},
     PAGES = {xxii+643},
      ISBN = {3-540-64324-9},
       DOI = {10.1007/978-3-662-12494-9},
       URL = {https://doi.org/10.1007/978-3-662-12494-9},
}

@article {Dismantlable,
    AUTHOR = {Brightwell, Graham R. and Winkler, Peter},
     TITLE = {Gibbs measures and dismantlable graphs},
   JOURNAL = {J. Combin. Theory Ser. B},
  FJOURNAL = {Journal of Combinatorial Theory. Series B},
    VOLUME = {78},
      YEAR = {2000},
    NUMBER = {1},
     PAGES = {141--166},
      ISSN = {0095-8956,1096-0902},
       DOI = {10.1006/jctb.1999.1935},
       URL = {https://doi.org/10.1006/jctb.1999.1935},
}

@article {Bernshteyn,
    AUTHOR = {Bernshteyn, Anton},
     TITLE = {Probabilistic constructions in continuous combinatorics and a
              bridge to distributed algorithms},
   JOURNAL = {Adv. Math.},
  FJOURNAL = {Advances in Mathematics},
    VOLUME = {415},
      YEAR = {2023},
     PAGES = {Paper No. 108895, 33},
      ISSN = {0001-8708,1090-2082},
       DOI = {10.1016/j.aim.2023.108895},
       URL = {https://doi.org/10.1016/j.aim.2023.108895},
}

@article {GaoJacksonbreakthrough,
    AUTHOR = {Gao, Su and Jackson, Steve},
     TITLE = {Countable abelian group actions and hyperfinite equivalence
              relations},
   JOURNAL = {Invent. Math.},
  FJOURNAL = {Inventiones Mathematicae},
    VOLUME = {201},
      YEAR = {2015},
    NUMBER = {1},
     PAGES = {309--383},
      ISSN = {0020-9910,1432-1297},
       DOI = {10.1007/s00222-015-0603-y},
       URL = {https://doi.org/10.1007/s00222-015-0603-y},
}

@article {Clivio,
    AUTHOR = {Clivio, Alberto},
     TITLE = {Tilings of a torus with rectangular boxes},
   JOURNAL = {Discrete Math.},
  FJOURNAL = {Discrete Mathematics},
    VOLUME = {91},
      YEAR = {1991},
    NUMBER = {2},
     PAGES = {121--139},
      ISSN = {0012-365X,1872-681X},
       DOI = {10.1016/0012-365X(91)90104-A},
       URL = {https://doi.org/10.1016/0012-365X(91)90104-A},
}

@incollection {Frobenius,
    AUTHOR = {Labrousse, Denis and Ram\'irez Alfons\'in, Jorge L.},
     TITLE = {A tiling problem and the {F}robenius number},
 BOOKTITLE = {Additive number theory},
     PAGES = {203--220},
 PUBLISHER = {Springer, New York},
      YEAR = {2010},
      ISBN = {978-0-387-37029-3},
       DOI = {10.1007/978-0-387-68361-4\_15},
       URL = {https://doi.org/10.1007/978-0-387-68361-4_15},
}

@article {TSSM,
    AUTHOR = {Brice\~no, Raimundo},
     TITLE = {The topological strong spatial mixing property and new
              conditions for pressure approximation},
   JOURNAL = {Ergodic Theory Dynam. Systems},
  FJOURNAL = {Ergodic Theory and Dynamical Systems},
    VOLUME = {38},
      YEAR = {2018},
    NUMBER = {5},
     PAGES = {1658--1696},
      ISSN = {0143-3857,1469-4417},
       DOI = {10.1017/etds.2016.107},
       URL = {https://doi.org/10.1017/etds.2016.107},
}

@article {Undecidability,
    AUTHOR = {Berger, Robert},
     TITLE = {The undecidability of the domino problem},
   JOURNAL = {Mem. Amer. Math. Soc.},
  FJOURNAL = {Memoirs of the American Mathematical Society},
    VOLUME = {66},
      YEAR = {1966},
     PAGES = {72},
      ISSN = {0065-9266,1947-6221},
}

@article{S-gap,
   title={On the existence of open and bi-continuing codes},
   volume={363},
   ISSN={1088-6850},
   url={http://dx.doi.org/10.1090/S0002-9947-2010-05035-4},
   DOI={10.1090/s0002-9947-2010-05035-4},
   number={3},
   journal={Transactions of the American Mathematical Society},
   publisher={American Mathematical Society (AMS)},
   author={Jung, Uijin},
   year={2010},
   month=oct, pages={1399–1417} }

@article {RokhlinLemmaAmenableGroups,
    AUTHOR = {Ornstein, Donald S. and Weiss, Benjamin},
     TITLE = {Entropy and isomorphism theorems for actions of amenable
              groups},
   JOURNAL = {J. Analyse Math.},
  FJOURNAL = {Journal d'Analyse Math\'ematique},
    VOLUME = {48},
      YEAR = {1987},
     PAGES = {1--141},
      ISSN = {0021-7670,1565-8538},
       DOI = {10.1007/BF02790325},
       URL = {https://doi.org/10.1007/BF02790325},
}

@article {MarkovRandomFieldMarkovChains,
    AUTHOR = {Chandgotia, Nishant and Han, Guangyue and Marcus, Brian and
              Meyerovitch, Tom and Pavlov, Ronnie},
     TITLE = {One-dimensional {M}arkov random fields, {M}arkov chains and
              topological {M}arkov fields},
   JOURNAL = {Proc. Amer. Math. Soc.},
  FJOURNAL = {Proceedings of the American Mathematical Society},
    VOLUME = {142},
      YEAR = {2014},
    NUMBER = {1},
     PAGES = {227--242},
      ISSN = {0002-9939,1088-6826},
       DOI = {10.1090/S0002-9939-2013-11741-7},
       URL = {https://doi.org/10.1090/S0002-9939-2013-11741-7},
}

@book {Kalikow,
    AUTHOR = {Kalikow, Steven and McCutcheon, Randall},
     TITLE = {An outline of ergodic theory},
    SERIES = {Cambridge Studies in Advanced Mathematics},
    VOLUME = {122},
 PUBLISHER = {Cambridge University Press, Cambridge},
      YEAR = {2010},
     PAGES = {viii+174},
      ISBN = {978-0-521-19440-2},
       DOI = {10.1017/CBO9780511801600},
       URL = {https://doi.org/10.1017/CBO9780511801600},
}
\section{Appendix: Tileability of sufficiently large boxes by a given set of boxes (by Jan Greb\'ik, Filip Ku\v{c}er\'ak and Aditya Thorat)}
Here, we will give a necessary and sufficient condition for the tileability of integral boxes (with flat boundary) with sufficiently large sidelengths by a given collection $ \R$ of integral boxes in $ \mathbb R^d$. This condition is sufficient for the existence of continuous tiling of $ F(2^{\Z^d})$ by the tiling space $ X(\R)$.  In \cite{Clivio} and \cite{Frobenius}, similar conditions for tiling arbitrarily large tori (tilings of boxes with periodic boundary configuration instead of flat boundary) with a given set of boxes are provided. 

Let $d\ge 1$. Given a vector $\overline{n}=(n_1,\dots,n_d)$ of positive integers, we write
$$R_{\overline{n}}=\{0,\dots,n_1-1\}\times \dots\times \{0,\dots,n_d-1\}$$
for the $d$-dimensional box determined by $\overline{n}$. 
Let $\mathcal{R}=\{R_{\overline{n}^1},\dots,R_{\overline{n}^k}\}$ be a collection of $d$-dimensional boxes, where $\overline{n}^i=(n^i_1,\dots,n^i_d)$ for every $1\le i\le k$.

Let $\CR^d$ be the set of all $d$-dimensional boxes.  Let $\partial : \CR^d \to \CR^{d-1}$ be a map that sends $d$-dimensional box $R_{\bar{n}}$ to the $(d-1)$-dimensional box determined by $(n_1, ..., n_{d-1})$. We will use the following simple observation.

\newcommand\lcm{\text{lcm}}

\begin{lemma}\label{lem:tilinglift}
    Let $d \in \BN$, let $\CR$ be a collection of boxes and let $R_{\ell}$ be a box associated to $\ell = (\ell_1, ..., \ell_d)$. If $\partial R_{\ell}$ can be tiled using $\partial\CR$ and $\ell_d$ is a multiple of $\lcm(n_d(R) : R \in \CR)$, then $R_{\ell}$ can be tiled using $\CR$.
\end{lemma}

\begin{theorem}\label{thm:eventuallyuniversaltiles}
    The following are equivalent:
    \begin{enumerate}
        \item There is $k_0 \in \mathbb{N}$ such that $\mathcal{R}$ tiles $R_{\ell}$ for every $\ell=(\ell_1,\dots,\ell_d)$ such that $\ell_1,\dots,\ell_d\ge k_0$.
        \item For every partition $\mathcal{R}=A_1\sqcup\dots\sqcup A_d$, (where some of the $ A_i$'s are possibly empty) there is $1\le j\le d$ such that
        $$\gcd\left(n_j(R):R \in A_j\right)=1.$$
    \end{enumerate}
\end{theorem}

\begin{proof}[Proof of (2) implies (1)]
    The proof proceeds using induction on $d$. Assume $d = 1$. Then using the assumption with the trivial partition, we know that $\gcd( n_1(R) : R \in \CR ) = 1$ and thus there exists a positive integer $k_0$ such that for every $\ell_1 \ge k_0$ there exists a collection of positive integers $\ell_1^R$ for $R \in \CR$ with the property that $\ell_1 = \sum_i \ell_1^Rn_1(R)$. Using such decomposition of $\ell_1$, we can clearly tile any one-dimensional box $R_{\ell_1}$ with $\ell_1 \ge k_0$ using $\ell^R_1$ many copies of each $R \in \CR$.

    Let $d > 1$ and let $p$ be a prime. Let $A^p$ be the set of boxes $R
    \in \CR$ with the property that $n_d(R)$ is divisible by $p$ and let $B^p$
    be its complement in $\CR$. Let $B_1^p \sqcup ... \sqcup B_{d-1}^p$ be a partition of $B^p$ so that $B_1^p \sqcup ... \sqcup B_{d-1}^p \sqcup A^p$ is a partition of $\CR$. As $\gcd(n_d(R) : R \in A^p) \neq 1$ (it is
    zero if $A^p$ is empty and {a multiple of} $p$ otherwise), we conclude that there exists $1 \le j \le d-1$ such that $\gcd(n_j(R) : R \in B_{j}^{p}) = 1$. { In particular, we must have that $A^p\not=\mathcal{R}$, or equivalently $B^{p}\not=\emptyset$, for any prime $p$.} As the partition of $B^p$ is arbitrary, we conclude that $\partial B^p$ satisfies (2) and so by the induction hypothesis, there exists a positive integer $k^p_0$ with the property that any
    $(d-1)$-dimensional box $R_{\ell}$ with $\ell_1, ..., \ell_{d-1} \ge
    k^p_0$ can be tiled with $\partial B^p$. 

    Let $m^p = \lcm(n_d(R) : R \in B^p)$. Assume that $\gcd(m^p : \text{$p$ prime})$ is not one, i.e. it is larger than one: the gcd cannot be zero as $B^p\not=\emptyset$ for every prime $p$. It thus holds that there exists a prime
    $q$ such that $q$ divides all $m^p$ and specifically divides $m^q$. By properties of least common multiple,
    there exists a box $R \in B^q$ with the property that $q$ divides
    $n_d(R)$; this is in direct contradiction with the definition of the set $B^q$.
    
    As $\gcd(m^p : \text{$p$ prime})$ is one, there exists a finite set of primes $P$ such that $\gcd(m^p : \text{$p \in P$})=1$. Thus, there exists a natural
    number $m_0$ with the property that for every $\ell_d \ge m_0$, there exists for every $p \in P$ a non-negative integer $\ell^p_d$  such that $\sum_{p \in P} \ell^p_d m^p = \ell_d$.  Let $k_0$ be a positive integer such that it is larger than
    $k_0^p$ for every $p \in P$ and $m_0$.

    Let $\ell \in \BN^d$ be a collection of sidelengths with $\ell_1, ..., \ell_d \ge k_0$ and let $R_\ell$ be the box determined by these sidelengths. Let $R_p$ be for every $p$ in $P$ the $d$-dimensional box determined by $(\ell_1, ..., \ell_{d-1}, m^p)$. 
    As $\ell_d \ge k_0 \ge m_0$, there exists for every $p \in P$ a non-negative integer $\ell^p_d$ with the property that $\sum_{p\in P} \ell^p_dm^p = \ell_d$. This implies that 
    $R_\ell$ can be tiled using the set $\{ R_p : p \in P \}$ by stacking $\ell_d^p$ many copies of each $R_p$ in the $d$-th dimension.
    It is enough to show that we can for each $p \in P$ tile $R_p$ by the set $\CR$. As $\ell_1, ..., \ell_{d-1} \ge k_0 \ge k^p_0$, we use the induction hypothesis to conclude that $\partial B^p$ tiles $\partial R_p$. By \Cref{lem:tilinglift}, it is immediate that $B^p$ and consequently $\CR$ tiles $R_p$ as $n_d(R_p) = m^{p} = \lcm( n_d(R) : R \in B^p )$.
\end{proof}

We say that a $d$-dimensional box $R$ is a {\em bar} if there exists $i \in [d]$ with the property that for every $j \in [d]$ it holds that $n_j(R) \neq 1$ if and only if $j = i$. We say that a collection of bars $\CB$ is orthogonal, if for every {distinct} $B_1, B_2 \in \CB$ and every $i \in [d]$, it holds that $n_i(B_1) \neq 1$ implies $n_i(B_2) = 1$.

\begin{proposition}\label{prop:onebarisenough}

    Let $\CB$ be an orthogonal collection of $d$-dimensional bars and let $R$ be a $d$-dimensional box. 
    If $\CB$ tiles $R$, then there exists $B \in \CB$ such that $B$ tiles $R$.
\end{proposition}
\begin{proof}[Proof of (1) implies (2) in \Cref{thm:eventuallyuniversaltiles} using Proposition~\ref{prop:onebarisenough}]
    Assume that there exists a partition $\CR = A_1 \sqcup ... \sqcup A_d$ such that 
    \[
        \gcd( n_i(R) : R \in A_i ) = d_i \neq 1
    \]
    for all $i \in [d]$. Let $I \subseteq [d]$ be the set of indices $i \in [d]$ for which $d_i \neq 0$ holds, i.e, the set of indices $i \in [d]$ such that the set $A_i$ is non-empty. For $i \in I$, let $p_i$ be a prime which divides $d_i$. Let $B_i$ be a bar with $n_i(B_i) = p_i$; clearly $B_i$ tiles the set $A_i$ and consequently $\CB = \{ B_i : i \in I \}$ tiles each element of $\CR = \bigsqcup_{i \in I} A_i$.
    
    Consider $R = R_{\ell}$, where $\ell=(\ell_1,\dots, \ell_d)$ is such that $\ell_1,\dots\ell_d\ge k_0$ and $\ell_i$ is not divisible by $p_i$ for every $i\in I$.
    By (1), we have that $\CR$ tiles $R$. As $\CB$ tiles $\CR$ it also tiles $R$. 
    As $\CB$ is an orthogonal set of bars, \Cref{prop:onebarisenough} yields that there is a bar $B \in \CB$ such that $B$ tiles $R$. Let $i \in I$ be an index such that $n_i(B) = p_i$, then $p_i$ divides $n_i(R)=\ell_i$ which contradicts the choice of $\ell_i$.
\end{proof}

Fix a dimension $d \in \BN$. Let $S = \BC[x_1, ..., x_d]$ be a ring of complex polynomials over $d$ variables, where we associate to a point $a \in \BN_0^d$ a monomial $f_a = x^a =  \prod_{i=1}^dx_i^{a_i}$ and to finite set of points $A \subset \BN_0^d$ we associate a polynomial $f_A = \sum_{a \in A}x^a$. Let $s \in \BN_0^d$ be a direction of a shift, we define the shift injection $\sigma_s : \BN_0^d \to \BN_0^d$ by $\sigma_s(a) = a + s$. Notice that $f_{\sigma_s A} = x^s f_A$.

Let $\CB = \{ B_1, ..., B_k \}$ be a collection of subsets $B_i \subseteq \BN_0^d$, we say that $\CB$ tiles $A$, if there exists a finite $n \in \BN$, choice of tiles $\beta : [n] \to \CB$ and choice of shifts $\alpha : [n] \to \BN_0^d$ such that $A = \bigsqcup_{i \in [n]}\sigma_{\alpha(i)} \beta(i)$. By previous discussion, it holds 
that if $\CB$ tiles $A$ with $\beta$ and $\alpha$, it holds that $f_A = \sum_{i \in [n]}x^{\alpha(i)}f_{\beta(i)}$ and consequently, $f_A$ is in the ideal $I_\CB$ generated by $f_{B_1}, ..., f_{B_k}$.

We now relate these concepts to our earlier definition of a box.
Namely, to a given box $R=R_\ell$, where $\ell=(\ell_1,\dots,\ell_d)\in \mathbb{N}^d$, we assign a polynomial
    \begin{equation}\label{eq:Poynomial}
            f_R =\sum_{a \in R}{x^a}
          =(1 + x_1 + ... + x_1^{\ell_1 - 1})(1 + x_2 + ... + x_2^{\ell_2 - 1})...(1 + x_d + ... + x_d^{\ell_d - 1}).
    \end{equation}
That is, $f_R$ coincide with the polynomial associated with the set $R_\ell=\{0,\dots,\ell_1-1\}\times \dots\{0,\dots,\ell_d-1\}$.
We are now ready to prove \Cref{prop:onebarisenough}.

\begin{proof}[Proof of \Cref{prop:onebarisenough}]
    Let $\CB = \{ B_1, ...,B_k \}$ be an orthogonal set of $d$-dimensional bars. 
    Without loss of generality, assume that $n_i(B_i) = b_i \neq 1$ for all $i \in [k]$. 
    Let $R=R_\ell$ be a $d$-dimensional box, where $\ell=(\ell_1,\dots,\ell_d)\in \mathbb{N}^d$. 
    Clearly, $R$ can be tiled by $\CB$ if and only if $\partial^{d-k}R$ can be tiled by $\partial^{d-k}\CB$. 
    As such, we can assume that $k = d$.

    Let $f_R$ be the polynomial associated to $R$ as in \eqref{eq:Poynomial}.
    Similarly, we have $f_{B_i} = 1 + x_i + ... + x_i^{b_i-1}$ for $B_i\in \CB$.
    Let $\omega=(\omega_m)_{m=1}^d \in \BC^d$ be such that $\omega_m=e^{2 \pi i / b_m}$.
    For every $j \in [d]$, it holds that
    \begin{align*}
        f_{B_j}(\omega) = 1 + \omega_j + ... + \omega_j^{b_j - 1} = \frac{\omega_j^{b_j} - 1}{\omega_j - 1 } = 0,
    \end{align*}
    and as such, $\omega$ is a common root of all polynomials $f_B$ for $B \in \CB$. As the polynomial $f_R$ associated to the box $R$ is in the ideal $I_\CB$ generated by the polynomials $f_B$ for $B \in \CB$ by assumption that $\CB$ tiles $R$, it holds that the polynomial $f_R$ vanishes on $\omega$ as well. Evaluating $f_R$ at $\omega$ yields
    
    \begin{align*}
        f_R(\omega)
          &=(1 + \omega_1 + ... + \omega_1^{\ell_1 - 1})(1 + \omega_2 + ... + \omega_2^{\ell_2 - 1})...(1 + \omega_d + ... + \omega_d^{\ell_d - 1})\\
          &=
          \bigg(\frac{\omega_1^{\ell_1} - 1}{\omega_1 - 1}\bigg)
          \bigg(\frac{\omega_2^{\ell_2} - 1}{\omega_2 - 1}\bigg)
          ...
          \bigg(\frac{\omega_d^{\ell_d} - 1}{\omega_d - 1}\bigg)
          = 0.
    \end{align*}
    As $\BC$ is a field, it holds that for some $i \in [d]$ we have $(\omega_i^{\ell_i} - 1)/(\omega_i - 1) = 0$ and consequently $\omega_i^{\ell_i} = 1$ as 
    $\omega_i$ has order $b_i$ as an element of $\BC^{\times}$, it holds that $b_i$ divides $\ell_i$. 
    We can now tile the box $R$ using the bar $B_i$. Without loss of generality, assume that $i = 1$. Using $\ell_1 / b_1 \in \BN$ many bars $B_1$, we can tile the box $S$ determined by $(\ell_1, 1, ..., 1) \in \BN^d$ and consequently any box determined by $(\ell_1, n_2, ..., n_d)$ for arbitrary $n_2, ..., n_d \in \BN$. As $R$ is a box of this form, we are done.
\end{proof}
\begin{corollary}\label{corollary: positive condition for CONTINUOUS}
    Let $ \R$ be a collection of boxes in $ \mathbb{R}^d$ with integer sidelengths such that for every partition $\mathcal{R}=A_1\sqcup\dots\sqcup A_d$, there is $1\le j\le d$ such that
        $$\gcd\left(n_j(R):R \in A_j\right)=1.$$
    Then there exists a continuous equivariant map from $ F(2^{\Z^d})$ to the tiling space $ X(\R)$.
\end{corollary}

\begin{proof}
    Let $k_0$ be as in \Cref{thm:eventuallyuniversaltiles}.  Let $ \R'= \{ R'_1, R'_2, \cdots, R'_N\}$ be the collection of all boxes in $\mathbb{R}^d$ whose sidelengths lie in $\{k_0, k_0+1 \}$.  Then \Cref{thm:eventuallyuniversaltiles} guarantees that each of the boxes $R'_i$ can be tiled by the tiles in $ \R$.  We fix a tiling $\mathcal{T}_i'$ of $R'_i$ by boxes in $\R$. Then there is a continuous equivariant map $\Psi': X(\R') \rightarrow X(\R) $ obtained by further tiling each occurrence of  the tile $ R'_i$ in an element $x \in X(\R')$ by the tiling $\mathcal{T}'_i$.  But by \Cref{theorem: perfect tiligs}, there exists a continuous equivariant map $ \Phi':F(2^{\Z^d}) \rightarrow X(\mathcal{R}')$.  Thus, $ \Phi= \Psi'\circ \Phi'$ gives us the required map.
\end{proof}

%\comm{Nishant}{Refer to theorem from introduction}
\end{document}